\documentclass[a4paper, 11pt, headings = small, abstract]{scrartcl}

\usepackage[english]{babel}
\usepackage[utf8]{inputenc}
\usepackage{geometry}

\usepackage[semibold]{libertine}
\usepackage[upint,amsthm]{libertinust1math}

\usepackage{amsmath}
\usepackage{amsfonts}
\usepackage{amssymb}
\usepackage{amsthm}
\usepackage{mathtools}
\usepackage{paralist}
\usepackage{comment}

\usepackage[authoryear]{natbib}

\usepackage{bibunits}

\defaultbibliography{biblio}
\defaultbibliographystyle{apalike}

\usepackage[dvipsnames]{xcolor}
\usepackage[colorlinks=true,linkcolor=blue,citecolor=blue,pdfborder={0 0 0}]{hyperref}
\usepackage[capitalize]{cleveref}
\usepackage{orcidlink}
\usepackage{bm}

\usepackage{algorithm, algpseudocode, algorithmicx,eqparbox,array}

\usepackage{graphicx}
\usepackage{enumitem}

\usepackage{paralist}

\usepackage{tikz}
\usetikzlibrary{positioning}
\usetikzlibrary{backgrounds}

\newtheorem{theorem}{Theorem}[section]
\newtheorem{proposition}[theorem]{Proposition}
\newtheorem{lemma}[theorem]{Lemma}
\newtheorem{corollary}[theorem]{Corollary}

\theoremstyle{definition}

\newtheorem{condition}[theorem]{Condition}
\newtheorem{model}[theorem]{Model}

\theoremstyle{remark}
\newtheorem{remark}[theorem]{Remark}
\newtheorem{example}[theorem]{Example}

\allowdisplaybreaks[4]

\numberwithin{equation}{section}

\newcommand{\R}{\mathbb{R}}

\newcommand{\Z}{\mathbb{Z}}
\newcommand{\N}{\mathbb{N}}
\newcommand{\eps}{\varepsilon}

\newcommand{\Gb}{\mathbb {G}}

\newcommand{\Ac}{\mathcal{A}}

\newcommand{\Lc}{\mathcal{L}}
\newcommand{\Mc}{\mathcal{M}}
\newcommand{\Nc}{\mathcal{N}}

\newcommand{\Xc}{\mathcal{X}}
\newcommand{\mc}[1]{\mathcal{#1}}
\newcommand{\mb}[1]{\mathbb{#1}}

\newcommand{\weak}{\rightsquigarrow}
\newcommand{\wconv}{\rightsquigarrow}

\newcommand{\pto}{\xrightarrow{\mathbb{P}}}

\newcommand{\Prob}{\mathbb{P}}    %probability
\newcommand{\Exp}{\operatorname{E}}
\newcommand{\Var}{\operatorname{Var}}
\newcommand{\Cov}{\operatorname{Cov}}

\newcommand{\diag}{\operatorname{diag}}
\newcommand{\argmin}{\operatornamewithlimits{\arg\min}}
\newcommand{\argmax}{\operatornamewithlimits{\arg\max}}
\newcommand{\diff}{\,\mathrm{d}}
\newcommand{\flo}[1]{\lfloor #1 \rfloor}

\DeclareMathOperator{\circf}{circ}
\DeclareMathOperator{\circmax}{circ{\text-}max}
\DeclareMathOperator{\slidmax}{slid{\text-}max}

\newcommand{\dbl}{{\ensuremath {\operatorname{d} }}}
\newcommand{\sbl}{{\ensuremath {\operatorname{s} }}}
\newcommand{\mbl}{{\ensuremath {\operatorname{m} }}}
\newcommand{\cbl}{{\ensuremath {\operatorname{c} }}}
\newcommand{\cblk}{{\ensuremath {\operatorname{c}}}}
\newcommand{\sblcbl}{{\sbl\text{-}\cbl}}

\newcommand{\rlhat}{{\hat{\mathrm{RL}}}}
\newcommand{\rl}{{\mathrm{RL}}}
\newcommand{\err}{{\mathrm{err}}}
\newcommand{\CI}{{\mathrm{CI}}}

\newcommand{\refstar}{\ref}

\begin{document}
\begin{bibunit}

\title{\fontsize{16}{19} Bootstrapping Estimators based on the Block Maxima Method}

\author{
Axel B\"ucher\thanks{Ruhr-Universität Bochum, Fakultät für Mathematik. Email: \href{mailto:axel.buecher@rub.de}{axel.buecher@rub.de}} \orcidlink{0000-0002-1947-1617}
\and
Torben Staud\thanks{Ruhr-Universität Bochum, Fakultät für Mathematik. Email: \href{mailto:torben.staud@rub.de}{torben.staud@rub.de}}\hspace{0.05cm} \orcidlink{0009-0004-4526-8585}
}

\date{\today}

\maketitle

\begin{abstract} 
The block maxima method is a standard approach for analyzing the extremal behavior of a potentially multivariate time series.
It has recently been found that the classical approach based on disjoint block maxima may be universally improved by considering sliding block maxima instead.
However, the asymptotic variance formula for estimators based on sliding block maxima involves an integral over the covariance of a certain family of multivariate extreme value distributions, which makes its estimation, and inference in general, an intricate problem. As an alternative, one may rely on bootstrap approximations: we show that naive block-bootstrap approaches from time series analysis are inconsistent even in i.i.d.\ situations, and provide a consistent alternative based on resampling circular block maxima. 
As a by-product, we show consistency of the classical resampling bootstrap for disjoint block maxima, and that estimators based on circular block maxima have the same asymptotic variance as their sliding block maxima counterparts. 
The finite sample properties are illustrated by Monte Carlo experiments, and the methods are demonstrated by a case study of precipitation extremes.

\end{abstract}

\color{black}
\noindent\textit{Keywords.} 
Bootstrap Consistency;
Disjoint and Sliding Block Maxima;
Extreme Value Statistics; 
Pseudo Maximum Likelihood Estimation;
Time Series Analysis.

\smallskip

\noindent\textit{MSC subject classifications.} 
Primary
62F40, %Bootstrap, jackknife and other resampling methods
62G32; %Statistics of extreme values; tail inference
% 62G20, %Asymptotic properties of nonparametric inference
% %62G09 %Nonparametric statistical resampling methods
% %60F17 %Functional limit theorems; invariance principles
% 62H05; %Characterization and structure theory for multivariate probability distributions; copulas
% %62G30 %Order statistics; empirical distribution functions
Secondary
% 62G09.  %Nonparametric statistical resampling methods
62E20. %Asymptotic distribution theory in statistics

\tableofcontents

\section{Introduction}

A common approach to statistically assess extreme events is the block maxima method.
For block size $r \in \N$, for instance $r=365$ corresponding to a year of daily observations, its starting point is the vector $\bm M_r=\max_{t=1}^r \bm X_t = (\max_{t=1}^r X_{t1}, \dots, \max_{t=1}^r X_{td})^\top$ of  componentwise maxima over $r$ successive observations from a time series $(\bm X_t)_{t \in \Z}$  in $\R^d$; with $\bm X_t = (X_{t1}, \dots, X_{td})^\top$. In view of a version of the classical extremal types theorem for multivariate time series, the distribution of $\bm M_r$ may be approximated by a multivariate extreme value distribution, for sufficiently large block size $r$. 
Both the distribution of $\bm M_r$ and parameters related to its potential weak limit distribution are common target parameters in extreme value statistics, with applications in finance, insurance or environmental statistics among others \citep{BeiGoeSegTeu04, Kat02}. For instance, in extreme event attribution studies (e.g., for cumulative daily precipitation amounts), one is typically interested in specific large quantiles or exeedance probabilities of $M_{365}$ that are defined by a recent observed extreme event. These parameters are directly linked to return levels and return periods of the underlying time series, and allow for assessing the effect of global warming on the observed event, for instance in terms of probability ratios \citep{Van17, philip2020protocol, Tra23}.

Formally, the statistical analysis starts from $\bm X_1, \dots, \bm X_n$, an observed stretch from an (approximately) stationary time series $(\bm X_t)_{t \in \Z}$. The classical block maxima method for estimating quantities associated with the law of $\bm M_r$ (or its limiting law for $r\to\infty$) consists of dividing the sampling period $\{1, \dots, n \}$ into $m$ disjoint blocks of length $r$ (for simplicity, we assume that $n=mr$), and of using the (disjoint) block maxima sample $(\bm M_{1,r}^{[\dbl]}, \dots, \bm M_{m,r}^{[\dbl]})$ as a starting point for statistical methods; here $\bm M^{[\dbl]}_{i,r}=\max_{t=(i-1)r+1, \dots, ir}\bm X_t$ denotes the $i$th disjoint block maximum. This classical approach can often be improved by instead considering the sliding block maxima sample $(\bm M_{1,r}^{[\sbl]}, \dots, \bm M_{n-r+1,r}^{[\sbl]})$ as a starting point; here $\bm M^{[\sbl]}_{i,r}=\max_{t=i, \dots, i+r-1}\bm X_t$ denotes the block maximum of size $r$ starting at ``day'' $i$. Note that this sample also provides an (approximately) stationary, but strongly auto-correlated sample from the law of $\bm M_r$. It was shown in \cite{BucSeg18-sl, ZouVolBuc21} among others that estimators that explicitly or implicitly involve empirical means of the sliding block maxima sample typically have the same expectation but a slightly smaller variance than the respective counterparts based on the disjoint block maxima sample.

It seems natural that a lower estimation variance offers the possibility of constructing smaller confidence intervals for a given confidence level, thereby for instance allowing for more precise statements regarding the aforementioned extreme event attribution. The main goal of this work is to find universal practical solutions for this heuristic based on suitable bootstrap approximations. A challenge consists of the fact that the large sample asymptotics of the sliding block maxima method depend on a blocking of blocks approach involving some intermediate blocking parameter converging to infinity at a well-balanced rate. As a consequence, standard bootstrap approaches for time series like the block bootstrap \citep{Lah03} would depend on such an intermediate ``blocking of blocks'' parameter as well (which must be avoided due to the typically small effective sample size in extreme value analysis), while naive approaches based on (blockwise) resampling of sliding block maxima will be inconsistent (see \cref{fig:incHists}, Remark~\ref{rem:slidboot-inconsistent} below, and Section~\refstar{sec:inconsistency-sliding} in the supplement).
As illustrated on the right-hand side of \cref{fig:incHists}, a novel approach based on what we call the circular (sliding) block maxima sample allows for consistent distributional approximations without relying on such an intermediate parameter sequence. At the same time, the method is computationally attractive because it avoids recalculation of any block maxima for the bootstrap sample.

\begin{figure}
    \centering
    \includegraphics[width=0.91\textwidth]{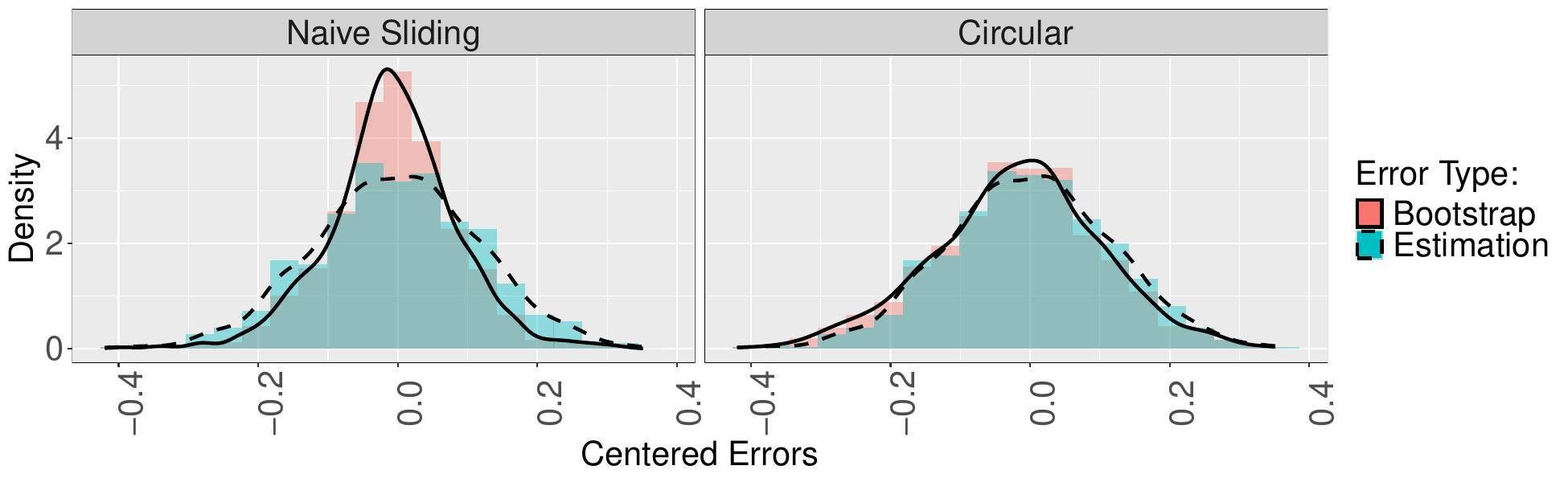}
    \caption{
    Empirical distribution of the estimation error and the bootstrap estimation error for 100-year return level estimation in Model~\ref{mod:armax-gpd} with $\gamma = -0.2, \beta = 0.5, m = 80, r = 365$. Left: naive sliding approach. Right: circular bootstrap approach. See Section~\ref{sec:bootgen} and Figure~\ref{fig:RlHists} for further details.
    }\vspace{-.2cm}
    \label{fig:incHists}
\end{figure}

In general, various variants of the bootstrap have been routinely applied in extreme value statistics (see, for instance, \citealp{EasTaw08, HusDav14}). However, to the best of our knowledge, respective statistical theory is only available for the peak-over-threshold method \citep{Pen08, Dre15, Dav18, KulSou20, Jen21, DehZho24}.
Concerning the block maxima method, within the regime where the block size is treated as a parameter sequence $r=r_n$ converging to infinity, the only reference we are aware of is \cite{DehZho24}, who derive an asymptotic expansion for the tail quantile process of bootstrapped disjoint block maxima. However, their results require an underlying IID sequence, they only concern the disjoint block maxima method and they do not eventually imply desirable classical consistency statements. While our main focus is on the sliding block maxima method, we establish consistency results for the disjoint block maxima method as a by-product, for serially dependent data.

Finally, we also derive limit results for empirical means and estimators based on the circular (sliding) block maxima sample, and prove that they show the same favorable asymptotic behavior as the sliding block maxima counterparts. 
As we will discuss in the conclusion, the result may be of independent interest when dealing with non-stationary situations involving, for example, trends, which is typical for climate extremes.

All methods in this paper are intrinsically linked to the block maxima method. In typical use cases of that method, target parameters directly depend on the block maxima distribution (and hence on~$r$); it can thus be considered a very natural approach in this situation. 
When target parameters depend on the marginal distribution of $\bm X_t$ (e.g., high quantiles), the peak-over-threshold method is typically considered the more natural and useful approach. As such, both methods have their raison d'être and are most suitable for different applications; see also the discussion in  Section~3.3 in \cite{buecher2020horse}.

The remaining parts of this paper are organized as follows: basic model assumptions and known results on the disjoint and sliding block maxima method are collected in Section~\ref{sec:mathpre}. The sample of circular block maxima is introduced in Section~\ref{sec:circmax} and supplemented by a central limit theorem. Appropriate bootstrap schemes are proposed in Section~\ref{sec:bootgen}, where we also provide respective consistency statements. The previous high-level results are applied to a specific estimation problem for univariate heavy tailed time series in Section~\ref{sec:frechet-small}. Finite-sample results on a large-scale simulation study and two case studies are presented in Sections~\ref{sec:sim} and \ref{sec:case}, respectively. Section~\ref{sec:conclusion} concludes. Some technical integrability and bias conditions are given in Section~\ref{sec:conditions}. 
Finally, all proofs, additional theoretical results and additional simulation results are provided in a supplementary material. The code used for the simulation study is available on \href{https://github.com/torbenstaud/Bootstrapping-block-maxima}{Github}, see \cite{Sta24}.
Throughout, weak convergence of sequences of random vectors/distributions is denoted by $\weak$.

\section{Mathematical preliminaries}
\label{sec:mathpre}

Throughout this section, we present basic model assumptions rooted in extreme value theory for stationary time series. We then summarize the most important aspects about the (disjoint and sliding) block maxima method, including a central limit theorem (CLT) for sample averages based on affinely standardized block maxima. We also present high level conditions under which the CLTs can be used as a basis for statistical inference and illustrate them with some examples.

\subsection{Basic model assumptions}
\label{sec:basicmodel}

An extension of the classical extremal types theorem to strictly stationary time series \citep{Lea83} implies that, under suitable broad conditions, affinely standardized maxima extracted from a stationary time series converge to the generalized extreme value distribution (GEV). This was generalized to the multivariate case in \cite{Hsi89}, where the marginals are necessarily GEV-distributed. We make this an assumption, and additionally require the scaling sequences to exhibit some common regularity inspired by the max-domain of attraction condition in the IID case \citep{deextreme}. 

\begin{condition}[Multivariate max-domain of attraction] \label{cond:mda}
Let $(\bm X_t)_{t\in\Z}$ denote a strictly stationary time series in $\R^d$ with continuous margins. There exist sequences $(\bm a_r)_r=(a_{r}^{(1)}, \dots, a_{r}^{(d)})_{r} \subset (0,\infty)^d, (\bm b_r)_r = ( b_{r}^{(1)}, \dots, b_{r}^{(d)})_r \subset \R^d$  and $\bm \gamma = (\gamma^{(1)}, \dots, \gamma^{(d)}) \in\R^d$, such that, for any $s>0$ and $j\in\{1, \dots, d\}$,
\begin{align}\label{eq:rvscale}
\lim\limits_{r \to \infty}\frac{a_{\lfloor rs \rfloor}^{(j)}}{a_{r}^{(j)}}  = s^{\gamma^{(j)}},
\qquad 
\lim\limits_{r \to \infty} \frac{b_{\lfloor rs \rfloor }^{(j)} -b_{r}^{(j)}}{a_{r}^{(j)}} = \frac{s^{\gamma^{(j)}} -1}{\gamma^{(j)}}, 
\end{align} 
where the second limit  is interpreted as $\log(s)$ if $\gamma^{(j)}=0$. Moreover,   	
for $r \to \infty$, 			
\begin{align}\label{eq:firstorder} 
\bm Z_r 
= 
\frac{\bm M_r - \bm b_r}{\bm a_r}
\weak 
\bm Z \sim G,
\end{align}
where $G$ denotes the cumulative distribution function (CDF) of a $d$-variate extreme-value distribution with marginal CDFs $G_{\gamma^{(1)}}, \dots, G_{\gamma^{(d)}}$ (with $G_\gamma(x)=\exp(-(1+\gamma x)^{-1/\gamma}) \bm 1(1+\gamma x >0)$ the CDF of the GEV($\gamma$)-distribution) and where division by the vector $\bm a_r$ is understood componentwise, that is, 
$
\bm Z_r = (Z_{r}^{(1)}, \dots, Z_{r}^{(d)})$ with $Z_{r}^{(j)} = \{ \max(X_1^{(j)}, \ldots, X_r^{(j)}) - b_{r}^{(j)} \} / {a_{r}^{(j)}}$ for $j \in \{1, \dots, d\}$.
\end{condition}

Note that (\ref{eq:rvscale}) and (\ref{eq:firstorder}) may for instance be deduced from Leadbetter's $D(u_n)$ condition, a domain-of-attraction condition on the associated IID sequence with stationary distribution equal to that of $\bm X_0$ and a weak requirement on the convergence of the CDF of $\bm Z_r,$ see Theorem 10.22 in \cite{BeiGoeSegTeu04}.

Throughout this paper, we assume to  observe $\bm X_1,\ldots, \bm X_n$, an excerpt from a $d$-dimensional time series $(\bm X_t)_{t\in\Z}$ satisfying Condition \ref{cond:mda}. Block maxima will be formed with respect to a block size parameter $r=r_n$ converging to infinity such that $r = o(n)$. 
For establishing the subsequent limit results, the serial dependence of $(\bm X_t)_t$ will additionally be controlled via the Rosenblatt mixing coefficient \citep{Bra05}. For two sigma-fields $\mathcal{F}_1, \mathcal{F}_2$ on a probability space $(\Omega, \mathcal{A}, \Prob)$, let
$
\alpha(\mathcal F_1, \mathcal F_2) 
= \sup_{A \in \mathcal{F}_1, B \in \mathcal{F}_2} 
| \Prob(A \cap B) - \Prob(A) \Prob(B) |.
$
For positive integer $p$, let $\alpha(p) := \alpha \big(\sigma((\bm X_t)_{t \leq 0}), \sigma((\bm X_t)_{t \geq p})\big)$, with $\sigma(\cdot)$ denoting the sigma-field generated by its argument. 

\begin{condition}[Block size and mixing]\label{cond:ser_dep}
The block size sequence $(r_n)_n$ satisfies, as $n\to\infty$:
\begin{compactenum}
\item[(a)] $r_n \to \infty$ and $r_n = o(n).$
\item[(b)] There exists a sequence $(\ell_n)_n \subset \N$ such that $\ell_n \rightarrow \infty $, $\ell_n = o(r_n)$, $\frac{r_n}{\ell_n}  \alpha(\ell_n) = o(1)$, and $\frac{n}{r_n} \alpha(\ell_n)=o(1)$.
\item[(c)] $\bigl(\frac{n}{r_n}\bigr)^{1+\omega} \alpha(r_n) = o(1) $ for some $\omega > 0$.
\end{compactenum}
\end{condition}

\subsection{Empirical means of rescaled disjoint or sliding block maxima}
\label{sec:empmean}

A central theoretical ingredient for establishing weak limit results on estimators based on the block maxima method is the weak convergence of either the tail quantile process or of empirical means of (unobservable) rescaled block maxima. Respective results on the former can be found in Theorem~2.1 in \cite{FerDeh15}, while the latter approach was taken in Theorems~2.6/B.1 in \cite{BucZan23}, in the proof of Theorem 2.6 in \cite{BucSeg18-sl} or in Theorem 2.4 in \cite{ZouVolBuc21}. Throughout this paper, we also follow the latter approach, and for completeness, we summarize the essence in a theorem and briefly summarize and discuss potential statistical applications.

For $i\in \{1, \dots, n-r+1\}$, write $\bm M_{r,i}= \max(\bm X_{i}, \dots, \bm X_{i+r-1})$, 
and define $\overline \Mc_{n,r}^{[\dbl]}= (\bm M_{r,i}: i \in I_n^\dbl)$ and $\overline \Mc_{n,r}^{[\sbl]}= (\bm M_{r,i}: i \in I_n^\sbl)$ as the (vanilla) disjoint and sliding block maxima samples, respectively, where $I_n^\dbl=\{(i-1)r+1: 1 \le i \le n/r\}$ and $I_n^\sbl=\{1, \dots, n-r+1\}$. 
We are interested in the associated empirical measures $n_\mbl^{-1} \sum_{i \in I_n^\mbl} \delta_{\bm M_{r,i}}$, or their versions based on rescaled block maxima $n_\mbl^{-1} \sum_{i \in I_n^\mbl} \delta_{(\bm M_{r,i} - \bm b_r)/\bm a_r}$, with $n_\mbl=|I_n^\mbl|$ and $\delta_{\bm z}$ the Dirac-measure at $\bm z$.
The fact that the sample size $n_\mbl$ depends on $\mbl$ is a notational nuisance which we subsequently resolve by the following asymptotically negligible modification: first, the disjoint block maxima sample may be identified with the sample $\Mc_{n,r}^{[\dbl]}=(\bm M_{r,1}, \dots, \bm M_{r,1}, \dots, \bm M_{r,n/r}, \dots, \bm M_{r, n/r})$ of size $n$ containing each disjoint block maximum exactly $r$ times; note that the respective empirical measures of $\overline \Mc_{n,r}^{[\dbl]}$ and $\Mc_{n,r}^{[\dbl]}$ are the same. Next, we define $\Mc_{n,r}^{[\sbl]}$ as the sliding block maxima sample of size $n$ calculated from the extended sample $(\bm X_1, \dots, \bm X_n, \bm X_{n+1}, \dots, \bm X_{n+r-1})$. In the subsequent asymptotic results, this modification is asymptotically negligible since the first $n-r+1$ maxima in $\Mc_{n,r}^{[\sbl]}$ are exactly the maxima in $\overline \Mc_{n,r}^{[\sbl]}$.
These modifications allow to define $(\bm M_{r,1}^{[\mbl]}, \dots, \bm M_{r,n}^{[\mbl]}) := \Mc_{n,r}^{[\mbl]}$, and using the notation  $\bm Z_{r,i}^{[\mbl]} := (\bm M_{r,i}^{[\mbl]}-\bm b_r)/\bm a_r$, we define
\[
\bar \Gb_{n,r}^{[\mbl]} = \sqrt{\frac{n}r} \big(\Prob_{n,r}^{[\mbl]} - P_r\big),
\quad
\Gb_{n,r}^{[\mbl]} = \sqrt{\frac{n}r} \big(\Prob_{n,r}^{[\mbl]} - P\big),
\]
where $\Prob_{n,r}^{[\mbl]} = \frac{1}{n} \sum_{i=1}^n \delta_{\bm Z_{r,i}^{[\mbl]}}$, $P_r = \Prob(\bm Z_{r} \in \cdot)$ and $P = \Prob(\bm Z \in \cdot)$. Finally, let
\begin{equation} 
\label{eq:Gxi-prod}
G_\xi(\bm x, \bm y)= 
    G(\bm x)^{\xi \wedge 1} G(\bm y)^{\xi \wedge 1} G(\bm x \wedge \bm y)^{1-(\xi \wedge 1)}, \qquad \bm x, \bm y \in \R^d,
\end{equation}
and note that $G_\xi$ is a $2d$-variate extreme-value distribution; see Formula (3.8) and its discussion in \cite{BucSta24} for further details.
The following result relies on some suitable integrability and bias conditions (Conditions \ref{cond:int_h} and \ref{cond:bias_dbsl}), which are postponed to Section~\ref{sec:conditions} in the appendix for presentation reasons.

\begin{theorem}
\label{theo:cltblocks1}
Under Conditions~\ref{cond:mda} and  \ref{cond:ser_dep}, 
for any finite set of real valued functions $h_1, \dots, h_q$ satisfying the integrability Condition \ref{cond:int_h}(a) with $\nu > 2/\omega$ where $\omega$ is from Condition \ref{cond:ser_dep}, we have, writing $\bm h=(h_1, \dots, h_q)^\top$,
\[
\bar \Gb_{n,r}^{[\mbl]} \bm h = \big(\bar \Gb_{n,r}^{[\mbl]} h_1, \dots, \bar \Gb_{n,r}^{[\mbl]} h_q\big)^\top \weak
\Nc_q(\bm 0, \Sigma_{\bm h}^{[\mbl]}),  \qquad \mbl \in \{ \dbl, \sbl\},
\]
where
\begin{equation}
\label{eq:asy_Cov}
    \big(\Sigma_{\bm h}^{[\dbl]}\big)_{j,j^\prime=1}^q = \Cov\big(h_j(\bm Z), h_{j^\prime}(\bm Z) \big), \quad
    \big(\Sigma_{\bm h}^{[\sbl]}\big)_{j,j^\prime=1}^q = 2 \int_0^1 \Cov\big(h_j(\bm Z_{1,\xi}), h_{j^\prime}(\bm Z_{2,\xi}) \big) \diff \xi,
\end{equation}
with $\bm Z\sim G$ from Condition~\ref{cond:mda} and $(\bm Z_{1,\xi}, \bm Z_{2,\xi}) \sim G_\xi$ from \eqref{eq:Gxi-prod}.
Moreover, $ \Sigma_{\bm h}^{[\sbl]} \le_L  \Sigma_{\bm h}^{[\dbl]}$, where $\le_L$ denotes the Loewner ordering. Finally, if the bias Condition \ref{cond:bias_dbsl} holds, then
\[
\Gb_{n,r}^{[\mbl]} \bm h = \big(\Gb_{n,r}^{[\mbl]} h_1, \dots, \Gb_{n,r}^{[\mbl]} h_q\big)^\top \weak
\Nc_q(\bm B_{\bm h}, \Sigma_{\bm h}^{[\mbl]}).
\]
\end{theorem}

In statistical applications, 
one is typically interested in distributional approximations for the estimation error of general statistics depending on the observable block maxima $\bm M_{r,i}$ (here and in the following, we omit the superscript `$\mbl$').
As discussed in the following remark, such approximations can be deduced from Theorem~\ref{theo:cltblocks1}.

\begin{remark}[Normal approximations for statistics depending on observable block maxima]
\label{rem:normal}
Suppose that $\hat {\bm \theta}_n = {\bm \varphi}_n(\bm M_{r,1}, \dots, \bm M_{r,n})$ is some estimator of interest for a target parameter ${\bm \theta}_r \in \R^p$ (with $p\in\N$), where ${\bm \varphi}_n$ is a measurable function on $(\R^d)^n$ with values in $\R^p$. Asymptotic normality of $\hat {\bm \theta}_n$ can often be deduced from asymptotic normality of some possibly different function evaluated in the (unobservable) rescaled block maxima  $\bm Z_{r,1}, \dots, \bm Z_{r,n}$; see, for instance, \cite{Seg01, BucZan23, Oor23, BucSta24}. In this remark, we provide some high-level conditions under which such a conclusion is possible.

As discussed in Example~\ref{ex:mnstat} below, for many functions ${\bm \varphi}_n$ of practical interest there exists a positive integer $q$ and functions ${\bm \psi}_n\colon (\R^d)^n \to \R^{q}$, $A \colon (0,\infty)^d \times \R^d \to \R^{p \times q}$ and $B \colon \R^d \times (0, \infty)^d \to \R^{p}$ such that 
\begin{equation}\label{eq:phi_decomposition}
    {\bm \varphi}_n(\bm m_1 , \dots, \bm m_n) 
    =
    A(\bm a, \bm b) {\bm \psi}_n\Big( \frac{\bm m_1 - \bm b}{\bm a}, \dots,  \frac{\bm m_n - \bm b}{\bm a} \Big)  + B(\bm a, \bm b)
\end{equation}
for all $\bm m_1, \dots, \bm m_n \in \R^d$ and $\bm b \in \R^d, \bm a \in (0,\infty)^d$.
In such a case, if we further assume that there exists a sequence ${\bm \vartheta}_r \in \R^{q}$ solving the equations 
\begin{equation}\label{eq:vartheta}
    {\bm \theta}_r = A(\bm a_r, \bm b_r) {\bm \vartheta}_r + B(\bm a_r, \bm b_r)
\end{equation}
(again, see below for examples), we immediately obtain that the estimation error of $\hat {\bm \theta}_n$ can be written as
\begin{equation}\label{eq:theta_vartheta}
\hat {\bm \theta}_n - {\bm \theta}_r = A(\bm a_r,  \bm b_r) ( \hat {\bm \vartheta}_n - {\bm \vartheta}_r),
\end{equation}
where $\hat {\bm \vartheta}_n := {\bm \psi}_n(\bm Z_{r,1}, \dots, \bm Z_{r,n})$ is a function of the rescaled block maxima. The difference $\hat {\bm \vartheta}_n-{\bm \vartheta}_r$ often allows for a linearization (for instance, by the delta-method):
there exist real-valued functions $h_1, \dots, h_{q}$ such that
\begin{align} \label{eq:linerror}
\sqrt{\frac{n}r}( \hat {\bm \vartheta}_n - {\bm \vartheta}_r) = \bar \Gb_{n,r} (h_1, \dots, h_q)^\top + R_n
\end{align}
for some remainder $R_n=o_\Prob(1)$.
Loosely spoken, if the error term $R_n$ is sufficiently small, the previous two displays and Theorem~\ref{theo:cltblocks1} imply the normal approximation
\begin{align} \label{eq:normaltheta}
\hat {\bm \theta}_n - {\bm \theta}_r 
=
\sqrt{\frac{r}n} A(\bm a_r, \bm b_r) \big\{ \bar \Gb_{n,r} (h_1, \dots, h_q)^\top + R_n \big\}
\approx_d 
\Nc_p\Big(0, \frac{r}n A(\bm a_r, \bm b_r) \Sigma_{\bm h} A(\bm a_r, \bm b_r)^\top \Big),
\end{align}
where $ \Sigma_{\bm h}$ is the matrix from Theorem~\ref{theo:cltblocks1}. More precisely, in the case where the matrix $A$ is invertible (in particular, $p=q$), 
we obtain the more precise result
\begin{align} \label{eq:weaktheta}
A(\bm a_r, \bm b_r)^{-1} \sqrt{\frac{n}r} (\hat {\bm \theta}_r - {\bm \theta}_r) = \bar \Gb_{n,r} (h_1, \dots, h_q)^\top + o_\Prob(1) \weak \Nc_p(\bm 0, \Sigma_{\bm h}).
\end{align}
\end{remark}

\begin{example}
\label{ex:mnstat}

(i) Empirical variance: consider the case $d=1$ and the empirical variance
$
\hat \theta_n = n^{-1} \sum_{i=1}^n (M_{r,i} - \bar{M}_{r,n})^2,
$
where $\bar M_{r,n} := n^{-1} \sum_{i=1}^n M_{r,i}$, considered as an estimator for $\theta_r := \Var(M_{r,1})$. In that case, \eqref{eq:phi_decomposition} is met with $p=1, q=2$, ${\bm \psi}_n(z_1, \dots, z_n)=((\frac1n \sum_{i=1}^n z_i)^2, \frac1n  \sum_{i=1}^n z_i^2)^\top$, $B(a,b)=0$ and $A(a,b) = (-a^2,a^2)$. Moreover, ${\bm \vartheta}_r=(\Exp[Z_{r,1}]^2, \Exp[Z_{r,1}^2])^\top$ is a solution to \eqref{eq:vartheta} satisfying \eqref{eq:linerror}: indeed, with $h_1(x)=x, h_2(x)=x^2$, and using $P_r h_1 \to P h_1$, we may write 
\begin{align*}
\sqrt{\frac{n}r}( \hat {\bm \vartheta}_n - {\bm \vartheta}_r)
=
\sqrt{\frac{n}{r}} \begin{pmatrix}
    (\Prob_{n,r} h_1)^2 - (P_r h_1)^2 \\
    \Prob_{n,r} h_2 - P_r h_2
\end{pmatrix}
&=
\sqrt{\frac{n}{r}} \begin{pmatrix}
    (2P h_1 + o_\Prob(1)) ( \Prob_{n,r} h_1 - P_r h_1 ) \\
    \Prob_{n,r} h_2 - P_r h_2
\end{pmatrix}
\\&=
\bar \Gb_{n,r}((2 P h_1) \cdot h_1, h_2)^\top + o_\Prob(1). 
\end{align*}
Assembling terms and solving for $\hat \theta_n - \theta_r$, we obtain
$
\sqrt{\frac{n}{r}} a_r^{-2} ( \hat \theta_n - \theta_r )
= 
\bar \Gb_{n,r} (h_2 - (2P h_1) \cdot h_1) + o_\Prob(1).
$
The asymptotic variance of the limiting distribution has been explicitly calculated in \cite{BucSta24}, see their Corollary 4.1 and Equations (C.1) and~(C.2). 

\smallskip

(ii) Probability weighted moments: consider the case $d=1$ and let ${\bm \varphi}_n: \R^n \to \R^p$ denote the vector containing the first $p$ empirical probability weighted moments. In that case, Equation~\eqref{eq:phi_decomposition} is met with $q=p$, ${\bm \psi}_n={\bm \varphi}_n$, $A(a,b)=\diag(a, \dots, a)$ and $B(a,b)=\diag(b, b/2, \dots, b/(p+1))$; see Formula (A.7) in \cite{BucZan23}. Moreover, Equation~\eqref{eq:linerror} can be deduced from (A.5) in that paper (for $p=3$), with functions $h_1(x)=x$, $h_2(x)=x G_\gamma(x) +  \int_x^\infty z \diff G_{\gamma}(z)$ and $h_3(x) = x G_\gamma^2(x)+2\int_x^\infty z G_\gamma(z) \diff G_\gamma(z)$. Finally, up to the treatment of a bias term,  \eqref{eq:weaktheta} corresponds to the assertion in their Theorem 3.2. 

\smallskip

(iii) Pseudo maximum likelihood estimation in the heavy tailed case: the main steps from Remark~\ref{rem:normal} also apply in a slightly modified setting tailored to the heavy tailed case; details are worked out in Section~\ref{sec:frechet-small} below.
\end{example}

In statistical applications, results like those in \eqref{eq:normaltheta} are typically used as a starting point for inference, for instance in the form of confidence intervals for $\hat {\bm \theta}_n$. Routinely, such intervals would be based on normal approximations involving consistent estimators for the variances on the right-hand side of those displays, which in our case requires estimation of $\bm a_r, \bm b_r$ and $ \Sigma_{\bm h} = \Sigma_{\bm h}^{[\mbl]}$. This can be a complicated task, especially for the sliding block maxima method in view of the complicated formula for the limiting variance. Alternatively, one can rely on bootstrap approximations instead of normal approximations. However, as already illustrated in the introduction, a naive approach for bootstrapping sliding block maxima fails (see also Remark~\ref{rem:slidboot-inconsistent} below),  which prompts us to first introduce the circular block maxima method in the next section.

\section{The circular block maxima sample}
\label{sec:circmax}

In this section we introduce the circular block maxima method. We show that respective empirical means yield the same asymptotic variance as the sliding block maxima counterparts. The approach is interesting in its own right (particularly for computational reasons), but most importantly for this paper it suggests a straightforward bootstrap approach for bootstrapping sliding block maxima that will be discussed in Section~\ref{sec:bootgen}.

Formally, given a sample $\mathcal X_n=(\bm X_1, \dots, \bm X_n)$ as before, the circular block maxima sample 
\[
\Mc_{n,r}^{[\cblk]}= ( \bm M_{r,1}^{[\cblk]}, \dots, \bm M_{r,n}^{[\cblk]})
\]
is a new sample of size $n$ containing suitable block maxima that asymptotically follow the distribution $G$ in \eqref{eq:firstorder}.
The sample depends on two parameters: the block length parameter $r$ and an integer-valued parameter $k$ (for instance, $k=2$ or $k=3$) that determines the length of the interval over which we apply the circular maximum operation. Details for the univariate case are provided in Figure~\ref{fig:circ}, where the sampling period $\{1, \dots, n\}$ has been decomposed into $m(k):=m_n(k) := n/(kr)$ blocks of size $kr$, say $I_{kr,1}, \dots, I_{kr, m(k)}$ with $I_{kr,i}=\{(i-1)kr+1, \dots, ikr\}$, and where we assume that $m(k) \in \N$ for simplicity. Note that every observation $X_s$ appears within  exactly $r$ maximum-operations, and hence has the same chance to become a block maximum (this is not the case for the plain sliding block maxima sample, where, for instance, the very first observation $X_1$ can appear only once as sliding block maximum). This observation is in fact the main motivation for the circularization approach within each $kr$-block.

\begin{figure}[t!]
\begin{center}
\scalebox{0.89}{
    \begin{tikzpicture}[framed, rounded corners]
        \draw 
            (5,2) node(sample) {$\mathcal X=(X_1, \ldots, X_n) \in \R^n$}
%            (-1,0) node(kr1) {$(X_1, \ldots, X_{kr})$}
%            (2,0) node {$\dots$}
            (5,0) node(kri) {$(X_s)_{s \in I_{kr,i}} = (X_{(i-1)kr+1}, \ldots, X_{ikr}) \in \R^{kr}$}
%            (7.5,0) node {$\dots$}
%            (11,0) node(krm) {$(X_{(m_k-1)kr+1}, \ldots, X_{m_kkr})$}
            (5,-2) node(kri2) {$({\color{RedOrange} 
                                        \underbrace{X_{(i-1)kr+1}, \ldots, X_{(i-1)kr+r-1}}_{\text{first $r-1$ observations}}
                                    },
                                    {\color{Cerulean} 
                                        \underbrace{X_{(i-1)kr+r}, \ldots, X_{ikr}}_{\text{last $kr-(r-1)$ observations}}
                                    })\in \R^{kr}$
                                    }
            (5,-4.5) node(ckri) {$({\color{RedOrange} 
                                        X_{(i-1)kr+1}, \ldots, X_{(i-1)kr+r-1}
                                    },
                                    {\color{Cerulean} 
                                        X_{(i-1)kr+r}, \ldots, X_{ikr}
                                    },
                                    {\color{RedOrange} 
                                        \underbrace{X_{(i-1)kr+1}, \ldots, X_{(i-1)kr+r-1}}_{\text{first $r-1$ observations get repeated}}
                                    })\in \R^{kr+r-1}$}
            (5,-7) node(ckrimax) {$(\underbrace{\max({\color{RedOrange}X_{(i-1)kr+1}, \dots,X_{(i-1)kr+r-1}}, {\color{Cerulean}X_{(i-1)kr+r}})}_{
                                        \genfrac{}{}{0pt}{1}{\text{maximum over the first block of size $r$}}{\text{in the previous vector}}}, 
                                    \dots,
                                    \underbrace{\max({\color{Cerulean}X_{ikr}}, {\color{RedOrange}X_{(i-1)kr+1}, \ldots, X_{(i-1)kr+r-1}})}_{
                                        \genfrac{}{}{0pt}{1}{\text{maximum over the last (i.e., $kr$th) block of size $r$}}{\text{in the previous vector}}}
                                    )\in \R^{kr}$}
%            (5,-9.5) node(ckrimax2) {$\circmax(\pi_{kr,i}(\mathbb X) \mid r) \in \R^{kr}$};
            (5,-9.5) node(ckrimax2) {$(M_{r,s}^{[\cblk]})_{s \in I_{kr,i}} = (M_{r,(i-1)kr+1}^{[\cblk]}, \dots, M_{r,ikr}^{[\cblk]}) = \circmax((X_{(i-1)kr+1}, \ldots, X_{ikr}) \mid r)$}
            (5,-11.5) node(ckrmax) {$\Mc_{n,r}^{[\cblk]} := (M_{r,1}^{[\cblk]}, \dots, M_{r,n}^{[\cblk]})$};
%            (5,-11.5) node(hckrimax) {$\displaystyle\frac1{kr} \sum_{s=1}^{kr} 
%                                        h(\circmaX_s(\pi_{kr,i}(\mathbb X) \mid r)) $};

%        \draw[|->, very thick] (sample) to node[right] {$\pi_{kr,1}$} (kr1);
        \draw[|->, very thick] (sample) to node[right] {$\pi_{kr,i}$ (extract obs.\ from the $i$th $kr$-block $I_{kr,i}$)} (kri);
%        \draw[|->, very thick] (sample) to node[right] {$\pi_{kr,m_k}$} (krm);
        \draw[<->, very thick] (kri) to node[right] {$=$ (rewrite block)} (kri2);
        \draw[|->, very thick] (kri2) to node[right] {$\circf(\cdot \mid r)$} (ckri);
        \draw[|->, very thick] (ckri) to node[right] {$\slidmax(\cdot \mid r)$} (ckrimax);
        \draw[<->, very thick] (ckrimax) to node[right] {$=:$} (ckrimax2);
        \draw[|->, very thick] (ckrimax2) to node[right] {(concatenate blocks)} (ckrmax);
%        \draw[|->, very thick] (ckrimax2) to node[right] {$\bar h_{kr}$} (hckrimax);
        
%        \node[above,font=\bfseries] at (current bounding box.north) {Illustration of calculating the $i$th $kr$-block of circular block maxima};
    \end{tikzpicture}
}
\end{center}
\vspace{-.6cm}
\caption{Illustration of calculating the circular block maxima sample.}\vspace{-.2cm}
\label{fig:circ}
\end{figure}

More formally, the notations used in Figure~\ref{fig:circ} are defined as follows: for a given vector $\bm W = (\bm w_1, \ldots, \bm w_n) \in (\R^d)^n$ and $i\in \{1, \dots, m(k)\}$, we let
\begin{align*}    
\pi_{kr,i}(\bm W) 
    &= (\bm w_s)_{s \in I_{kr,i}} = \big( \bm w_{(i-1)kr+1}, \dots, \bm w_{ikr} \big)
\end{align*}
denote the projection of $\bm W$ to its coordinates defined by the $i$-th $kr$-block. For  $\bm V =(\bm v_1, \dots, \bm v_q) \in (\R^d)^q$ with $q \in \N_{\ge r}$, consider the sliding-maxima-operation
\begin{align*}
\slidmax(\bm V \mid r) 
    &= \Big(\max_{s\in [1:r]} \bm v_s, 
        \max_{s\in [2:r+1]} \bm v_s, \dots, \max_{s\in [q-r+1:q]} \bm v_s\Big) 
        \in (\R^d)^{q-r+1},
\end{align*}
where maxima over vectors in $\R^d$ are understood componentwise and where $[i:j] = \{i, i+1, \dots, j\}$.
Next, for 
$\bm {U} = (\bm u_1, \dots, \bm u_{kr}) \in (\R^d)^{kr} $, the circularization function is defined as
\begin{align*}
\circf(\bm U \mid r)
    &:= \big( \circf_1(\bm U \mid r), \dots, \circf_{kr+r-1}(\bm U \mid r) \big) 
    := (\bm u_1, \dots, \bm u_{kr}, \bm u_1, \dots, \bm u_{r-1}) \in (\R^d)^{kr+r-1}.
\end{align*}
Finally, we define $\circmax(\bm U\mid r):=\slidmax( \circf(\bm U \mid r ) \mid r)$, that is,
\begin{align*}
\circmax(\bm U \mid r)
    &= \big(\circmax_1(\bm U \mid r), \dots, \circmax_{kr}(\bm U \mid r) \big) \\
    &= \Big(\max_{s\in [1:r]} \circf_{s}(\bm U \mid r), 
        \max_{s\in [2:r+1]} \circf_s(\bm U \mid r), \dots, \max_{s\in [kr:kr+r-1]} \circf_{s}(\bm U \mid r)\Big) 
        \in (\R^d)^{kr}.
\end{align*}
The entire construction is also summarized in pseudo-code in Algorithm~\ref{alg:kmax-paper}.

\begin{algorithm}[t]
\caption{(Univariate) circ-max Method}
\label{alg:kmax-paper}
\begin{algorithmic}[1]
\Require Sample \( x = (x_1, \dots, x_n) \in \R^n \), block size \( r \), blocking parameter \( k \)
\Function{Circmax}{$x, r, k$}
\State \textbf{Partition:} Divide \( x \) into \( m(k) = n / (k r) \) disjoint blocks of size \( k r \). 
\For{each block $(x_{(i-1)kr+1}, \dots, x_{ikr}) \in \R^{kr}$ with \( i = 1, \dots, m(k) \)}
    \State $\mathcal{X}_{kr,i}^\text{circ} \gets (x_{(i-1) k r + 1}, \dots, x_{i k r}, x_{(i-1) k r + 1}, \dots,x_{(i-1) k r + r-1}) \in \R^{kr+r-1}$ \\
    \Comment{Circular trafo (repeat first $r-1$ observations)}
    \State $(M_{r, s}^{[\cblk]})_{s \in I_{kr,i}} \gets \text{slid-max}(\mathcal{X}_{kr,i}^{\text{circ}}  \mid r) \in \R^{kr}$  \Comment{Sliding maxima of circular blocks}
\EndFor
\State $\Mc_{n,r}^{[\cblk]} \gets (M_{r,1}^{[\cblk]}, \dots, M_{r,n}^{[\cblk]}) \in \R^n$  \Comment{Concatenate}
\State \Return \( \Mc_{n,r}^{[\cblk]} \)
\EndFunction
\end{algorithmic}
\end{algorithm}\vspace{-.2cm}

It is constructive to take a closer look at the choices $k=1$ and $k=n/r$ (the number of disjoint blocks). In the former case, the circmax-sample is the same as the disjoint block maxima sample, but with every observation repeated exactly $r$ times. In the latter case, the first $n-r+1$ observations of the circmax-sample coincide with the sliding block maxima sample, so the two are asymptotically equivalent. New samples are obtained for every other choice of $k$, and in the subsequent developments we will mostly be interested in choices like $k=2$ or $k=3$.

A key feature of the circmax-sample consists of the fact that it can be stored and evaluated efficiently. More precisely, for each $kr$ block, the number of distinct circmax-values in that block is typically very small (typically around 10 for realistic block and sample sizes as observed in the simulation study); for instance, the largest observation in each block is necessarily appearing exactly $r$ times. As such, the entire circmax-sample can also be regarded as a weighted sample whose (random) size corresponds to the number of distinct values in the circmax-sample.

The following central result shows that the circmax-sample can be considered as a sample from the `correct' limit distribution $G$ from \eqref{eq:firstorder} (a similar result on joint convergence of two circmax-observations can be found in Proposition \refstar{prop:overlap_wconv_kloop} of the supplement). 
Let 
\[
\bm Z_{r,s}^{[\cblk]} 
:= 
\frac{\bm M_{r,s}^{[\cblk]} - \bm b_r}{ \bm a_r},
\qquad s\in\{1, \dots, n\};
\]
note that we suppress the dependence on $k$ for notational convenience. 

\begin{proposition} [Weak convergence of circular block maxima]\label{prop:Z_kloop_xi_wconv}
Suppose that Condition~\ref{cond:mda} and \ref{cond:ser_dep}(a) are met and that $\alpha(r_n) \to 0$. Then, for every fixed $k \in \N$ and $\xi \in [0,k)$, we have 
\[
\bm Z_{r,1+\lfloor\xi r \rfloor}^{[\cblk]} \wconv \bm Z \sim G \qquad(n \to \infty).
\]
\end{proposition}

Due to Proposition~\ref{prop:Z_kloop_xi_wconv}, statistical methods based on the circular block maxima method may be expected to work asymptotically. As discussed in Section~\ref{sec:empmean}, a central ingredient for studying respective methods is weak convergence of empirical means. Adopting the notation from that section, we denote the empirical processes associated with the normalized sample $\bm Z_{r,1}^{[\cblk]}, \dots, \bm Z_{r,n}^{[\cblk]}$ by
\begin{align}
\bar \Gb_{n,r}^{[\cblk]} = \sqrt{\frac{n}r} \big(\Prob_{n,r}^{[\cblk]} - P_{n,r}^{[\cblk]}\big),
\quad
\tilde \Gb_{n,r}^{[\cblk]} = \sqrt{\frac{n}r} \big(\Prob_{n,r}^{[\cblk]} - P_r\big),
\quad
\Gb_{n,r}^{[\cblk]} = \sqrt{\frac{n}r} \big(\Prob_{n,r}^{[\cblk]} - P\big),
\label{eq:gn}
\end{align}
where $\Prob_{n,r}^{[\cblk]} = \frac{1}{n} \sum_{s =1}^n \delta_{\bm Z_{r,s}^{[\cblk]}}$, $P_{n,r}^{[\cblk]}=\frac1{kr} \sum_{s=1}^{kr} \Prob(\bm Z_{r,s}^{[\cblk]} \in \cdot)$, $P_r = \Prob(\bm Z_{r,1} \in \cdot)$ and $P = \Prob(\bm Z \in \cdot)$. 
The following result can be regarded as a circmax-counterpart of Theorem~\ref{theo:cltblocks1}. 
Recall that the integrability and bias Conditions~\ref{cond:int_h}-\ref{cond:bias_cbl} have been postponed to Section~\ref{sec:conditions} in the appendix, where we also provide an extended discussion of the bias in \ref{rem:bias_dhk}.

\begin{theorem}
\label{theo:cltblocks2}
Suppose that Conditions~\ref{cond:mda} and \ref{cond:ser_dep} are met. 
Then, for fixed $k \in \N_{\ge 2}$ and any finite set of real valued functions $h_1, \dots, h_q$ 
satisfying the integrability Condition \ref{cond:int_h}(b) with $\nu > 2/\omega$,  where $\omega$ is from Condition \ref{cond:ser_dep}, we have $
    \lim_{n \to \infty} \Cov(\bar \Gb_{n,r}^{[\cblk]} \bm h ) = \Sigma_{\bm h}^{[\sbl]}
$ and 
\[
\bar \Gb_{n,r}^{[\cblk]} \bm h 
\weak
\Nc_q(\bm 0, \Sigma_{\bm h}^{[\sbl]})
\]
with $\Sigma_{\bm h}^{[\sbl]}$ from \eqref{eq:asy_Cov}. Moreover, if the bias Condition \ref{cond:bias_cbl} holds, then
\[
\tilde \Gb_{n,r}^{[\cblk]} \bm h
\weak
\Nc_q(\tfrac{D_{\bm h,k}+E_{\bm h}}{k}, \Sigma_{\bm h}^{[\sbl]})
\]
(with $D_{\bm h,k}$ often equal to zero, see Remark~\ref{rem:bias_dhk}),
and if additionally the bias Condition \ref{cond:bias_dbsl} is met, then
\[
\Gb_{n,r}^{[\cblk]} \bm h
\weak
\Nc_q(B_{\bm h} + \tfrac{D_{\bm h,k}+E_{\bm h}}{k}, \Sigma_{\bm h}^{[\sbl]}).
\]
\end{theorem}

The proof of Theorem~\ref{theo:cltblocks2} (see in particular Proposition \refstar{prop:overlap_wconv_kloop} in the supplement) shows that the sample of $kr$-blocks containing the circular block maxima can be considered asymptotically independent. This is akin to the sample of disjoint block maxima and, as discussed after Theorem~\ref{theo:cltblocks1}, it suggests to define bootstrap versions of $\Gb_{n,r}^{[\cblk]}$ by drawing $m(k)$-times from the sample of $kr$-blocks with replacement (which is computationally efficient: recall that all circular block maxima in a $kr$-block could be efficiently stored as a weighted sample with typically very few observations).

\section{Bootstrapping block maxima estimators}
\label{sec:bootgen}

Throughout this section, we discuss bootstrap approximations for the empirical processes from Section~\ref{sec:circmax}. For completeness, we also cover the disjoint block maxima case, as it corresponds to the circular block maxima sample when $k=1$.

For fixed $k\in\N$, consider the circmax-sample $\mathcal M_{n,r}^{[\cblk]}$. Independent of the observations, let $\bm W_{m(k)}=(W_{m(k),1}, \dots, W_{m(k), m(k)}) = (W_{1}, \dots, W_{m(k)})$ be multinomially distributed with $m(k)$ trials and class probabilities $(m(k)^{-1}, \dots, m(k)^{-1})$. Given the sample $\mathcal M_{n,r}^{[\cblk]}$, the bootstrap sample $\mathcal M_{n,r}^{[\cblk],*}$ is obtained by repeating the observations $(M_{r,s}^{[\cblk]})_{s \in I_{kr,i}}$ from the $i$th $kr$-block exactly $W_{m(k),i}$-times, for every $i=1, \dots, m(k)$. The respective empirical measure of the rescaled bootstrap sample $\bm Z_{r,1}^{[\cblk],*}, \dots, \bm Z_{r,n}^{[\cblk],*}$ with $\bm Z_{r,1}^{[\cblk],*} = (\bm M_{r,s}^{[\cblk],*}-\bm b_r)/\bm a_r$ can then be written as
\[
\hat \Prob_{n,r}^{[\cblk],*}  =  
\frac{1}{n} \sum_{s =1}^{n} \delta_{\bm Z_{r,s}^{[\cblk],*}}
=
\frac{1}{n} \sum_{i =1}^{m(k)} W_{m(k),i} \sum_{s \in I_{kr,i}}\delta_{\bm Z_{r,s}^{[\cblk]}},
\]
and the associated (centered) empirical process is given by
\begin{align}
\label{eq:gnboot}
\hat \Gb_{n,r}^{[\cblk],*} = \sqrt{\frac{n}r} ( \hat \Prob_{n,r}^{[\cblk],*} - \Prob_{n,r}^{[\cblk]} ).
\end{align}
Note that the conditional distribution of $\hat \Gb_{n,r}^{[\cblk],*}$ given the data $\bm X_1, \dots, \bm X_n$ can in practice be approximated to an arbitrary precision based on repeated sampling of $\bm W_{m(k)}$ (if $\bm a_r, \bm b_r$ were known; we discuss extensions below).
The following result shows that the bootstrap process $\hat \Gb_{n,r}^{[\cblk],*}$ provides a consistent distributional approximation for $\bar \Gb_{n,r}^{[\cblk]}$ and for $\bar \Gb_{n,r}^{[\sbl]}$.

Again, note that the integrability and bias conditions~\ref{cond:int_h}-\ref{cond:bias_cbl} referenced in the results of this section can be found in the Section~\ref{sec:conditions}.
\begin{theorem}[Asymptotic validity of the circmax-resampling bootstrap]\label{theo:mult_boot}
Suppose Conditions \ref{cond:mda} and \ref{cond:ser_dep} are met. Then, for fixed $k \in \N_{\ge 2}$ and $\bm h$ satisfying the integrability Condition~\ref{cond:int_h}(b) with $\nu > 2/\omega$ where $\omega$ is from Condition \ref{cond:ser_dep}, we have, for $\mbl \in \{\sbl, \cblk\}$,
\begin{align*}
&d_K \Big( \mathcal L \big( \hat \Gb_{n,r}^{[\cblk],*} \bm h  \mid \mathcal X_n \big) , 
\mathcal L \big( \bar \Gb_{n,r}^{[\mbl]} \bm h \big) \Big) = o_\Prob(1)
\end{align*}
as $n \to \infty$,
where $d_K$ denotes the Kolmogorov metric for probability measures on $\R^q$. As a consequence, if the bias Condition~\ref{cond:bias_dbsl} is met and if $r$ is chosen sufficiently large to guarantee that $B_{\bm h}=0$, we have
\begin{align*}
&d_K \Big( \mathcal L \big( \hat \Gb_{n,r}^{[\cblk],*} \bm h \mid \mathcal X_n \big) , 
\mathcal L \big( \Gb_{n,r}^{[\sbl]} \bm h \big) \Big) = o_\Prob(1),
\end{align*}
and if also the bias Condition~\ref{cond:bias_cbl} is met  with $E_{\bm h} = D_{\bm h,k} =0$, we have
\begin{align*}
&d_K \Big( \mathcal L \big( \hat \Gb_{n,r}^{[\cblk],*} \bm h \mid \mathcal X_n \big) , 
\mathcal L \big( \Gb_{n,r}^{[\cblk]} \bm h \big) \Big) = o_\Prob(1).
\end{align*}
\end{theorem}

Remarkably, the consistency of the circmax bootstrap for the sliding block maxima method does not require the bias Condition~\ref{cond:bias_cbl} (nor $E_{\bm h}=D_{\bm h,k}=0$ in case it is met).

\begin{remark}[Bootstrapping the disjoint block maxima empirical process]
\label{rem:db_bst}
Consider the disjoint block maxima sample, whose empirical means coincide with empirical means based on the circular block maxima sample with $k=1$. The results from Theorem~\ref{theo:mult_boot} continue to hold for $k=1$ ($D_{\bm h,k}=E_{\bm h}=0$ is then immediate), but with `$\sbl$' replaced by `$\dbl$' at all instances.
\end{remark}

\begin{remark}[Inconsistency of naive resampling of sliding block maxima]
\label{rem:slidboot-inconsistent}
It seems natural to bootstrap sliding block maxima empirical means by blockwise resampling, that is, by the sliding block maxima analogue of \eqref{eq:gnboot}, 
\[
\hat \Gb_{n,r}^{[\sbl],*} = \sqrt{\frac{n}r} ( \hat \Prob_{n,r}^{[\sbl],*} - P_r ), \qquad
\hat \Prob_{n,r}^{[\sbl],*}  =  
\frac{1}{n} \sum_{s =1}^{n} \delta_{\bm Z_{r,s}^{[\sbl],*}}
=
\frac{1}{n} \sum_{i =1}^{m(k)} W_{m(k),i} \sum_{s \in I_{kr,i}}\delta_{\bm Z_{r,s}^{[\sbl]}}.
\]
Note that $\hat \Gb_{n,r}^{[\sbl],*}$ depends on $k$, which is suppressed from the notation. However, unlike $\hat \Gb_{n,r}^{[\cblk],*}$, this process is inconsistent for both $\bar\Gb_{n,r}^{[\sbl]}$ and $\bar{\Gb}_{n,r}^{[\cblk]}$, for any $k\in\N$. 
An heuristic explanation is provided in Section~\refstar{sec:inconsistency-sliding} in the supplement, where we also formally prove that, considering the case $d=q=1$ for simplicity, 
\begin{align}
\label{eq:slidboot-inconsistent}
d_K \Big( \mathcal L \big( \hat \Gb_{n,r}^{[\sbl],*} h  \mid \mathcal X_n \big) , 
\Nc(\bm 0, \Sigma_{h}^{(k)}) \Big) = o_\Prob(1),
\end{align}
where
$
\Sigma_{h}^{(k)} = \Sigma_h^{[\sbl]} - \frac2k \int_0^1 \xi \Cov(h(\bm Z_{1,\xi}), h(\bm Z_{2,\xi})) \diff \xi. 
$
If $\Var(h(\bm Z))>0$, we have $\Sigma_h^{(k)} <\Sigma_h^{[\sbl]}$ for any $k\in\N$ by Lemma~\refstar{lem:sliding-covpos} in the supplement,
so in view of Theorem~\ref{theo:cltblocks2} the bootstrap process has a smaller asymptotic variance than needed for approximating $\bar\Gb_{n,r}^{[\sbl]}$ or $\bar{\Gb}_{n,r}^{[\cblk]}$. This circumstance is one of the main motivations for working with circular block maxima.
\end{remark}

In statistical applications, as discussed after Theorem~\ref{theo:cltblocks1}, we do not want to bootstrap the estimation error of empirical means involving the (unobservable) $\bm Z_{r,i}$, but the estimation error of general statistics depending on the block maxima $\bm M_{r,i}$ itself. We follow the general setting of Remark~\ref{rem:normal}, that is,  $\hat {\bm \theta}_n = {\bm \varphi}_n(\bm M_{r,1}, \dots, \bm M_{r,n})$ is some estimator of interest for a target parameter ${\bm \theta}_r \in \R^p$.
If similar equations/linearizations as in \eqref{eq:theta_vartheta} and \eqref{eq:linerror} can be shown to hold for the (unobservable) circmax-bootstrap counterpart $\hat {\bm \vartheta}_n^{[\cblk],*} = {\bm \psi}_n(\bm Z_{1,r}^{[\cblk],*}, \dots, \bm Z_{n,r}^{[\cblk],*})$, that is, if
\begin{align}
\label{eq:bootlin1}
\hat {\bm \theta}_n^{[\cblk],*} - \hat {\bm \theta}_n^{[\cblk]} 
&= A(\bm a_r, \bm b_r) (\hat {\bm \vartheta}_n^{[\cblk],*} - \hat {\bm \vartheta}_n^{[\cblk]}), 
\\
\label{eq:bootlin2}
\sqrt{\frac{n}r}( \hat {\bm \vartheta}_n^{[\cblk],*} - \hat {\bm \vartheta}_n^{[\cblk]}) 
&= \hat \Gb_{n,r}^{[\cblk],*} (h_1, \dots, h_{q})^\top + o_\Prob(1)
\end{align}
then, loosely speaking, the distributional approximation $\Gb_{n,r}^{[\cblk]} \bm h \approx_d (\Gb_{n,r}^{[\cblk],*} \bm h \mid \Xc_n)$ from Theorem~\ref{theo:mult_boot} carries over to the distributional approximation $\hat {\bm \theta}_n^{[\cblk]} - {\bm \theta}_r \approx_d (\hat {\bm \theta}_n^{[\cblk],*}- \hat {\bm \theta}_n^{[\cblk]} \mid \Xc_n)$. Formally, we have the following result.

\begin{proposition} \label{prop:M_Bst_cons}
Fix $k\in \N_{\ge 2}$ and assume that Conditions~\ref{cond:mda} and \ref{cond:ser_dep} are met.
Suppose $\hat {\bm \theta}_n^{[\mbl]} = {\bm \varphi}_n(\bm M_{r,1}^{[\mbl]}, \dots, \bm M_{r,n}^{[\mbl]})$ is some estimator of interest for a target parameter ${\bm \theta}_r \in \R^p$ such that \eqref{eq:phi_decomposition}, \eqref{eq:vartheta} and \eqref{eq:linerror} (with $\bar \Gb_{n,r}=\bar \Gb_{n,r}^{[\mbl]}$) is met for some function $\bm h=(h_1, \dots, h_q)^\top$ satisfying the conditions of Theorem \ref{theo:cltblocks2}. Moreover, assume that \eqref{eq:bootlin1} and \eqref{eq:bootlin2} are met. 
Then, if $\Sigma_{\bm h}^{[\sbl]}$ is invertible, for $\mbl \in \{\sbl, \cblk\}$,
\begin{align*}
&d_K \Big( \mathcal L \big( \hat {\bm \theta}_n^{[\cblk],*}  - \hat {\bm \theta}_n^{[\cblk]} \mid \mathcal X_n \big) , 
\mathcal L \big( \hat {\bm \theta}_n^{[\mbl]}-{\bm \theta}_r  \big) \Big) = o_\Prob(1).
\end{align*}
\end{proposition}

The results from Proposition \ref{prop:M_Bst_cons} are sufficient for showing that  basic bootstrap confidence intervals \citep{DavHin97} asymptotically hold their intended level. As a proof of concept, details on a specific example are given in the next section, see also Corollary~\refstar{cor:confint-fremle} in the supplement.

\section{Application: bootstrapping the pseudo-maximum likelihood estimator for the Fr\'echet distribution}
\label{sec:frechet-small}

We provide details how the previous methods and results can be used for a specific important estimation problem, namely maximum likelihood estimation for parameters associated with a univariate heavy-tailed version of the MDA Condition~\ref{cond:mda}. An extended version of this section, including formal mathematical statements and proofs, can be found in the supplement, Section \refstar{sec:frechet}.

\begin{condition}[Fréchet Max-Domain of Attraction]\label{cond:mda_Fre-small}
Let $(X_t)_{t\in\Z}$ denote a strictly stationary univariate time series with continuous margins. There exists some $\alpha_0 > 0$ and some sequence $(\sigma_r)_{r \in \N} \subset (0,\infty)$ such that
\[
\lim_{r \to \infty} \frac{\sigma_{\lfloor rs \rfloor }}{ \sigma_r} = s^{1/\alpha_0}, \,\, (s > 0)
\quad \text{ and } \quad
\frac{\max(X_1, \ldots, X_r)}{\sigma_r}  \wconv Z \sim P_{\alpha_0,1} \qquad (r \to \infty),
\]
where $P_{\alpha, \sigma}$ denotes the Fréchet-scale family on $(0,\infty)$ defined by its CDF $F_{\alpha,\sigma}(x) = \exp(-(x/\sigma)^{-\alpha})$; here $(\alpha, \sigma) \in (0,\infty)^2$.
\end{condition}

Suppose $X_1, \dots, X_n$ is an observed sample from $(X_t)_{t\in\Z}$ as in Condition~\ref{cond:mda_Fre-small}, and let $r=r_n$ denote a block length parameter.
In view of the heuristics and results from Sections~\ref{sec:mathpre} and \ref{sec:circmax}, the associated block maxima samples $\Mc_{n,r}^{[\mbl]}$ with $\mbl \in \{\dbl, \sbl, \cblk\}$ (with $k \in \N$ fixed) can all be considered approximate samples from $P_{\alpha_0, \sigma_r}$. As in \cite{BucSeg18a, BucSeg18-sl}, this suggests to estimate $(\alpha_0, \sigma_r)$ by maximizing the independence Fréchet-log-likelihood, that is, we define
\begin{align}
\label{eq:mle-frechet-small}
\hat {\bm \theta}_n^{[\mbl]} := (\hat \alpha_n^{[\mbl]}, \hat \sigma_n^{[\mbl]})^\top := \argmax_{{\bm \theta}= (\alpha, \sigma) \in (0,\infty)^2}
\sum_{M_i \in \Mc_{n,r}^{[\mbl]}}  \ell_{\bm \theta} (M_i \vee c),
\end{align}
where, for ${\bm \theta}=(\alpha, \sigma)^\top \in (0,\infty)^2$, 
$
%\label{eq:frechetloglik}
\ell_{\bm \theta}(x) = \log(\alpha/\sigma) - (x/\sigma)^{-\alpha} - (\alpha+1) \log(x/\sigma)
$
denotes the log density of the Fréchet distribution $P_{(\alpha, \sigma)}$ and where $c>0$ denotes an arbitrary truncation constant. As shown in the last-named references, there exists a unique maximizer of \eqref{eq:mle-frechet-small} (with probability tending to one), and the rescaled estimation error $\sqrt{n/r}(\hat \alpha_n^{[\mbl]} - \alpha_0, \hat \sigma_n^{[\mbl]} / \sigma_{r_n} - 1)$ for both $\mbl=\dbl$ and $\mbl=\sbl$ is asymptotically normal under suitable regularity conditions. 

For fixed $k\in\N$, consider the circmax-sample $\mathcal M_{n,r}^{[\cblk]}$ and its bootstrap version $\mathcal M_{n,r}^{[\cblk], *}$ as constructed in Section~\ref{sec:bootgen}. In the supplement we show that, under suitable regularity assumption, the independence Fréchet-log-likelihood 
\[
{\bm \theta} \mapsto 
\sum_{s=1}^n \ell_{\bm \theta}(M_{r,s}^{[\cblk],*} \vee c) =
\sum_{i=1}^{m(k)} W_{m(k), i} \sum_{s \in I_{kr,i}} \ell_{\bm \theta}(M_{r,s}^{[\cblk]} \vee c).
\]
has a unique maximimizer (with probability converging to one), say $(\hat \alpha_n^{[\cblk],*}, \hat \sigma_n^{[\cblk],*})$, and that the conditional distribution of the rescaled bootstrap-estimation error, $\sqrt{n/r}(\hat \alpha_n^{[\cblk],*} - \hat \alpha_n^{[\cblk]}, \hat \sigma_n^{[\cblk],*} / \hat \sigma_n - 1)$, given the observations is close to the distribution of $\sqrt{n/r}(\hat \alpha_n^{[\mbl]} - \alpha_0, \hat \sigma_n^{[\mbl]} / \sigma_{r_n} - 1)$ for both $\mbl \in {\sbl, \cblk}$: for $n\to\infty$, we have
\begin{align*}
d_K\bigg[ 
\Lc\bigg( 
\sqrt{\frac{n}r}
        \begin{pmatrix}
            \hat \alpha_n^{[\cblk],*} - \hat \alpha_n^{[\cblk]} \\
            \hat \sigma_n^{[\cblk],*}/\hat \sigma_n^{[\cblk]} - 1
        \end{pmatrix} \,\Big|\, \mathcal X_n\bigg) ,
\sqrt{\frac{n}r}
        \begin{pmatrix}
            \hat \alpha_n^{[\mbl]} - \alpha_0 \\
            \hat \sigma_n^{[\mbl]}/\sigma_{r_n} - 1
        \end{pmatrix}
        \bigg] = o_\Prob(1).
\end{align*}

The result in the previous display allows for statistical inference on the parameters $\alpha_0, \sigma_{r_n}$, for instance in the form of confidence intervals. We provide details on $\sigma_r$, using the circular block bootstrap approximation to the sliding block estimator: for $\beta \in (0,1)$, let 
$q_{\hat \sigma_n^{[\cblk],*}}(\beta)$ denote the $\beta$-quantile of the conditional distribution of 
$\hat \sigma_n^{[\cblk],*}$ given the data, 
that is,
$q_{\hat \sigma_n^{[\cblk],*}}(\beta)= (F_{\hat \sigma_n^{[\cblk],*}})^{-1}(\beta)$, where
$
F_{\hat \sigma_n^{[\cblk],*}}(x) = \Prob( \hat \sigma_n^{[\cblk],*} \le x \mid \Xc_n)$ for $x \in \R$.
Note that the quantile may be approximated to an arbitrary precision by repeated bootstrap sampling. Consider the following version of the basic bootstrap confidence interval \citep{DavHin97}
\begin{align*}
I_{n,\sigma}^{(\sbl, \cblk)}(1-\beta) &= 
\Big[\hat \sigma_n^{[\sbl]}  + \{ \hat \sigma_n^{[\cblk]} - q_{\hat \sigma_n^{[\cblk],*}}(1-\tfrac\beta2) \} , \hat \sigma_n^{[\sbl]}  + \{ \hat \sigma_n^{[\cblk]} - q_{\hat \sigma_n^{[\cblk],*}}(\tfrac\beta2)\} \Big].
\end{align*}
Note that the interval is anchored to the sliding blocks estimator $\hat \sigma_n^{[\sbl]}$, and that its size is controlled by quantiles of the circular blocks estimation error, $\hat \sigma_n^{[\cblk],*} - \hat \sigma_n^{[\cblk]}$.
We will show in the supplement, see Corollary~\refstar{cor:confint-fremle}, that
\[
\lim_{n \to \infty} \Prob\big( \sigma_{r_n} \in I_{n,\sigma}^{(\sbl, \cblk)}(1-\beta)\big) = 1-\beta,
\]
and that an analogous result can be derived for the disjoint block maxima estimator. In view of the fact that the disjoint block maxima estimator exhibits a larger asymptotic estimation variance, the width of the disjoint blocks confidence interval is typically larger than the one of the sliding-circular method, for every fixed confidence level (for an empirical illustration, see Figure~\refstar{fig:plotFreFixNCi} for the related problem of providing a confidence interval for the shape parameter $\alpha$).

\section{Simulation study}
\label{sec:sim}

The finite-sample properties of the estimators and of the bootstrap approaches have been investigated in a large scale Monte-Carlo simulation study. Two asymptotic regimes were considered: first, the case where the block size $r$ is fixed and the number of disjoint blocks is increasing which can be regarded as the most important use case for the block maxima method), and second, the case where $n$ is fixed and the block size is treated as a tuning parameter. For the sake of brevity, we only report results for the fixed block size case in the main paper; the other results are reported in the supplement.

Several target parameters associated with the law of $\bm M_r$ have been considered. In the main paper, we only report results on the univariate case and the estimation of the $(1-1/T)$-quantile of $M_r$, denoted as  $\rl(T,r)$, with $T \in \{50,100,200\}$. Note that $\rl(T,r)$ is a standard target parameter in environmental extremes, as it corresponds to the $(T,r)$-return level of the underlying time series, that is, it will take on average $T$ IID disjoint blocks of size $r$ until the first such block whose maximum exceeds $\rl(T,r)$. In the supplement, we also present results for the expected value $\mu_r = \Exp[M_{r}]$ (univariate case), as well as for Spearman's rho $\rho=\rho(\bm M_r)$ (bivariate case).

Regarding return level estimation, we proceed as follows: the DoA Condition~\ref{cond:mda} implies that the law of $M_r$ is approximately GEV with location parameter $b_r$, scale parameter $a_r$ and shape parameter $\gamma$.  As a consequence, $\rl(T,r) \approx \rl^\circ(T,r)$, where 
\[
\rl^\circ(T,r) = a_r \frac{c_T^{-\gamma}-1}\gamma + b_r, \qquad c_T=-\log(1-1/T)
\]
is the $(1-1/T)$-quantile of the $\mathrm{GEV}(b_r, a_r, \gamma)$ distribution. This suggests the disjoint, sliding and circular block maxima estimators
\begin{align}
\label{eq:rlhat}
\rlhat^{[\mbl]}
:=
\rlhat^{[\mbl]}(T,r)
= \hat a_r^{[\mbl]} \frac{c_T^{-\hat \gamma^{[\mbl]}}-1}{\hat \gamma^{[\mbl]}} + \hat b_r^{[\mbl]}, \qquad \mbl \in \{\dbl, \sbl, \cblk\},
\end{align}
where $(\hat a_r^{[\mbl]}, \hat b_r^{[\mbl]} , \hat \gamma^{[\mbl]})$ are suitable estimators obtained by fitting the GEV-distribution to the respective sample of block maxima; here, $\cblk$ denotes the circmax estimator with parameter $k=2$ (other choices lead to similar results). Throughout, as in Section~\ref{sec:frechet-small}, we employ (pseudo)-maximum likelihood estimation, that is, 
\begin{align*}
\hat \theta^{[\mbl]} := (\hat a_r^{[\mbl]}, \hat b_r^{[\mbl]} , \hat \gamma^{[\mbl]}) 
\in 
\argmin_{\theta} \sum_{M_i \in \mathcal M_{n,r}^{[\mbl]} } \ell_\theta(M_i),
\end{align*}
where $\ell_\theta$ denotes the log-likelihood function of the GEV family.

Two different classes of time series models have been employed: the first one is the ARMAX-GPD-Model, which depends on two simple parameters controlling the serial dependence and the shape parameter $\gamma$ in the DoA condition, respectively, and the second one comprises three rather specific models designed for the extreme value analysis of precipitation amounts, temperatures, and stock returns, respectively. The first class of models will be treated in this section, and the next one will be part of the  next section.

The following model serves as a basis for understanding the influence of the serial dependence and the heaviness of the tail on the estimators and the bootstrap. Moreover, the finite-sample results will be used to construct a data-adaptive size-correction for the basic bootstrap confidence interval.

\begin{model}[ARMAX-GPD-Model]
\label{mod:armax-gpd}
$(X_t)_{t\in\Z}$ is a stationary real-valued time series whose stationary distribution is the generalized Pareto distribution $\mathrm{GPD}(0,1,\gamma)$ with CDF\ 
\begin{align*}
F_\gamma(x) := 
\begin{cases}
\left( 1- (1+\gamma x)^{-1/\gamma} \right) \bm{1}(x \geq 0), \quad &  \gamma > 0,\\
\left( 1- (1+\gamma x)^{-1/\gamma} \right) \bm{1}(0 \leq x \leq -1/\gamma), \quad &\gamma < 0,\\
\left( 1- \exp(-x) \right) \bm{1}(x \geq 0), & \gamma = 0.
\end{cases}
\end{align*}
After transformation to the $\text{Fréchet}(1)$-scale, the temporal dynamics correspond to the ARMAX(1)-model (Example 10.3 in \citealp{BeiGoeSegTeu04}). More precisely, the time series $(Y_t)_t$ with $Y_t= F_W^{\leftarrow}(F_\gamma(X_t))$ satisfies the recursion
\begin{equation} \label{eq:ArmaxRec}
    Y_t = \max \big( \beta Y_{t-1}, (1-\beta) W_t \big) \quad \forall t \in \Z
\end{equation}
for an IID sequence $(W_t)_{t \in \Z}$ of Fréchet-distributed random variables with shape parameter 1 and some $\beta \in [0,1)$;
here, $F_W$ denotes the CDF\ of a Fréchet(1) distributed random variable and $F^\leftarrow$ denotes the generalised inverse of $F$. Throughout the simulation study, we consider the choices $\beta =0$ (iid case) and $\beta=0.5$ combined with $\gamma\in\{-0.2, -0.1, 0, 0.1, 0.2\}$, 
giving a total of ten different models (for the sake of brevity, we only present results for $|\gamma| \in \{0, 0.2\}$).
Note that the process $(X_t)_t$ 
is exponentially beta mixing and
satisfies Condition~\ref{cond:mda} with $Z\sim \mathrm{GEV}(\gamma)$, $a_r = \{r(1-\beta)\}^\gamma$ and $b_r = \{r(1-\beta)^\gamma -1 \}/\gamma$. Moreover, $\Prob(M_r \le x) = F_\gamma(x)^{\beta + (1-\beta) r}$, whence the true return level is given by $\rl(T,r) = F_\gamma^{-1}( (1-1/T)^{1/\{\beta+(1-\beta)r\}})$.
\end{model}

For the simulation experiments, the block size was fixed to $r = 365$ as it corresponds to the number of days in a year - a common block size in environmental applications. The effective sample size (i.e., the number of disjoint blocks of size $r$, abbreviated by $ m = \flo{n/r}$ hereafter) has been varied between $50$ and $100$ resulting in total sample sizes $n$ between $18,250$ and $36,500$. For the sake of brevity, we only present results for $T=100$, and we also write $\rl=\rl(T=100, r=365)$.

\smallskip
\noindent \textbf{Performance of the estimators.}
We start by comparing the four estimators from \eqref{eq:rlhat} in terms of their variance, squared bias and mean squared error (MSE). 
The results are summarized in Figure~\ref{fig:RlRelMseBias}. As suggested by the theory, both the sliding blocks estimator and the two circular blocks estimators perform uniformly better than the respective disjoint blocks estimator. The improvement gets smaller as the observations are getting more heavy-tailed; a phenomenon that has already been observed in the literature; see e.g. \cite{BucZan23}, \cite{BucSta24}. 
It is noteworthy that the additional bias effect introduced by the circular blocks estimators is negligible: both circular and sliding blocks estimators perform almost identically in all scenarios.
In general, the contribution of the squared bias to the MSE was found to be relatively small for all estimators, and typically much smaller than the variance component (median around $1.4\%$, maximum $12\%$).
Furthermore, the serial dependence does not change the qualitative results significantly. 

\begin{figure}[t!] 
\centering
\makebox{\includegraphics[width=0.91\textwidth]{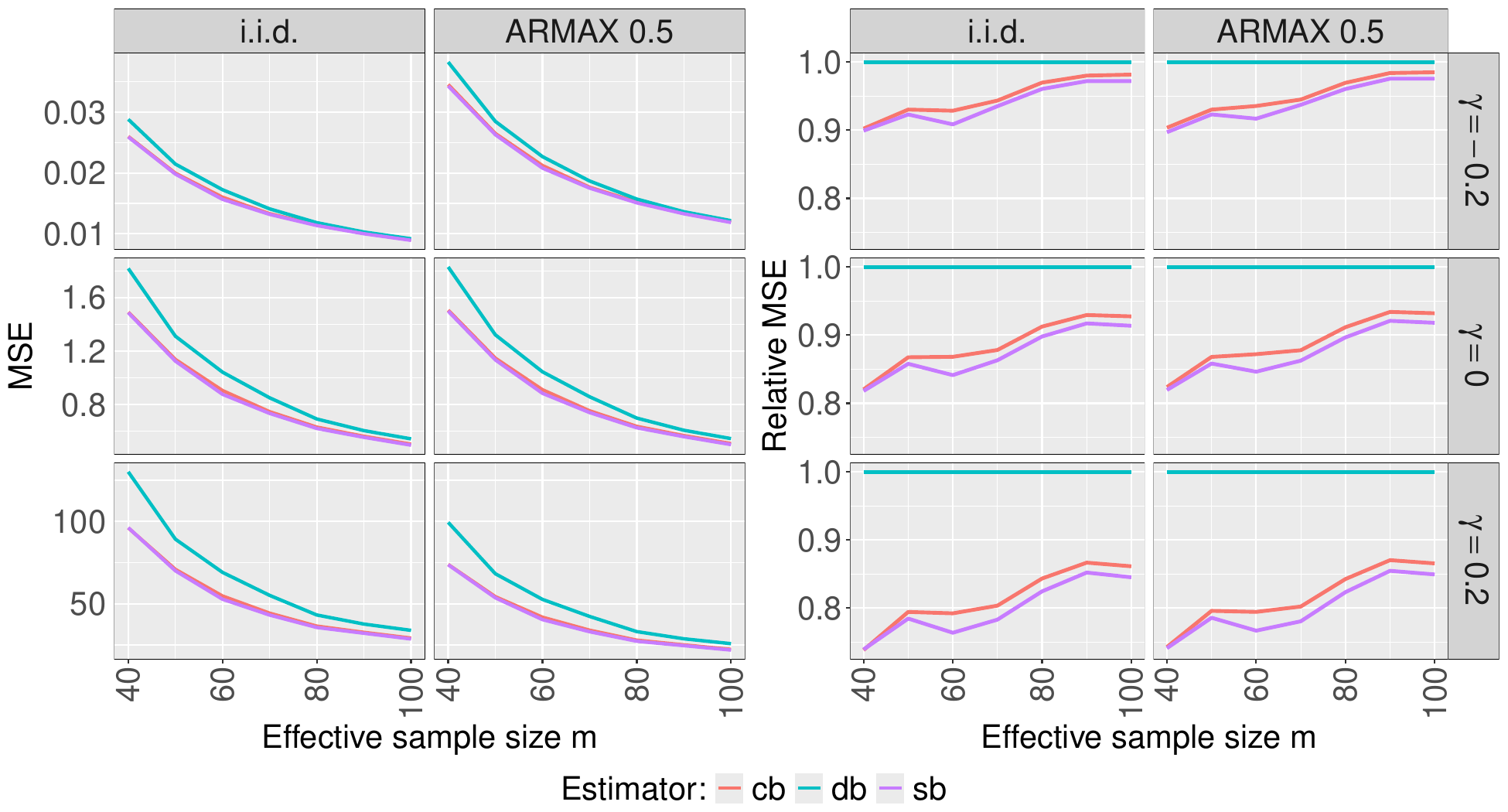}}\vspace{-.1cm}
\caption{Return level estimation with fixed block size $r = 365$ and ‘annuity’ $T = 100$. Left: Mean squared error $\mathrm{MSE}(\rlhat^{[\mbl]})$. Right: relative MSE with respect to the disjoint blocks method, i.e., $\mathrm{MSE}(\rlhat^{[\mbl]}) / \mathrm{MSE}(\rlhat^{[\dbl]})$.}\label{fig:RlRelMseBias}	
\vspace{-.1cm}
\end{figure}

\smallskip
\noindent \textbf{Performance of the bootstrap.}
We consider each of the bootstrap estimators $\rlhat^{[\mbl], \ast} := \rlhat^{[\mbl], \ast}(T=365,r=365)$ with $\mbl \in \{\dbl, \sbl, \cblk\}$,  with number of bootstrap replications set to $B=1,000$. Define 
\begin{align}
\label{eq:err}
\err^{[\mbl]} := \rlhat^{[\mbl]} - \rl,
\qquad
\err^{[\mbl], \ast} := \rlhat^{[\mbl], \ast} - \rlhat^{[\mbl]}
\end{align}
as the estimation error and the (conditional) bootstrap estimation error, respectively.
Recall that our results from the previous sections suggest that the conditional circular bootstrap estimation error $\err^{[\cblk], \ast}$  is consistent for the sliding estimation error $\err^{[\sbl]}$, that $\err^{[\dbl], \ast}$ is consistent for $\err^{[\dbl]}$, and that $\err^{[\sbl], \ast}$ is inconsistent for $\err^{[\sbl]}$. 

\begin{figure}[t!] 
\centering
\makebox{\includegraphics[width=.95\textwidth]{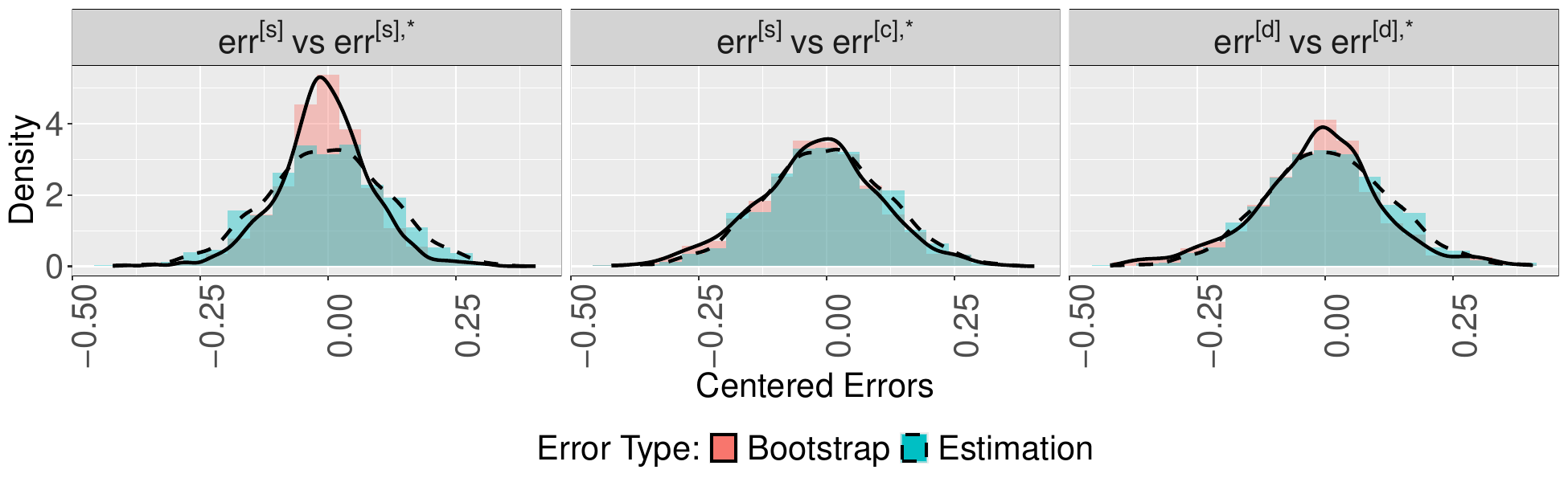}}\vspace{-.1cm}
\caption{Comparison of the error distributions from \eqref{eq:err} for return level estimation with fixed block size $r = 365$ and ‘annuity’ $T = 100$ in Model~\ref{mod:armax-gpd} with $\gamma = -0.2, \beta = .5, m = 80$.
}\label{fig:RlHists}	
\vspace{-.1cm}
\end{figure}

All three statements are illustrated in Figure~\ref{fig:RlHists} 
by means of histograms for the parameter choice $\gamma = -0.2, \beta=0.5, T=100$ and $m = 80$; other choices lead to similar results.

In the next step, we evaluate the performance of the three bootstrap approaches in terms of their ability to provide accurate estimates of the estimation variance $\sigma_\mbl^2 = \Var(\rlhat^{[\mbl]})$ with $\mbl \in \{\dbl, \sbl\}$.
For that purpose, we estimate the respective variance by the empirical variance of the sample of bootstrap estimates; recall that each such sample is of size $B=1,000$, for every $\mbl\in\{\dbl, \sbl, \cblk \}$. Despite the fact that the developed theory does not guarantee these estimators to be consistent (essentially as weak convergence does not imply convergence of moments in general), we observe a decent empirical performance:
indeed, in Figure \ref{fig:rlBstVar}, we depict the average over the $N=1,000$ bootstrap estimates, along with the true estimation variances that were determined in a presimulation based on $10^6$ repetitions; note that this variance depends reciprocally on the effective sample size $m$.
We observe that the bootstrap estimates are reasonably close to their target values for the disjoint and circmax method, while the naive sliding blocks bootstrap underestimates the true variance substantially, as indicated in Remark \ref{rem:slidboot-inconsistent} and as already apparent in Figure~\ref{fig:RlHists}.

\begin{figure}[t!] 
\centering
\makebox{\includegraphics[width=0.91\textwidth]{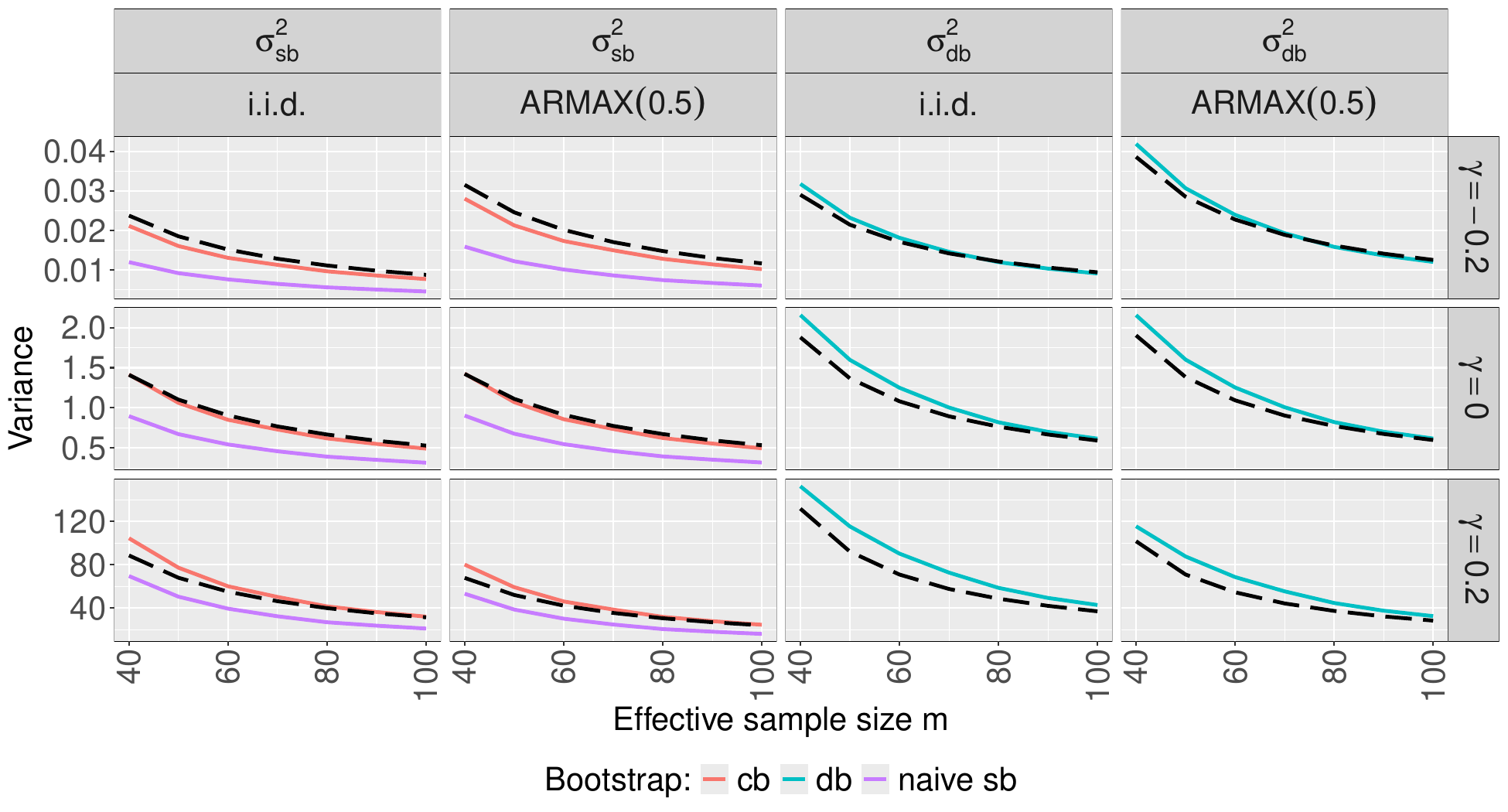}}\vspace{-.1cm}
\caption{Bootstrap-based estimation of the return level estimation variance $\sigma_\mbl^2 = \Var(\rlhat^{[\mbl]})$ with fixed block size $r=365$ and `annuity' $T=100$. Left two columns: target parameter $\sigma_\sbl^2$ (dashed line), with the two colored lines representing the empirical variance of the naive sliding and the circular bootstrap sample, respectively, averaged over $N=1,000$ simulation runs. Right two columns: the same with target parameter $\sigma_\dbl^2$ (dashed line) and the colored line the empirical variance of the disjoint bootstrap sample. 
}\label{fig:rlBstVar}	
\vspace{-.1cm}
\end{figure}

Finally, we evaluate the performance of the bootstrap approaches in terms of their ability to provide accurate confidence intervals of pre-specified coverage $1-\alpha=0.95$; clearly, smaller intervals of similar coverage would be preferable. For that purpose, we restrict ourselves to a version of the basic bootstrap confidence interval \citep{DavHin97} (see also Section~\ref{sec:frechet-small} and Corollary~\refstar{cor:confint-fremle} of the supplement) defined as
\begin{align*}
\CI^{[\sblcbl]} &= \Big[\rlhat^{[\sbl]} - \err^{[\cbl],*}_{\lfloor(1-\alpha/2) B \rfloor: B}), \rlhat^{[\sbl]} - \err^{[\cbl],*}_{\lfloor(\alpha/2) B \rfloor: B}) \Big], \\
\CI^{[\dbl]} &= \Big[\rlhat^{[\dbl]} - \err^{[\dbl],*}_{\lfloor(1-\alpha/2) B \rfloor: B}), \rlhat^{[\dbl]} - \err^{[\dbl],*}_{\lfloor(\alpha/2) B \rfloor: B}) \Big],
\end{align*}
where $\err^{[\mbl],*}_{1: B} \le \dots \le \err^{[\mbl],*}_{B: B}$ denote the order statistics of the $B$ bootstrap replicates of the estimation error $\err^{[\mbl],*}$; note that the sliding-circular interval is anchored to the sliding estimator, from which it moves away using the circular bootstrap estimation error.
The empirical coverage and the average widths are depicted in Figure~\ref{fig:rlBstCiVar}, where we omit the naive sliding method because of its inconsistency. We find that in most scenarios the desired coverage is not reached by any of the methods, in particular for smaller sample sizes.
The disjoint blocks approach has the best coverage overall, albeit with only minimal advantages and sometimes on a par with or even slightly worse than the cb-method. However, the price for this is universally wider confidence intervals.

\begin{figure}[t!] 
\centering
\makebox{\includegraphics[width=0.91\textwidth]{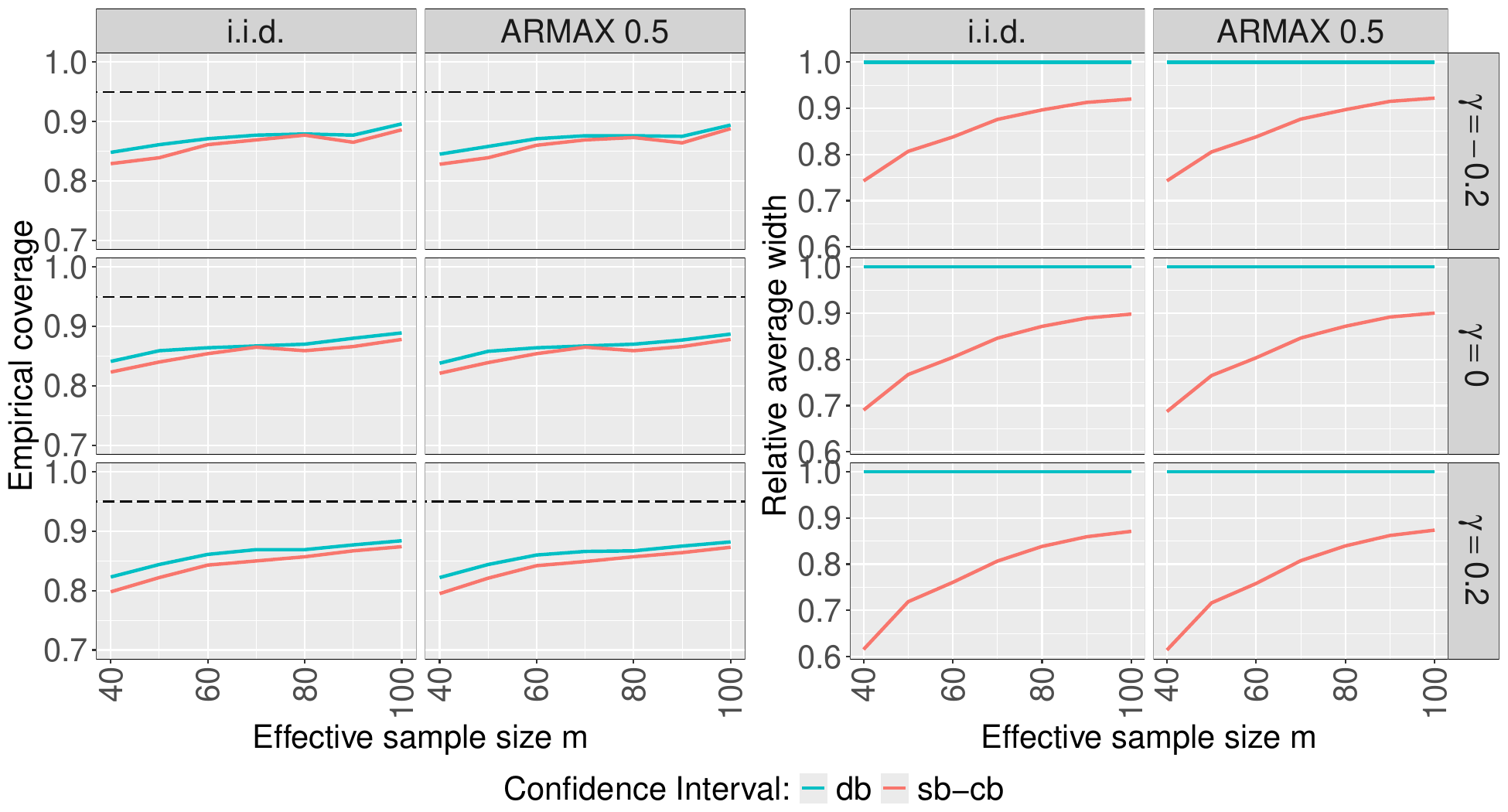}}\vspace{-.1cm}
\caption{Basic bootstrap confidence intervals for $\rl(T=100, r=365)$ based on the sliding-circular and the disjoint blocks approach.
Left: empirical coverage with intended coverage of $95\%$ (dashed line). 
Right: relative average width with respect to the disjoint method, i.e., $\mathrm{width}(\CI^{[\sblcbl]}) / \mathrm{width}(\CI^{[\dbl]})$.
}\label{fig:rlBstCiVar}	
\vspace{-.1cm}
\end{figure}

\smallskip
\noindent \textbf{Size-corrected confidence intervals.}
The results described in the previous paragraph show that the raw versions of the basic bootstrap confidence interval do not hold their intended level, in particular for small sample sizes.

To address this situation, we consider enlarged confidence intervals of the form 
\[
\CI(c)= \Big[\rlhat - c \times \err^{*}_{\lfloor(1-\alpha/2) B \rfloor: B}, \rlhat - c \times \err^{*}_{\lfloor(\alpha/2) B \rfloor: B} \Big], 
\]
where  $c \ge 1$ denotes an enlargement factor. The enlargements  allow for a size correction: for each scenario under consideration, we determine (a posteriori) the factor $c=c(\gamma, m, \beta, \mbl)$ such that the respective confidence interval has empirical coverage exactly equal to $1-\alpha=0.95$; here, $\mbl \in \{ \dbl, \sblcbl\}$. We may then compare the widths of the resulting confidence intervals, which are illustrated in Figure~\ref{fig:simaRelWidthFact}. We observe that the interval anchored to the sliding blocks estimator is universally smaller than the one anchored to the disjoint blocks estimator, with great advantages for small effective sample sizes.

\begin{figure}
    \centering
    \includegraphics[width=0.91\textwidth]{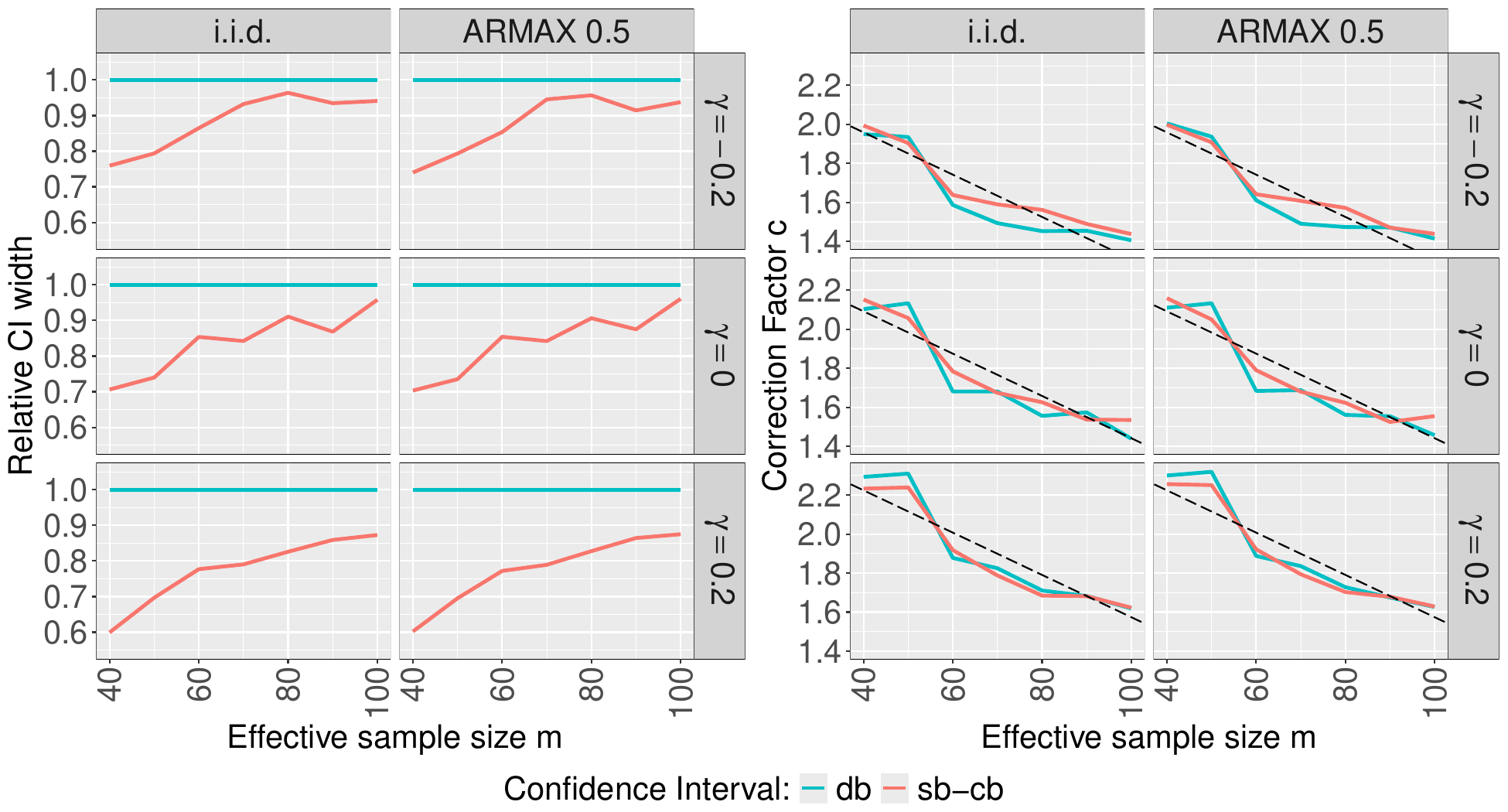}
    \caption{Left: Relative width of the size-corrected confidence interval for $\rl(T=100, r=365)$ rooted at the sliding estimator with respect to the one rooted at the disjoint blocks estimator, i.e., $\mathrm{width}(\CI^{[\sblcbl]}(c))  / \mathrm{width}(\CI^{[\dbl]}(c))$.
    Right: Enlargement factors $c$ that were needed to obtain confidence intervals for $\rl(T=100, r=365)$ of exact empirical coverage $1-\alpha=0.95$. The respective fitted regression curves $m \mapsto c(m, \gamma)$ are plotted as black dashed lines. }
    \label{fig:simaRelWidthFact}
\end{figure}

The results described in the previous paragraph are very encouraging and clearly show the potential advantages of the sliding-circular method. Unfortunately, however, the procedure cannot be applied in practice, as it relies on matching observed empirical coverages obtained from repetitions that are unavailable in practice. In the following, we describe a simple approach that leads to a data-adaptive method that can also be used in practice.

We start from the above factors $c$ for which the observed (empirical) coverage is exactly 95\%. The obtained values of $c$ for return level estimation, $\mathrm{RL}( T=100, r=365)$, are shown in Figure~\ref{fig:simaRelWidthFact} (right panel). We observe that the factors are decreasing in the number of years/blocks $m$ and increasing in the tail index $\gamma$, and with only little difference between the disjoint and sliding-circular approach and between the different values of $\beta$. To support these visual findings we run 
a simple linear regression with regressors 
$m$ (number of blocks), 
$\gamma$ (tail index), 
$\mbl \in \{\dbl, \sblcbl\}$ (method) and 
$\beta$ (time series parameter of the ARMAX-model), which showed that the number of blocks and the tail index are indeed the only significant parameters. The estimated linear relationship is 
$c(m, \gamma) = 2.48 -  0.01 \, m + 0.68\, \gamma$; 
the curves $m \mapsto c(m, \gamma)$ are also depicted in Figure~\ref{fig:simaRelWidthFact}.

The obtained linear relationship yields a rule of thumb for applications with unseen data of size $m$: simply apply $\CI(\hat c)$ with correction factor $\hat c= c(m, \hat \gamma)$, with $\hat \gamma$ estimated from the sample under consideration.
This method has been applied to the entire simulation design involving Model~\ref{mod:armax-gpd}; the respective results are depicted in Figure~\ref{fig:faCoRlCovWidth}. We observe that the coverage is close to 95\% for both methods, and that the bootstrap interval anchored to the sliding blocks estimator is still smaller than the disjoint blocks counterpart.

\begin{figure}
    \centering
    \includegraphics[width=0.91\textwidth]{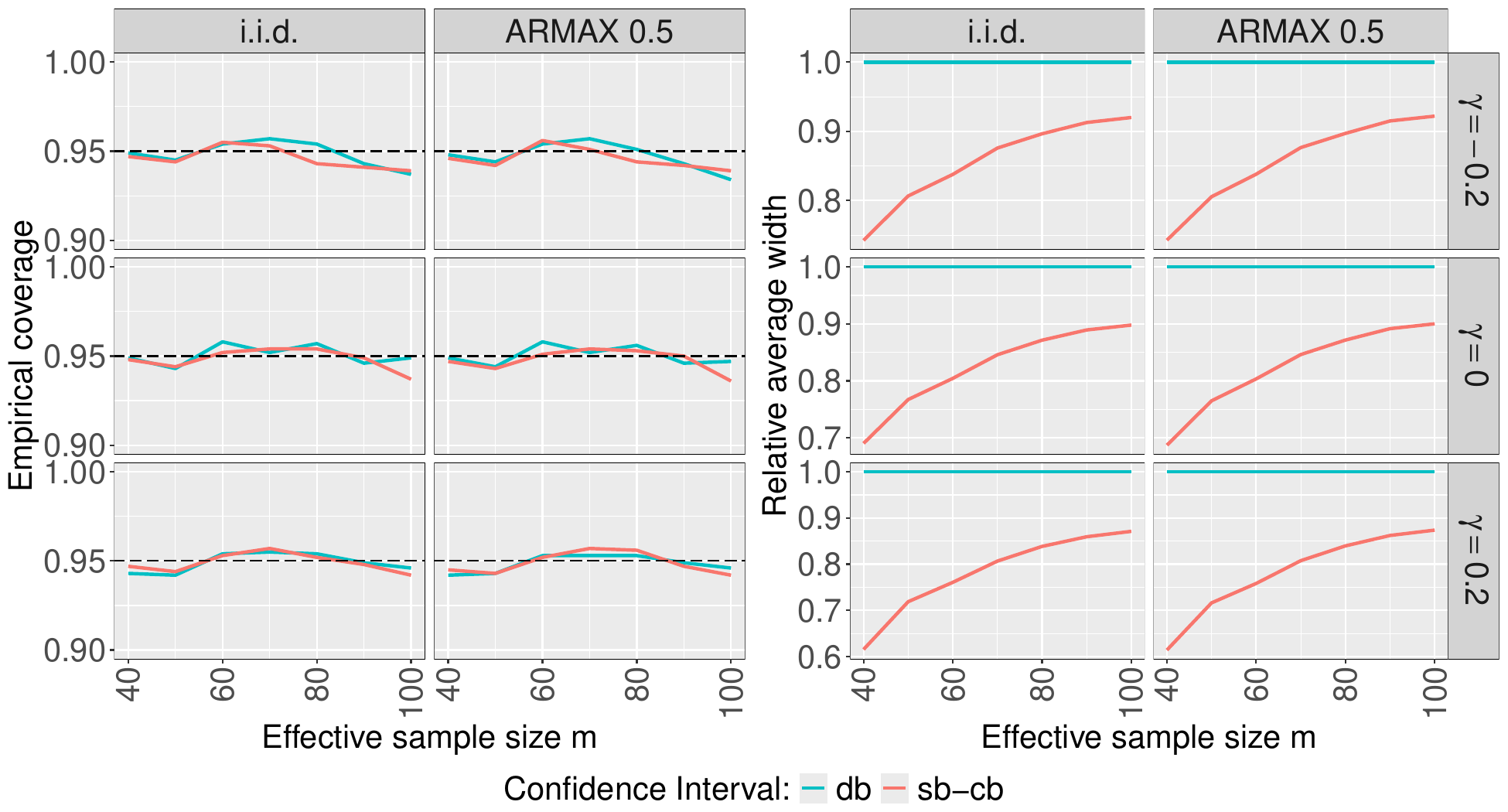}
    \caption{Left: empirical coverage of the data-adaptive size-corrected confidence intervals $\CI^{[\dbl]}(\hat c)$ and $\CI^{[\sblcbl]}(\hat c)$. Right: Respective relative width $\mathrm{width}(\CI^{[\sblcbl]}(\hat c))  / \mathrm{width}(\CI^{[\dbl]}(\hat c))$}. 
    \label{fig:faCoRlCovWidth}
\end{figure}

Although the approach described above yields very promising results,  some shortcomings should not go unmentioned: first, the above correction factors were obtained for return level estimation with $T=100$, and in applications with a different target parameter (e.g., the 200-year return level or small exceedance probability), further simulation results have shown that other factors may be needed. 
Similar caution is required if the block size is outside the range $m \in [40,100]$ or if the tail index is outside the range $\gamma \in [-.2,.2]$. In both cases, additional simulation experiments would be needed to derive suitable correction factors. A more detailed examination of this issue is postponed to future research.

\section{Case study}
\label{sec:case}

\subsection{Specific time series models for three selected applications} \label{subsec:sim-applied-models}

In this section, we provide finite-sample results for three selected time series models that were trained using real data and that allow for arbitrary long simulation. The first two data models correspond to meteorological variables (daily cumulative precipitation amounts, and daily maximal temperatures during the summer months, respectively), and were simulated using weather generators:
first, regarding precipitation, we employed the stationary version of the stochastic, multi-site, multivariate weather generator from \cite{Ngu21, Ngu24}, trained on observational data from the Bamberg station \citep{Mac24} for the period 1961–2010, where precipitation is modeled using a monthly-based extended GPD. Second, regarding temperature, we employed the weather generator from \cite{wxgen19}, trained on observed records at the Hohenpeißenberg station in Germany over an (approximately stationary) 20 years window from 2000 to 2020.
The third data set corresponds to financial log-returns, for which we employed the classical GARCH(1,1)-model with parameters chosen based on a fit to the SP500-stock market index for the time span 1995 - 2024 calculated using the R-package \texttt{rugarch} \citep{rug24}; the respective parameters are $\mu \approx 6.92\cdot 10^{-4}, \omega \approx 2.16\cdot 10^{-6}, \alpha \approx 1.16\cdot 10^{-1}, \beta \approx 8.69 \cdot 10^{-1}$.

or all three data sets, the target parameter is the $100$-year return level, which was estimated using $m\in \{40, \dots, 100\}$ years of observations for both temperature and precipitation and $m \in \{20, \dots, 80\}$ for log-returns. Here, a year (and hence the block size) was chosen as $r=250$ for the financial data set (there are 250 trading days in a year), as $r=365$ for the precipitation data set, and as $r=92$ for the temperature data sets (there are 92 summer days in a year; we refrained from using a year here due to the substantial seasonality of temperature). The unknown true parameter was assessed in a preliminary Monte Carlo experiment using $10^6$ identically distributed `years' of simulated observations.
For completeness, we also state respective estimates for the shape parameter~$\gamma$:  0.220 (log-returns), 0.228 (precipitation), and $-0.145$ (temperature).

\begin{figure}[t]
    \centering
    \includegraphics[width=0.90\textwidth]{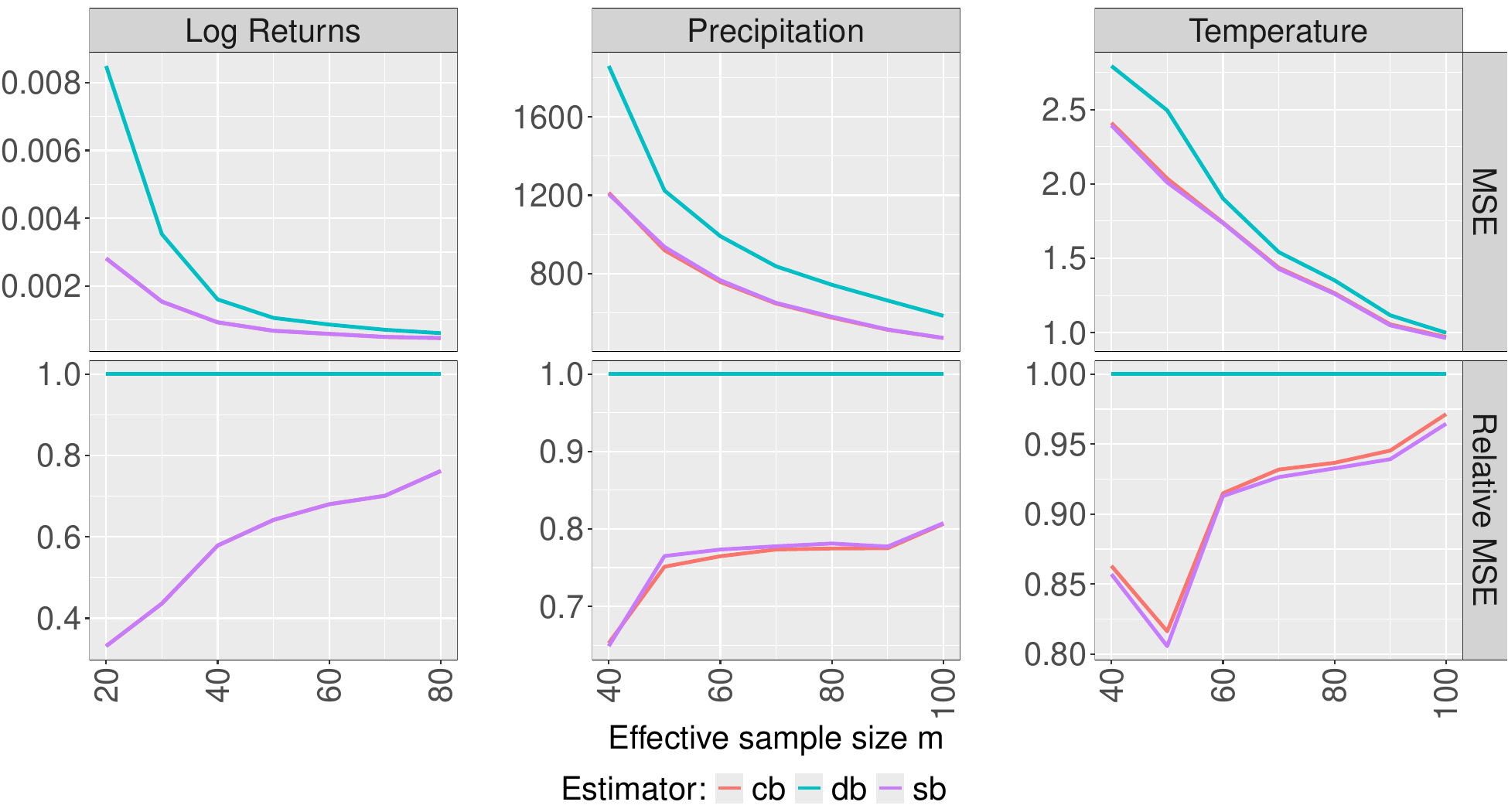}
    \caption{Performance of the different return level estimators with $T=100$ for the models from Section~\ref{subsec:sim-applied-models}. 
    Top:  
    Absolute MSE of the respective estimator and model combinations.
    Bottom: 
    Relative MSE with respect to the MSE of the disjoint estimator.
    }
    \label{fig:finPrcpTempEst}
\end{figure}

\begin{figure}[t]
    \centering
    \includegraphics[width=.90\textwidth]{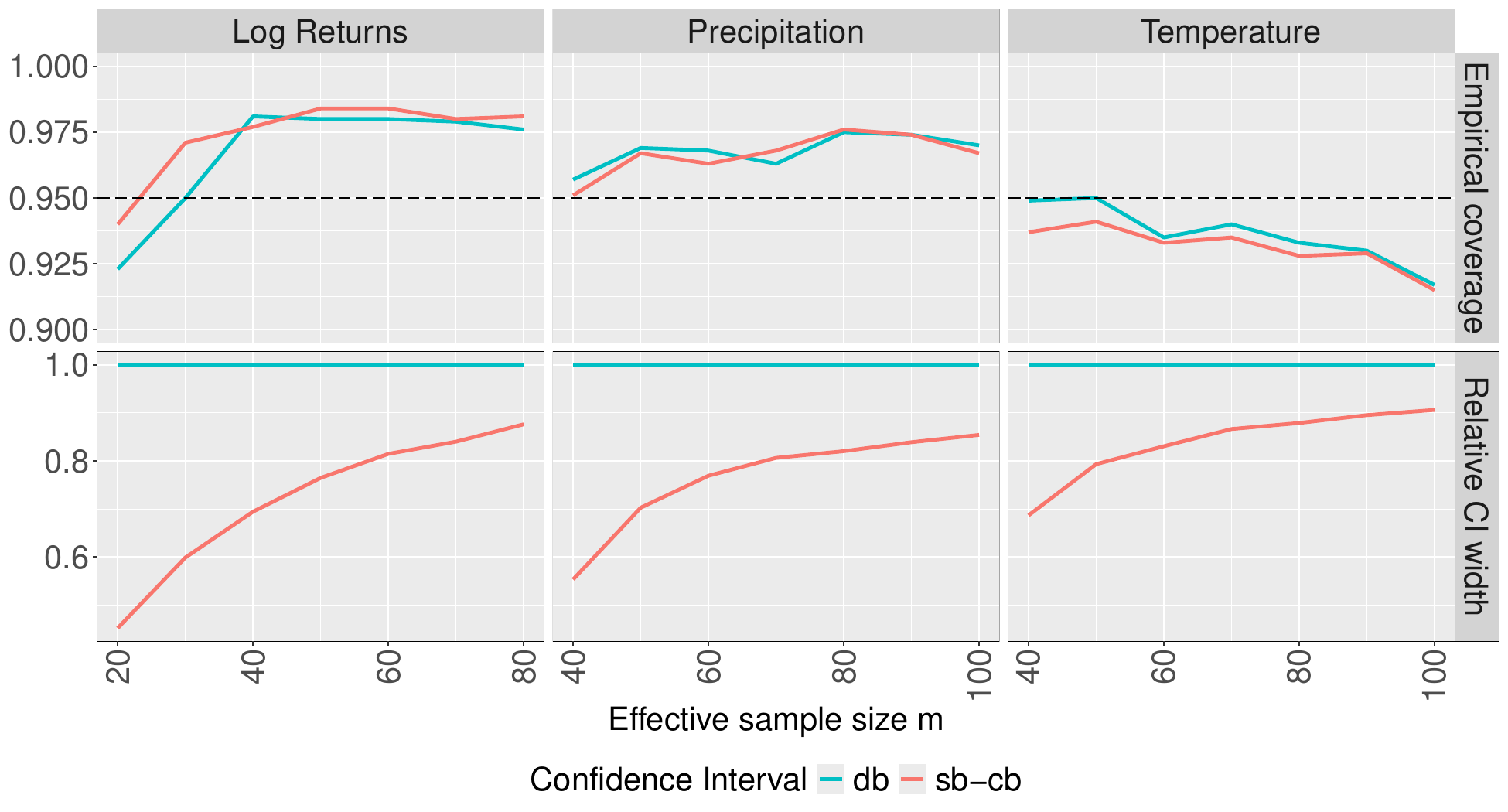}
    \caption{Data-adaptive size-corrected confidence intervals $\CI(\hat c)$ for return level estimation with $T=100$ for the models from Section~\ref{subsec:sim-applied-models}. Top: empirical coverage. Bottom: relative width $\mathrm{width}(\CI^{[\sblcbl]}(\hat c))  / \mathrm{width}(\CI^{[\dbl]}(\hat c))$. 
    }
    \label{fig:finPrcpTemp}
\end{figure}

We only report results on the performance of the estimators and the data-adaptive size-corrected confidence intervals described in the previous section, which are summarized in Figures~\ref{fig:finPrcpTempEst} and~\ref{fig:finPrcpTemp}, respectively. As in the previous section, the sliding and circmax estimators behave very similar, and they consistently outperform the disjoint version.
Regarding the performance of the confidence intervals, we observe that both approaches exhibit similar coverage close to the intended level of $95\%$, and that the sliding-circular intervals are considerably smaller than the disjoint counterparts.

\subsection{Precipitation data from a weather station}
\label{subsec:case-study-precipitation}

We consider daily accumulated precipitation amounts at a German weather station in Hohenpeißenberg, in the 145 year period from 1879 to 2023, resulting in $52,960$ daily observations in total. 
As target parameters, we consider the expected yearly maximum precipitation ($\mathrm{Rx1day}$) and the $T=100$ year return level, both of which correspond to a block size of $r=365$. To account for possible non-stationarities in the target parameters over such a long observation period, we conduct the subsequent analyses on moving windows of 40 years. More precisely, for each fixed 40-year window $w$, the target parameter may be informally written as $\mu(w) = \Exp[M_{365}(w)]$ and as $\rl(w)= F_{M_{365}(w)}^{-1}(1-1/T))$, where $M_{365}(w)$ denotes a generic annual maximum variable corresponding to the climate over the 40 years under consideration.

Each moving window $w$ contains approximately $n = 14,600$ daily observations, and the target parameters are estimated using the disjoint and sliding blocks estimator. 
Respective data-adaptive size-corrected basic bootstrap confidence intervals are constructed using the sliding-circular and the the classical disjoint blocks approach as described in the previous section. 
The resulting estimates and confidence bands are presented in Figure~\ref{fig:caseStudyCbands}, where the years on the x-axis correspond to the endpoints of the 40-year window. 
Regarding return level estimation, we observe rather wide confidence intervals for both methods, with the sliding-circular blocks confidence intervals being consistently narrower than the disjoint blocks versions, with an average relative width of about 0.68. For the mean, however, no clear benefit is apparent; the average width remains close to 1 (0.996). This observation may be linked to the relatively minor variance reduction — roughly $4\%$ — when switching from disjoint to sliding blocks estimation in settings with heavy-tailed observations (here, the average estimated shape is 0.12); see also Section~\refstar{sec:armax-mu} in the supplementary material.

\begin{figure}[t!] 
\centering
\makebox{\includegraphics[width=0.91\textwidth]{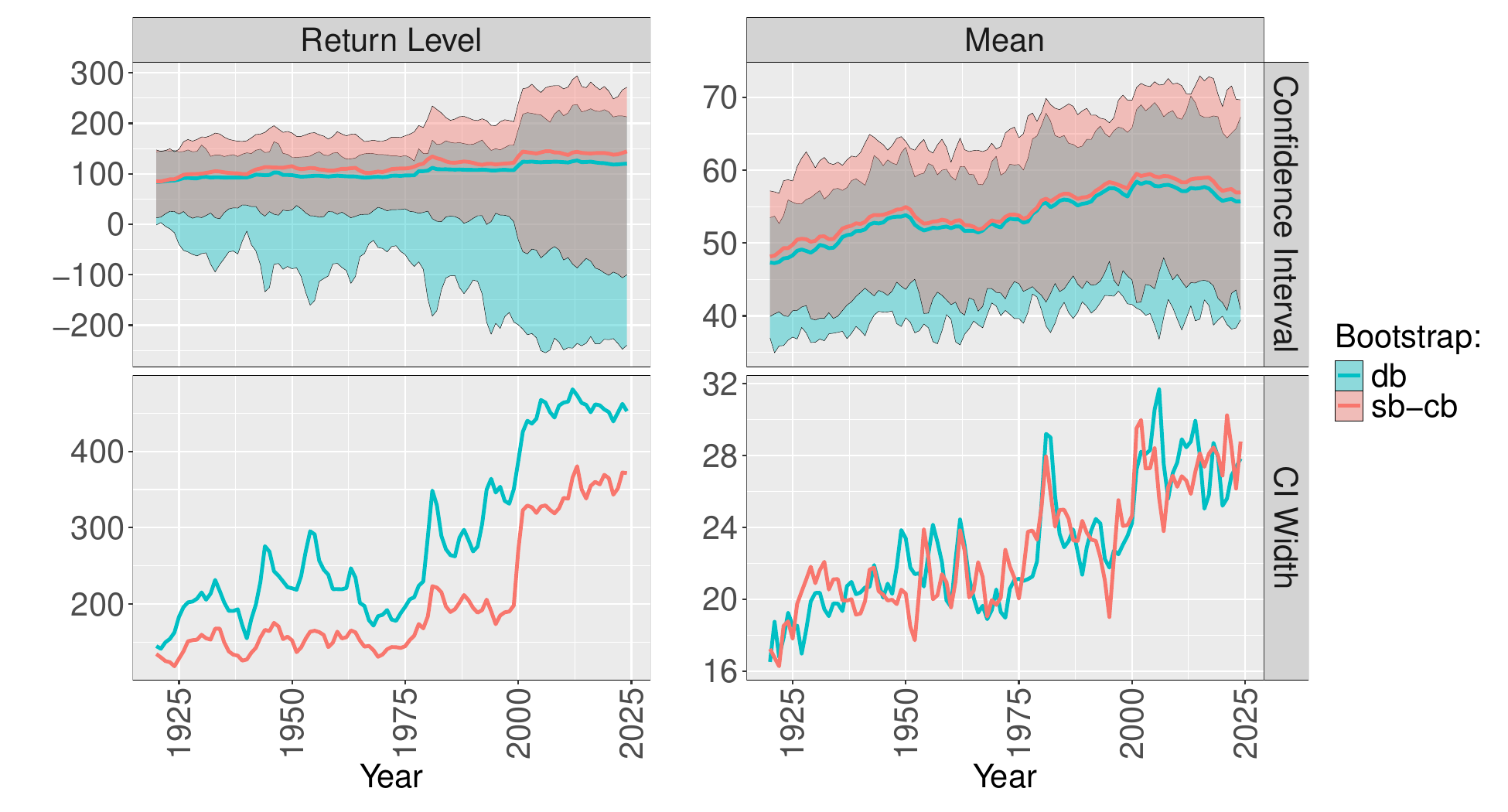}}\vspace{-.1cm}
\caption{Size-corrected confidence Intervals $\CI(\hat c)$ with intended $95\%$ coverage for the estimation of $\rl(w)$ and $\mu(w)$. First row: Widths of the respective confidence intervals. Second row: Estimations (solid lines) alongside with the db and sb-cb based confidence intervals (for a moving window of 40 years). All curves have been smoothed by a moving average filter of size 2.
}\label{fig:caseStudyCbands}	
\vspace{-.1cm}
\end{figure}

\section{Conclusion}
\label{sec:conclusion} 

Both the block-maxima method and the bootstrap are time-honored statistical methods that have seen wide use in applied statistics for extremes. Surprisingly, bootstrap consistency has never been proven, not even for the classical block-maxima method based on disjoint blocks. In this paper, respective consistency statements were established under high-level conditions on the data-generating process. A new approach, called the circular block-maxima method, has been proposed to allow for valid and computationally efficient bootstrap inference regarding the sliding block maxima method. The approach may be of independent interest for extreme-value analysis of non-stationary extremes, for instance in the presence of a temporal trend. Indeed, a reasonable model assumption inspired by many applications of the disjoint block maxima method would consist of the assumption that the circular block-maxima within a fixed $kr$-block of observations, say the $i$th one, follow an extreme-value distribution whose distribution/parameters are depending on $i$. Respective methods could be studied mathematically under a suitable triangular array structure, and we conjecture that advantages over the disjoint block maxima method will eventually show up in respective estimation variances. Other possible extensions would be the inclusion of inner-seasonal non-stationarities in the data-generating process or the generalization to empirical processes indexed by non-finite function classes. Also, a more in-depth analysis regarding a data-adapative size correction would also be an interesting possibility for future research.
%%%%%%%%%%%%%%%%%%%%%%%%%%%%%%%%%%%%%%%%%%%%%%%%%%%%%%%%%%%%%%%%%%%%%%%%%%%%%%%%%%

\section{Appendix: Integrability and bias conditions}
\label{sec:conditions}

In order to integrate $\bm h(\bm Z_{r,i})$ to the limit, we need mild asymptotic integrability conditions. For $\mbl \in \{\dbl, \sbl \}$ the condition is standard. As the circmax-sample permutates the underlying time series on a non-negligible part of the sample, we need a more involved assumption on asymptotic integrability. In most cases both conditions can be verified by similar arguments. Note that (b) implies (a) by letting $t=r$ in the supremum. Moreover, for $k=1$ (disjoint blocks case), (a) and (b) are equivalent.

\begin{condition}[Asymptotic Integrability]
\label{cond:int_h}
    Fix $k \in \N.$ Let $\bm h = (h_1, \ldots, h_q)^\top\colon \R^d \to \R^q$ be an a.e. continuous function with respect to the Lebesgue-measure on $\R^d.$ There exists a $\nu > 0$  such that:
    \begin{compactenum}
        \item[(a)] $\limsup_{r \to \infty} \Exp[\| \bm h(\bm Z_{r,1})\|^{2+\nu}] < \infty.$
        \item[(b)] $\limsup_{r \to \infty} \sup_{t=1,\ldots, r} \Exp[\| \bm h\big( (\max \{\bm X_1, \ldots, \bm X_t, \bm X_{(k-1)r+t+1}, \ldots, \bm X_{kr} \} - \bm b_r)/\bm a_r \big)\|^{2+\nu}] < \infty.$
    \end{compactenum}
\end{condition}

The next two conditions specify the asymptotic bias resulting from the approximation of the distribution of the various block maxima of size $r_n$ by the extreme value distribution~$G$.

\begin{condition}[Disjoint and sliding blocks bias]\label{cond:bias_dbsl}
Let $r=r_n\to \infty$ with $r_n=o(n)$ and let $\bm h = (h_1, \ldots, h_q)^\top\colon \R^d \to \R^q$ be measurable such that $h(\bm Z_{r,1})$ and $h(\bm Z)$ are integrable. The following limit exists:
\begin{equation}\label{eq:cond_bias_dbsl}
    \bm B_{\bm h} =  \lim_{n \to \infty}  
    \sqrt{\frac{n}{r}} \big\{ \Exp\big[\bm h(\bm Z_{r,1})] - \Exp\big[\bm h( \bm Z) \big] \big\}.
\end{equation}
\end{condition}

Circularization as applied within the cirmax-sample changes the order of observations used to calculate some of the cirular block maxima. The next conditions captures the resulting bias by decomposing it into two pieces, the first of which is due to a coupling with an independent copy $(\tilde{\bm X}_t)_t$ of $(\bm X_t)_t.$

\begin{condition}
\label{cond:bias_cbl} 
Let $r=r_n\to \infty$ with $r_n=o(n)$ and let $\bm h = (h_1, \ldots, h_q)^\top\colon \R^d \to \R^q$ be measurable and $k \in \N_{\geq 2}$ be fixed. The following expectations and limits exist:
\begin{alignat}{3}\label{eq:cond_bias_cbl}
\begin{split}
    \bm D_{\bm h,k} =&  \lim_{n \to \infty} \sqrt{\frac{n}{r}} \cdot \frac{1}{r}\sum_{t=1}^{r-1} \Exp\Big[\bm h \Big( \frac{\max (\bm X_1, \dots, \bm X_t, \bm X_{{(k-1)r+t+1}},  \dots, \bm X_{kr}) -  \bm b_r}{ \bm a_r} \Big) \\ 
    & \hspace{4.5 cm }- \bm h \Big( \frac{\max (\bm X_1, \dots, \bm X_t, \tilde{\bm X}_{1},  \dots, \tilde{\bm X}_{r-t}) - \bm b_r}{ \bm a_r} \Big) \Big], \\
    \bm E_{\bm h} = & \lim_{n \to \infty}  \sqrt{\frac{n}{r}} \cdot \frac{1}{r} \sum_{t=1}^{r-1} \Exp\Big[ \bm h \Big( \frac{\max (\bm X_1, \dots, \bm X_t, \tilde{\bm X}_{1},  \dots, \tilde{\bm X}_{r-t}) - \bm b_r}{ \bm a_r} \Big) - \bm h\big(\bm Z_{r,1}\big) \Big].
\end{split}
\end{alignat}
\end{condition}

\begin{remark}
\label{rem:bias_dhk}
In Theorem~\ref{theo:cltblocks2}, the asymptotic bias of the $\tilde{\Gb}_n^{[\cblk]}$-process, $D_{\bm h,k} + E_{\bm h}$, is due to the fact that the rescaled circular block maxima $\max (\bm X_{(k-1)r + t + 1}, \ldots, \bm X_{kr}, \bm X_1, \ldots, \bm X_{t})$ with index $s =(k-1)r + 1 +t$ for some $t\in \{1, \dots, r-1\}$ are not equal in distribution to a generic block maximum variable $\bm M_r$.
This bias has two sources that are quantified in Condition~\ref{cond:bias_cbl}: $D_{\bm h,k}$ is a measure for the average (over $t$) difference induced by the approximation of $\max (\bm X_{(k-1)r + t + 1}, \ldots, \bm X_{kr}, \bm X_1, \ldots \bm X_{t})$ by $\max (\bm X_{(k-1)r + t + 1}, \ldots, \bm X_{kr}, \tilde{\bm X}_1, \ldots, \tilde{ \bm X}_{t})$ (with $(\tilde{\bm X}_t)_t$ an independent copy of $(\tilde{\bm X}_t)_t$), while $E_{\bm h}$  is a measure for the average difference induced by the approximation of $\max (\bm X_{(k-1)r + t + 1}, \ldots, \bm X_{kr}, \tilde{\bm X}_1, \ldots, \tilde{ \bm X}_{t})$ by that of a generic block maximum variable $\bm M_r$. Note that the asymptotic bias $\bm E_{\bm h}$ already appeared in \cite{BucZan23} in a slightly modified form  in the setting of observing piecewise-stationary time series.
In Lemma \refstar{lem:Eh_zero} of the supplement we provide simple sufficient conditions for beta mixing time series that imply $D_{\bm h,k} = E_{\bm h} = 0$.
\end{remark}

\section*{Acknowledgements} 
The authors are grateful to two unknown referees and an associate editor for their constructive comments on an earlier version of this paper.
Financial support by the German Research Foundation (DFG grant number 465665892) and by Ruhr University Research School (funded by Germany’s Excellence Initiative - DFG GSC 98/3) is gratefully acknowledged.
Computational infrastructure and support were provided by the Centre for Information and Media Technology at Heinrich Heine University Düsseldorf.
The authors are grateful to Johan Segers for fruitful discussions on the circular block maxima method, to the participants of the Oberwolfach Workshop on ``Mathematics, Statistics, and Geometry of Extreme Events in High Dimensions'' for their valuable comments, and to Viet Dung Nguyen from GFZ Helmholtz Centre for Geosciences Potsdam for providing us with the simulated data set on daily precipitation amounts.

\putbib
\end{bibunit}
\newpage

\begin{bibunit}

\newpage

\begin{center}
	{\huge\textbf{Supplement to the paper:}\\
\textbf{``Bootstrapping Estimators based on the Block Maxima Method''}}

	\vspace{.5cm}
	
	{{\Large Axel B\"ucher and Torben Staud}} \\
    \vspace{.28cm}
    \textit{Ruhr-Universität Bochum, Fakultät für Mathematik}
	
	\vspace{.48cm}

\end{center}

\begin{center}
\textbf{Abstract}
\vspace{.4cm}

\begin{minipage}{.89\textwidth}
\begin{compactitem}
    \item In Section~\ref{sec:frechet}, we provide a detailed treatment of the pseudo-maximum likelihood estimator for the Fréchet distribution from Section \ref*{sec:frechet-small} in the main paper. 
    \item Section \ref{sec:proofs} contains all proofs for the main paper and for Section~\ref{sec:frechet}, with some auxiliary results postponed to Section~\ref{sec:auxres}. 
    \item In Section~\ref{sec:inconsistency-sliding}, we provide an extended discussion of the inconsistency of the naive sliding bootstrap, including heuristic arguments why the circular bootstrap solves the problem.
    \item A refined result on the asymptotic normality of mean estimators in the ARMAX-GPD-Model is stated and proven in Section~\ref{sec:armax-mu}.
    \item Finally some additional simulation results are presented in Section~\ref{sec:sim-add}.
    \item All numberings in the main article are preceded by numbers (e.g. Theorem 3.1), while the numberings in the appendix are always preceded by letters (e.g. Condition A.1).
\end{compactitem}
 \end{minipage}
 \begin{minipage}{.06\textwidth}~
 \end{minipage}
\end{center}
\vspace{.3cm}

% \clearpage

% \tableofcontents

\appendix
\section{Application: bootstrapping the pseudo-maximum likelihood estimator for the Fr\'echet distribution}
\label{sec:frechet}

This section is an extended version of Section~\refstar{sec:frechet-small} from the main paper. To make reading easier, the section is self-contained, with some notations from the main article being reintroduced. For instance, the following condition is Condition~\refstar{cond:mda_Fre-small} from the main paper.

\begin{condition}[Fréchet Max-Domain of Attraction]\label{cond:mda_Fre}
Let $(X_t)_{t\in\Z}$ denote a strictly stationary univariate time series with continuous margins. There exists some $\alpha_0 > 0$ and some sequence $(\sigma_r)_{r \in \N} \subset (0,\infty)$ such that
\[
\lim_{r \to \infty} \frac{\sigma_{\lfloor rs \rfloor }}{ \sigma_r} = s^{1/\alpha_0}, \,\, (s > 0)
\quad \text{ and } \quad
\frac{\max(X_1, \ldots, X_r)}{\sigma_r}  \wconv Z \sim P_{\alpha_0,1} \qquad (r \to \infty),
\]
where $P_{\alpha, \sigma}$ denotes the Fréchet-scale family on $(0,\infty)$ defined by its CDF $F_{\alpha,\sigma}(x) = \exp(-(x/\sigma)^{-\alpha})$; here $(\alpha, \sigma) \in (0,\infty)^2$.
\end{condition}

Suppose $X_1, \dots, X_n$ is an observed sample from $(X_t)_{t\in\Z}$ as in Condition~\ref{cond:mda_Fre}, and let $r=r_n$ denote a block length parameter.
In view of the heuristics and results from Sections~\refstar{sec:mathpre} and \refstar{sec:circmax}, the associated block maxima samples $\Mc_{n,r}^{[\mbl]}$ with $\mbl \in \{\dbl, \sbl, \cblk\}$ (with $k \in \N$ fixed) can all be considered approximate samples from $P_{\alpha_0, \sigma_r}$. As in \cite{BucSeg18a, BucSeg18-sl}, this suggests to estimate $(\alpha_0, \sigma_r)$ by maximizing the independence Fréchet-log-likelihood, that is, we define
\begin{align}
\label{eq:mle-frechet}
\hat {\bm \theta}_n^{[\mbl]} := (\hat \alpha_n^{[\mbl]}, \hat \sigma_n^{[\mbl]})^\top := \argmax_{{\bm \theta}= (\alpha, \sigma) \in (0,\infty)^2}
\sum_{M_i \in \Mc_{n,r}^{[\mbl]}}  \ell_{\bm \theta} (M_i \vee c),
\end{align}
where, for ${\bm \theta}=(\alpha, \sigma)^\top \in (0,\infty)^2$, 
\begin{align} 
\label{eq:frechetloglik}
\ell_{\bm \theta}(x) = \log(\alpha/\sigma) - (x/\sigma)^{-\alpha} - (\alpha+1) \log(x/\sigma), \quad x>0,
\end{align}
denotes the log density of the Fréchet distribution $P_{(\alpha, \sigma)}$ and where $c>0$ denotes an arbitrary truncation constant.
For the case $\mbl \in \{\dbl, \sbl\}$, the rescaled estimation error of $\hat {\bm \theta}_n^{[\mbl]}$ is known to be asymptotically normal (and independent of $c$), that is, under suitable additional assumptions, 
\begin{align} \label{eq:frechetml}
\sqrt{\frac{n}r}
\begin{pmatrix} 
\hat \alpha_n^{[\mbl]} - \alpha_0 \\
\hat \sigma_n^{[\mbl]}/\sigma_{r_n} -1
\end{pmatrix}
\wconv \mathcal N_2(\bm \mu, \Sigma^{[\mbl]}),
\end{align}
for some $\bm \mu\in \R^2$ and some $\Sigma^{[\mbl]} \in \R^{2 \times 2}$ positive definite, see \cite{BucSeg18a, BucSeg18-sl}. Subsequently, we will extend these results to $\mbl=\cblk$, and we will derive a consistent bootstrap scheme for the estimation error.

For fixed $k\in\N$, consider the circmax-sample $\mathcal M_{n,r}^{[\cblk]}$. Independent of the observations, let $\bm W_{m(k)}=(W_{m(k),1}, \dots, W_{m(k), m(k)}) = (W_{1}, \dots, W_{m(k)})$ be multinomially distributed with $m(k)$ trials and class probabilities $(m(k)^{-1}, \dots, m(k)^{-1})$. As in Section~\refstar{sec:bootgen}, conditional on $\mathcal M_{n,r}^{[\cblk]}$, the bootstrap sample $\mathcal M_{n,r}^{[\cblk],*}$ is obtained by repeating the observations $(M_{r,s}^{[\cblk]})_{s \in I_{kr,i}}$ from the $i$th $kr$-block exactly $W_{m(k),i}$-times, for every $i=1, \dots, m(k)$. We are going to show that the independence Fréchet-log-likelihood 
\[
{\bm \theta} \mapsto 
\sum_{s=1}^n \ell_{\bm \theta}(M_{r,s}^{[\cblk],*} \vee c) =
\sum_{i=1}^{m(k)} W_{m(k), i} \sum_{s \in I_{kr,i}} \ell_{\bm \theta}(M_{r,s}^{[\cblk]} \vee c).
\]
has a unique-maximimizer (with probability converging to one), say $(\hat \alpha_n^{[\cblk],*}, \hat \sigma_n^{[\cblk],*})$, and that the conditional distribution of the rescaled bootstrap-estimation error, $\sqrt{n/r}(\hat \alpha_n^{[\cblk],*} - \hat \alpha_n^{[\cblk]}, \hat \sigma_n^{[\cblk],*} / \hat \sigma_n - 1)$, given the observations is close to the distribution of $\sqrt{n/r}(\hat \alpha_n^{[\mbl]} - \alpha_0, \hat \sigma_n^{[\mbl]} / \sigma_{r_n} - 1)$ for both $\mbl=\sbl$ and $\mbl=\cblk$. We also show how this result can be used to derive valid asymptotic confidence intervals.

A couple of conditions akin to those imposed in \cite{BucSeg18a, BucSeg18-sl} will be needed.

\begin{condition}[All disjoint block maxima of size $\lfloor r_n/2 \rfloor$ diverge]\label{cond:div_Fre}
For all $c > 0$, the probability of the event that all disjoint block maxima of block size $\tilde r_n=\lfloor r_n/2 \rfloor$ are larger than $c$ converges to one.
\end{condition}

This is Condition~2.2 in \cite{BucSeg18-sl}, and guarantees that the probability of the event that all (blocksize $r_n$) block maxima in $\Mc_{n,r}^{[\mbl]}$  are larger than $c$ converges to one as well, for $\mbl \in \{\dbl, \sbl, \cblk\}$. This will guarantee that the Fréchet-log-likelihoods under consideration are well-defined, with probability converging to one. The following condition is Condition 3.4 in \cite{BucSeg18a}; recall that $M_{r,1}=\max(X_1, \dots, X_r)$.

\begin{condition}[Moments]\label{cond:mom_Fre}
There exists $\nu > 2/\omega$ with $\omega$ from Condition \refstar{cond:ser_dep} such that 
\[
\limsup_{r \to \infty} \Exp[g_{\nu, \alpha_0}\big( (M_{r,1} \vee 1)/\sigma_r \big)] < \infty,
\]
where $g_{\nu, \alpha_0}(x) := \{x^{-\alpha_0} \bm1 \{x \leq e\}+ \log(x) \bm 1\{x > e\}\}^{2+\nu}$.
\end{condition}

\begin{condition}[Bias] \label{cond:bias_Fre}
There exists $c_0 > 0$ such that, for $j\in\{1,2,3\}$, the limits $C(f_j) := \lim_{n \to \infty} C_n(f_j)$ and $C_k(f_j) := \lim_{n \to \infty} C_{n,k}(f_j)$ exist, where
\begin{align*}
C_n(f_j) &:= \sqrt{\frac{n}r} \left( \Exp\big[f_j\big( (M_{r,1} \vee c_0) /\sigma_r \big)\big] - P_{\alpha_0,1}f_j \right) \\
C_{n,k}(f_j) &:= \sqrt{\frac{n}r} \Big( \frac{1}{r} \sum_{s=(k-1)r+1}^{kr}\big\{ \Exp\big[f_j\big((M_{r,s} \vee c_0)/\sigma_r \big)\big] - \Exp\big[f_j\big( (M_{r,1} \vee c_0) /\sigma_r \big)\big] \big\} \Big)
\end{align*}
where $f_j:(0,\infty)\to \R$ are defined as
\begin{equation} \label{eq:Fref_i}
    f_1(x) = x^{-\alpha_0}, \quad 
    f_2(x) = x^{-\alpha_0} \log x, \quad
    f_3(x) =\log x.     
\end{equation}
\end{condition}

The first part of this condition on $C_n$ is Condition 3.5 in \cite{BucSeg18a}, and provides control of the bias for those block maxima which are calculated based on $r$ successive observations. For the circmax-sample, the last $r$ observations in a $kr$-block of that sample are not of that form; their contribution to the bias is controlled by the condition on $C_{n,k}$.

Subsequently, we fix a truncation constant $c > 0$, and we write
\[
\Gb_{n,r}^{[\cblk]} = \sqrt{\frac{n}r} \big(\Prob_{n,r}^{[\cblk]} - P_{(\alpha_0,1)}\big),
\quad
\hat \Gb_n^{[\cblk],*} = \sqrt{\frac{n}{r}}( \hat \Prob_{n,r}^{[\cblk], \ast} - \Prob_{n,r}^{[\cblk]})
\]
where 
\[
\Prob_{n,r}^{[\cblk]} = \frac{1}{n} \sum_{s =1}^n \delta_{(M_{r,s}^{[\cblk]} \vee c) / \sigma_r},
\qquad
\hat \Prob_{n,r}^{[\cblk],*} = \frac{1}{n} \sum_{i=1}^{m(k)} W_{m(k), i} \sum_{s \in I_{kr,i}}  \delta_{(M_{r,s}^{[\cblk]} \vee c) / \sigma_r};
\]
the double use of notation with (\refstar{eq:gn}) and (\refstar{eq:gnboot}) should not cause any confusion. Furthermore, we suppress the dependence on $c$ in the notation, which is motivated by the fact that the limiting distribution does not depend on $c$, as shown in the following theorem.

\begin{theorem}
\label{theo:boot-mle-blockmax}
Fix $k \in \N_{\ge 2}$ and $c>0$, and suppose that Conditions \refstar{cond:ser_dep} and \ref{cond:mda_Fre}--\ref{cond:bias_Fre} are satisfied. Then, with probability tending to one, there exists a unique maximizer $(\hat \alpha_n^{[\cblk]}, \hat \sigma_n^{[\cblk]})$ of the Fréchet log-likelihood ${\bm \theta} \mapsto \sum_{i=1}^n \ell_{\bm \theta}(M_{r,s}^{[\cblk]} \vee c)$, and this maximizer satisfies
\begin{align*}
\sqrt{\frac{n}r}
        \begin{pmatrix}
            \hat \alpha_n^{[\cblk]} - \alpha_0 \\
            \hat \sigma_n^{[\cblk]}/\sigma_{r_n} - 1
        \end{pmatrix}
        =
M(\alpha_0) \Gb_{n,r}^{[\cblk]} (f_1, f_2, f_3)^\top + o_\Prob(1)
\wconv 
\Nc_2(M(\alpha_0) (\bm C + k^{-1}\bm C_k), \Sigma^{[\sbl]}),
\end{align*}
where $\bm C=(C(f_1), C(f_2), C(f_3))^\top, \bm C_k=(C_k(f_1), C_k(f_2), C_k(f_3))^\top$ and 
\[
M(\alpha_0) = \frac{6}{\pi^2} \begin{pmatrix}
\alpha_0^2 & \alpha_0(1-\gamma) & - \alpha_0^2 \\
\gamma-1 & -(\Gamma''(2)+1)/\alpha_0 & 1-\gamma \end{pmatrix},
\quad
 \Sigma^{[\sbl]} = 
\begin{pmatrix}
0.4946 \alpha_0^2 & -0.3236 \\
-0.3236 & 0.9578 \alpha_0^{-2}
\end{pmatrix}
\]
with
$\Gamma(z) = \int_0^\infty t^{z-1} {e}^{-t} \diff t$ the Euler Gamma function and $\gamma = 0.5772\ldots$ the Euler-Mascheroni constant.

Moreover,  also with probability tending to one, there exists a unique maximizer $(\hat \alpha_n^{[\cblk],*}, \hat \sigma_n^{[\cblk],*})$ of ${\bm \theta} \mapsto \sum_{s=1}^n \ell_{\bm \theta}(M_{r,s}^{[\cblk],*} \vee c)$, and this maximizer satisfies
\begin{align*}
\sqrt{\frac{n}r}
        \begin{pmatrix}
            \hat \alpha_n^{[\cblk],*} - \hat \alpha_n^{[\cblk]} \\
            \hat \sigma_n^{[\cblk],*}/ \hat \sigma_{n}^{[\cblk]} - 1
        \end{pmatrix}
        =
M(\alpha_0) \hat \Gb_{n,r}^{[\cblk],*} (f_1, f_2, f_3)^\top + o_\Prob(1)
\wconv 
\Nc_2(0, \Sigma^{[\sbl]}).
\end{align*}
Fix $\mbl \in \{ \sbl, \cblk\}$. If $r_n$ is chosen sufficiently large such that, for $j\in\{1,2,3\}$, $C(f_j)=0$ ($\mbl=\sbl$) or $C(f_j)=C_k(f_j)=0$ ($\mbl=\cblk$) in Condition~\ref{cond:bias_Fre}, we have bootstrap consistency in the following sense:
\begin{align*}
d_K\bigg[ 
\Lc\bigg( 
\sqrt{\frac{n}r}
        \begin{pmatrix}
            \hat \alpha_n^{[\cblk],*} - \hat \alpha_n^{[\cblk]} \\
            \hat \sigma_n^{[\cblk],*}/\hat \sigma_n^{[\cblk]} - 1
        \end{pmatrix} \,\Big|\, \mathcal X_n\bigg) ,
\sqrt{\frac{n}r}
        \begin{pmatrix}
            \hat \alpha_n^{[\mbl]} - \alpha_0 \\
            \hat \sigma_n^{[\mbl]}/\sigma_{r_n} - 1
        \end{pmatrix}
        \bigg] = o_\Prob(1).
\end{align*}
\end{theorem}

The result in Theorem~\ref{theo:boot-mle-blockmax} allows for statistical inference on the parameters $\alpha_0, \sigma_r$, for instance in the form of confidence intervals. We provide details on $\sigma_r$, using the circular block bootstrap approximation to the sliding block estimator: for $\beta \in (0,1)$, let 
$q_{\hat \sigma_n^{[\cblk],*}}(\beta)$ denote the $\beta$-quantile of the conditional distribution of 
$\hat \sigma_n^{[\cblk],*}$ given the data, 
that is,
$q_{\hat \sigma_n^{[\cblk],*}}(\beta)= (F_{\hat \sigma_n^{[\cblk],*}})^{-1}(\beta)$, where
$
F_{\hat \sigma_n^{[\cblk],*}}(x) = \Prob( \hat \sigma_n^{[\cblk],*} \le x \mid \Xc_n)$ for $x \in \R$.
Note that the quantile may be approximated to an arbitrary precision by repeated bootstrap sampling. Consider the basic bootstrap confidence interval \citep{DavHin97} 
\begin{align*}
I_{n,\sigma}^{[\sbl, \cblk]}(1-\beta) &= 
\big[\hat \sigma_n^{[\sbl]} + \hat \sigma_n^{[\cblk]} - q_{\hat \sigma_n^{[\cblk],*}}(1-\tfrac\beta2),\hat \sigma_n^{[\sbl]} + \hat \sigma_n^{[\cblk]}- q_{\hat \sigma_n^{[\cblk],*}}(\tfrac\beta2) \big].
\end{align*} \vspace{-.4cm}

\begin{corollary}
\label{cor:confint-fremle}
Under the conditions of Theorem~\ref{theo:boot-mle-blockmax} with $r_n$ chosen sufficiently large such that, for $j\in\{1,2,3\}$, $C(f_j) = 0$ in Condition~\ref{cond:bias_Fre}, we have, for any $\beta\in(0,1)$,
\[
\lim_{n \to \infty} \Prob( \sigma_{r_n} \in I_{n,\sigma}^{[\sbl, \cblk]}(1-\beta)) = 1-\beta.
\]
\end{corollary}

The above result only concerns the sliding block maxima estimator, but an analogous result can be derived for the disjoint block maxima estimator and under additional bias conditions for the circmax estimator. In view of the fact that the disjoint block maxima estimator exhibits a larger asymptotic estimation variance, the width of the disjoint blocks confidence interval is typically larger than the one of the circmax method, for every fixed confidence level (for an empirical illustration, see Figure~\refstar{fig:plotFreFixNCi} for the related problem of providing a confidence interval for the shape parameter $\gamma$).

%%%%%%%%%%%%%%%%%%%%%%%%%%%%%%%%%%%%%%%%%%%%%%%%%%%%%%%%%%%%%%%%%%%%%%%%%%%%%%%%%%

\section{Proofs}
\label{sec:proofs}

\begin{proof}[Proof of Theorem~\refstar{theo:cltblocks1}]
The proof is a simplified version of our proof of Theorem~\ref{theo:cltblocks2} below, whence we omit further details. Under slightly different conditions, a proof can be found in Theorem B.1 in \cite{BucZan23} or Theorem 8.7 in \cite{BucSta24}. 
\end{proof}

\begin{proof}[Proof of Theorem~\refstar{theo:cltblocks2}]
Throughout, we omit the upper index $[\cblk].$ We start by showing that
\begin{equation}
\label{eq:proof_kloop_sliding_clt_4}
    \lim_{n \to \infty} \Cov(\bar \Gb_{n,r} \bm h ) = \Sigma_{\bm h}^{[\sbl]}.
\end{equation}
By elementary arguments, it is sufficient to consider the case $q=1$. For $i \in \{1, \dots, m(k)\}$ write
\begin{align} 
\label{eq:dni}
D_{n,i} = \frac{1}{kr} \sum_{s \in I_{kr,i}} \big\{ h(\bm Z_{r,s}) - \Exp[h(\bm Z_{r,s})] \big\},
\end{align}
such that $\bar \Gb_{n,r} h_j= \sqrt{\frac{n}r} m(k)^{-1} \sum_{i=1}^{m(k)} D_{n,i}$.
By stationarity and $n/r=km(k)$, we have
\begin{equation}\label{eq:proof_kloop_sliding_clt_1}
    \Var\big(\bar{\mb G}_{n,r}h \big)
    = k \Var(D_{n,1})+ r_{n1} + r_{n2},
\end{equation}
where
$
    r_{n1}  = 2k ( 1-\frac{1}{m(k)} )  \Cov(D_{n,1}, D_{n,2})
$ and $
    r_{n2} = 2k \sum_{d=2}^{m(k)-1} (1- \frac{d}{m(k)}) \Cov(D_{n,1}, D_{n,1+d}).
$
Thus, the proof of \eqref{eq:proof_kloop_sliding_clt_4} is finished once we show that 
\begin{align}\label{eq:proof_kloop_sliding_clt_3}
    \lim_{n\to\infty} k \Var(D_{n,1}) = \Sigma_h^{[\sbl]}, \quad \lim_{n \to \infty} r_{n1}=0, \quad\lim_{n \to \infty} r_{n2}=0.
\end{align}
We start with the former, and for that purpose we define, for $\xi, \xi'\in[0,k)$,
\begin{align*}
    f_{r}(\xi, \xi^\prime) &:= \Cov\big(h( \bm Z_{r, 1+ \lfloor r\xi \rfloor} ) , h(\bm Z_{r, 1+ \lfloor r\xi^\prime \rfloor}) \big),\\
    g_{(k)}(\xi, \xi^\prime) & := \Cov\big(h(\bm Z^{(1)}_{|\xi - \xi^\prime|, (k)}), h(\bm Z^{(2)}_{|\xi - \xi^\prime|, (k)})\big), \\
    g(\xi, \xi^\prime) &:= \Cov\big(h(\bm Z^{(1)}_{|\xi - \xi^\prime|}), h(\bm Z^{(2)}_{|\xi - \xi^\prime|})\big),
\end{align*}
where $(\bm Z^{(1)}_{|\xi - \xi^\prime|, (k)}, \bm Z^{(2)}_{|\xi - \xi^\prime|, (k)})$ has CDF $ G_{|\xi - \xi^\prime|}^{(k)}$ from \eqref{eq:G_xi_kloop} and where $(\bm Z^{(1)}_{|\xi - \xi^\prime|}, \bm Z^{(2)}_{|\xi - \xi^\prime|})$ has CDF $G_{|\xi-\xi^\prime|}$ from (\refstar{eq:Gxi-prod}).  Observe that, by Proposition \ref{prop:overlap_wconv_kloop}, Condition \refstar{cond:int_h} and Example 2.21 in \cite{Van98}, $\lim_{n\to\infty} f_{r}(\xi, \xi^\prime) = g_{(k)}(\xi, \xi^\prime)$. Hence, by Condition \refstar{cond:int_h} and Dominated Convergence
\begin{align}\label{eq:proof_kloop_sliding_clt_2} 
         k \Var(D_{n,1})
        =
        \frac{1}{kr^2}\sum_{s=1}^{kr}\sum_{t=1}^{kr} \Cov\big[h(\bm Z_{r,s}), h(\bm Z_{r,t})\big] 
        &\nonumber = \frac{1}{k} \int_0^{k}\int_0^{k} f_{r}(\xi, \xi^\prime) \diff \xi^\prime \diff \xi \\
        & = \frac{1}{k} \int_0^{k}\int_0^{k} g_{(k)}(\xi, \xi^\prime) \diff \xi^\prime \diff \xi  + o(1).
\end{align}
Now let $u(x) := g(x,0)$ and note that $g(\xi,\xi')=u(|\xi-\xi^\prime|)$  and that $u(x)=0$ for $x > 1$. By the definition of $G_{|\xi-\xi^\prime|}^{(k)}$ in \eqref{eq:G_xi_kloop} we have $G_{|\xi-\xi^\prime|}^{(k)} = G_{|\xi-\xi^\prime|}$ if $0 \le |\xi-\xi^\prime| \le k-1$, which also implies that $g_{(k)}(\xi, \xi^\prime) = g_{(k)}(|\xi-\xi^\prime|,0)=g(|\xi-\xi^\prime|,0)=u(|\xi-\xi^\prime|)$ for $\xi \le k-1.$ Therefore,
\begin{align} \label{eq:intdec2}
\int_0^{k}\int_\xi^{k} g_{(k)}(\xi, \xi^\prime) \diff \xi^\prime \diff \xi  
    &=\nonumber \int_0^{k}\int_{\xi}^{k \wedge (\xi + k-1)}u(\xi^\prime-\xi) \diff \xi^\prime \diff \xi 
    + \int_0^{1}\int_{\xi + k-1}^k g_{(k)}(\xi^\prime-\xi,0) \diff \xi^\prime \diff \xi \\
    &=\int_0^{k}\int_{\xi}^{k \wedge (\xi+1)} u(\xi^\prime-\xi) \diff \xi^\prime \diff \xi 
    + \int_0^{1}\int_{k-1+\xi}^k g_{(k)}(\xi^\prime-\xi,0) \diff \xi^\prime \diff \xi,
\end{align}
where we used that $u(x)=0$ for $x>1$ at the second equality.
After substituting $x=\xi'-\xi$ in the inner integral and noting that $\int_0^1 u(x) \diff x = \Sigma_h^{[\sbl]} /2,$ the first integral on the right-hand side of \eqref{eq:intdec2} can be written as
\begin{align*} 
    \int_0^{k}\int_{\xi}^{k \wedge (\xi+1)} u(\xi^\prime-\xi) \diff \xi^\prime \diff \xi
    =
        \int_0^{k}\int_{0}^{(k-\xi) \wedge 1} u(x) \diff x \diff \xi 
        &=
        \int_0^{k-1}\int_{0}^{1} u(x) \diff x \diff \xi +  \int_{k-1}^k\int_{0}^{k-\xi} u(x) \diff x \diff \xi  \\
        &= (k-1) \frac{ \Sigma_h^{[\sbl]}}2 + \int_0^1 \int_0^y u(x) \diff x \diff y,
\end{align*}
where we used the substitution $y=k-\xi$ in the last step. Moreover, substituting $x=k-\xi'+\xi$ in the inner integral, the second integral on the right-hand side of \eqref{eq:intdec2} can be written as
\begin{align*} 
        \int_0^{1}\int_{k-1+\xi}^k g_{(k)}(\xi^\prime-\xi,0) \diff \xi^\prime \diff \xi
        &=
        \int_0^1 \int_\xi^1 g_{(k)}(k-x,0)\diff x \diff \xi
        =
        \int_0^1 \int_\xi^1 u(x)\diff x \diff \xi,
\end{align*} 
where we used that $g_{(k)}(k-x,0) = u(x)$ for $x \in[0,1]$ by the definition of $G_{k-x}^{(k)}$. 
Overall, again using that $\int_0^1 u(x) \diff x = \Sigma_h^{[\sbl]} /2$, the previous three displays imply that 
\[
    \int_0^{k}\int_\xi^{k} g_{(k)}(\xi, \xi^\prime) \diff \xi^\prime \diff \xi  = k\frac{ \Sigma_h^{[\sbl]}}2.
\]
This implies \eqref{eq:proof_kloop_sliding_clt_3} in view of the symmetry $\int_0^{k}\int_\xi^{k} g_{(k)}(\xi, \xi^\prime) \diff \xi^\prime \diff \xi = \int_0^{k}\int_0^\xi g_{(k)}(\xi, \xi^\prime) \diff \xi^\prime \diff \xi$ and \eqref{eq:proof_kloop_sliding_clt_2}. 

Similar arguments as before, invoking asymptotic independence of $\bm Z_{r,s}$ and $\bm Z_{r,t}$ for $s\in I_{kr,1}$ and $t\in I_{kr,2}$ (see Proposition~\ref{prop:overlap_wconv_kloop}) and a similar identification of the covariances as integrals, imply that $r_{n1}=o(1)$. Finally, by Conditions \refstar{cond:ser_dep}(c) and \refstar{cond:int_h} and Lemma 3.11 in \cite{DehPhi02} we obtain 
\begin{align*}
        \frac{|r_{n2}|}k 
        & \lesssim m(k) \alpha(kr)^{\nu/(2+\nu)}
        \lesssim \Big( \Big(\frac{n}{r}\Big)^{1+2/\nu} \alpha(kr)\Big)^{v/(2+\nu)} = o(1),
\end{align*}
since $2/\omega < \nu$,
where we have used that there is a lag of at least $kr$ between the observations making up $D_{n,1}$ and $D_{n,1+d}$. Overall, we have shown all three assertions in \eqref{eq:proof_kloop_sliding_clt_3}, and hence the proof of \eqref{eq:proof_kloop_sliding_clt_4} is finished.

To verify the asymptotic normality, by the Cramér-Wold Theorem, it is again sufficient consider the case $q=1$. Recalling that $\bar{\mb G}_{n,r} h$ is centered, the case $\Sigma_h^{[\sbl]} = 0$ follows from \eqref{eq:proof_kloop_sliding_clt_4}; thus assume $\Sigma_h^{[\sbl]} > 0.$  
Following \cite{BucSeg18-sl}, we use a suitable blocking technique.
Choose $m^\ast = m_n^\ast \in \N$ with $3 \leq m^\ast \leq m(k)$ such that $m^\ast \to \infty$ and $m^\ast = o((n/r)^{\nu/(2(1+\nu))})$ with $\nu$ from Condition~\refstar{cond:int_h}. Next, define $p := p_n := m(k)/m^\ast$ and assume for simplicity that $p \in \N, n/r \in \N$. Furthermore, let $J_j^+ := \{(j-1)m^\ast + 1, \ldots, jm^\ast-1\}, J_j^- := \{jm^\ast \}$ and note that $|J_j^+| = (m^\ast -1).$
Then,
\begin{equation}\label{eq:proof_kloop_sliding_clt_5}
        \bar{\mb G}_{n,r}  h
        = \frac{1}{\sqrt{p}} \sum_{j=1}^p ( S_{n,j}^+ + S_{n,j}^- ),
\quad \text{ where } \quad
        S_{n,j}^\pm := \sqrt{\frac{np}{r}}\frac{1}{m(k)} \sum_{i \in J_j^\pm}  D_{n,i}
\end{equation}
with $D_{n,i}$ from \eqref{eq:dni}.
Noticing that the time series $(S_{n,j}^-)_j$ is stationary, we have
\begin{align*}
        &\phantom{{}={}} \Var\Big(\frac{1}{\sqrt{p}} \sum_{j=1}^p S_{n,j}^- \Big) 
        \leq 3\Var(S_{n,1}^-) + \sum_{d=2}^{p-1} | \Cov(S_{n,1}^-, S_{n,1+d}^-) | =: r_{n3} + r_{n4}.
\end{align*}
By Conditions \refstar{cond:ser_dep}(a) and \refstar{cond:int_h} we have $r_{n3} = O(1/m^\ast) = o(1)$. Furthermore, by Condition~\refstar{cond:int_h} and Lemma~3.11 from \cite{DehPhi02} with $s=r=1/(2+\nu)$ and $t=\nu/(2+\nu)$, we have 
$\sup_{d \geq 2} | \Cov(S_{n,1}^-, S_{n,1+d}^-) | 
\lesssim 10 k(m^\ast)^{-1} \alpha^{\nu/(2+\nu)}(kr)$, 
whence $r_{n4} \lesssim r^{-1} n (m^\ast)^{-2} \alpha^{\nu/(2+\nu)}(kr) = o((m^*)^{-2})=o(1)$ by Condition \refstar{cond:ser_dep} and the fact that $2/\nu<\omega$.
Put together, in view of $\Exp[S_{n,j}^-]=0$, these convergences imply $p^{-1/2} \sum_{j=1}^p S_{n,j}^- = o_\Prob(1).$

As there is a lag of at least $kr$ between the observations making up $S_{n,j_1}^+$ and $S_{n,j_2}^+$ for $j_1 \neq j_2$ a standard argument involving characteristic functions, a complex-valued version of Lemma 3.9 in \cite{DehPhi02}, Condition \refstar{cond:ser_dep} (b) and Levy's Continuity Theorem allows for assuming that the $(S_{n,j}^+)_j$ are independent.
Hence, we may subsequently apply Lyapunov's central limit theorem. To verify its conditions, note that, since  $\Var(p^{-1/2} \sum_{j=1}^p S_{n,j}^-) = o(1)$ and by using \eqref{eq:proof_kloop_sliding_clt_4} and \eqref{eq:proof_kloop_sliding_clt_5}, 
$
        \Var( p^{-1/2}\sum_{j=1}^p  S_{n,j}^+ ) =  
         \Var(\bar{\mb G}_{n,r} h) + o(1) 
        = \Sigma_h^{[\sbl]} + o(1).
$
Hence,
\begin{equation*}
        \frac{\sum_{j=1}^p \Exp[ \big|S_{n,j}^+\big|^{2+\nu} ]}{ \{ \Var(\sum_{j=1}^p S_{n,j}^+) \}^{1+\nu/2}}
        \lesssim \frac{p(np/r)^{1+ \nu/2} (m^\ast / m(k))^{2+\nu}}{(p\sigma^2/2)^{1+\nu/2}} 
        \lesssim \frac{(m^\ast)^{1+\nu}}{(n/r)^{\nu/2}} =o(1),
\end{equation*}
by Condition \refstar{cond:int_h} and the choice of $m^\ast.$ An application of Lyapunov's central limit theorem implies the assertion. 

To prove the additional weak convergences note that 
\[
    \tilde \Gb_{n,r}^{[\cblk]} h = 
    \bar \Gb_{n,r}^{[\cblk]}  h + \sqrt{\frac{n}{r}} ( P_{n,r}^{[\cblk]} - P_r) h, \quad 
    \Gb_{n,r}^{[\cblk]} h = 
    \bar \Gb_{n,r}^{[\cblk]} h + \sqrt{\frac{n}{r}} ( P_{n,r}^{[\cblk]} - P_r + P_r - P) h.
\]
Denote by $(\tilde{ \bm X}_t)_t$ an independent copy of $(\bm X_t)_t$. By stationarity we have 
\begin{align*}
     (P_{n,r}^{[\cblk]} - P_r) h = & 
     \frac{1}{m(k)}\sum_{i=1}^{m(k)} \frac{1}{kr}\sum_{s=1}^{kr} \Exp[h(Z_{r,(i-1)kr+s})] - P_r h \\
     & = \frac{1}{kr} \sum_{s = (k-1)r + 2}^{kr} \big( \Exp[h(Z_{r,s})] - P_r h \big) \\
     & = \frac{1}{kr} \sum_{t=1}^{r-1} \Exp\Big[h \Big( \frac{\max ( X_1, \dots,  X_t, X_{{(k-1)r+t+1}},  \dots, X_{kr}) -  b_r}{a_r} \Big) -  h\Big( \frac{\bm M_{r,1} - b_r}{a_r}\Big) \Big].
\end{align*}
Condition \refstar{cond:bias_cbl} then gives $\sqrt{n/r} (P_{n,r}^{[\cblk]} - P_r) h \to  k^{-1} ( D_{h, k} + E_{h}).$ Furthermore, Condition \refstar{cond:bias_dbsl} immediately implies $\sqrt{n/r}(P_r - P)h \to B_{h}.$ Combined with the first assertion and the penultimate display, we obtain the additional two claimed weak convergences.
\end{proof}

\begin{proof}[Proof of Theorem \refstar{theo:mult_boot}] 
The second and third assertion follow immediately from the first one and the triangular inequality. For the proof of the first one, we omit the upper index $\cblk$.
By Theorem~\refstar{theo:cltblocks1} (for $\mbl=\sbl$) or Theorem~\refstar{theo:cltblocks2} (for $\mbl=\cblk$) and Lemma 2.3 from \cite{BucKoj2019note}, it is sufficient to show that, unconditionally, 
\[
(\hat \Gb_{n,r}^{[1]} \bm h,  \hat \Gb_{n,r}^{[2]} \bm h) \weak \Nc_q(\bm 0, \Sigma^{[\sbl]}_h) \otimes \Nc_q(\bm 0, \Sigma^{[\sbl]}_h),
\]
where, for $b\in\{1,2\},$ $\hat \Gb_{n,r}^{[b]} = \sqrt{\frac{n}r} \frac1n \sum_{i=1}^{m(k)} (W_{m(k),i}^{[b]}-1) \sum_{s \in I_{kr,i}} \delta_{\bm Z_{r,s}}$ with $\bm W^{[b]}_{m(k)} = (W_{m(k),1}^{[b]}, \dots, W_{m(k), m(k)}^{[b]})$ two i.i.d.\ copies of $\bm W_{m(k)}$. By the Cram\'er-Wold device, it is sufficient to consider $q=1$, and for this in turn, yet again by the Cram\'er-Wold device, it is sufficient to show that
\begin{align}\label{eq:proof_efron_boot_1}
    T_n & := \sum_{b=1}^2 a_b \sqrt{\frac{n}r} \frac1n \sum_{i=1}^{m(k)} \big( W_{m(k),i}^{[b]} -1 \big) \sum_{s \in I_{kr,i}} h(\bm Z_{r,s}) 
       \wconv \Nc_2 \big(0,  (a_1^2+a_2^2) \Sigma_h^{[\sbl]} \big)
\end{align}
for all $\bm a=(a_1, a_2)^\top \in \R^2$. Note that $\sum_{i=1}^{m(k)} \big( W_{m(k),i}^{[b]} -1 \big) =0$ and that $\sum_{s \in I_{kr,i}}\Exp[h(\bm Z_{r,s})]$ does not depend on $i$, whence we may write
\begin{align}
\label{eq:Tn-centered}
     T_n 
     &= \nonumber
     \sum_{b=1}^2 a_b \sqrt{\frac{n}r} \frac1n \sum_{i=1}^{m(k)} \big( W_{m(k),i}^{[b]} -1 \big) \Big\{ \sum_{s \in I_{kr,i}}  h(\bm Z_{r,s}) - \Exp[h(\bm Z_{r,s})] \Big\}\\
     &=
     \sum_{b=1}^2 a_b \sqrt{\frac{n}r} \frac1{m(k)} \sum_{i=1}^{m(k)} \big( W_{m(k),i}^{[b]} -1 \big) D_{n,i}
\end{align}   
with $D_{n,i}$ from \eqref{eq:dni}.
The subsequent proof strategy,  known as \textit{Poissonization},  consists of removing the dependence of the multinomial multipliers by introducing rowwise i.i.d.\ multiplier sequences that approximate the multinomial multipliers. More precisely, we employ the construction from Lemma~\ref{lem:poiss}:
for each fixed $n \in \N$, let 
$(\bm U_{j,m(k)}^{[b]})_{j \in \N, b \in \{1, 2\}}$ be i.i.d.\ multinomial vectors with 1 trial and $m(k)$ classes, with class probabilities $1/m(k)$ for each class, and independent of $(\bm X_t)_{t\in\mathbb Z}$. We may then assume that
\[
{\bm W}_{m(k)}^{[b]} 
= \sum_{j=1}^{m(k)} \bm U_{j,m(k)}^{[b]}.
\]
Further, independent of $(\bm U_{j,m(k)}^{[b]})_{j \in \N, b \in \{1, \dots, 2\}}$ and of $(\bm X_t)_{t\in\mathbb Z}$, let $(N_{m(k)}^{[b]})_{b=1,\ldots,M}$ be i.i.d.\ $\mathrm{Poisson}(m(k))$ distributed random variables. Define
\[
    \tilde{\bm W}_{m(k)}^{[b]} =
    (\tilde{W}_{m(k),1}^{[b]}, \dots, \tilde{W}_{m(k), m(k)}^{[b]}) = \sum_{j=1}^{N_{m(k)}^{[b]}} \bm U_{j,m(k)}^{[b]}.
\]
By Lemma \ref{lem:poiss} the random vectors $\tilde{\bm W}_{m(k)}^{[1]}, \tilde{\bm W}_{m(k)}^{[2]}$ are i.i.d.\ $\mathrm{Poisson}(1)^{\otimes m(k)}$ distributed. Then, in view of \eqref{eq:Tn-centered}, we may write $T_n = \tilde T_n + a_1 R_{n1} + a_2 R_{n2}$, where
\begin{align*}
    \tilde T_n 
    & = \sqrt{\frac{n}r} \frac1{m(k)} \sum_{i=1}^{m(k)} \Big\{ \sum_{b=1}^2 a_b \big( \tilde W_{m(k), i}^{[b]} -1 \big) \Big\} D_{n,i}, \\
    R_{nb} &=  \sqrt{\frac{n}r} \frac1{m(k)} \sum_{i=1}^{m(k)} \big( W_{m(k), i}^{[b]} -\tilde W_{m(k), i}^{[b]} \big) D_{n,i}.
\end{align*}
We may apply Lemma \ref{lem:aux_mult_clt} to obtain that $\tilde T_n \wconv \Nc(0, (a_1^2+a_2^2) \Sigma_h^{[\sbl]})$. For the proof of \eqref{eq:proof_efron_boot_1}, it hence remains to verify that $R_{nb} = o_\Prob(1)$ for $b\in\{1,2\}$. We will suppress the upper index $[b]$ in the following.
First write 
\[
        R_{nb} = \sqrt{\frac{n}r} \frac1{m(k)} \sum_{i=1}^{m(k)} D_{n,i} \mathrm{sgn}(N_{m(k)}-m(k)) \sum_{j=1}^\infty \bm 1(i \in I_{m(k)}^j),
\]
where $I_{m(k)}^j = \{ i \in \{1, \ldots, m(k)\} \colon |\tilde W_{m(k),i} - W_{m(k),i}| \geq j \}.$ 

We will start by showing that $\Prob(A_n) = o(1),$ where $A_n := \{|I_{m(k)}^3| > 0 \}=\{ \exists i \in \{1, \dots, m(k)\}: |\tilde W_{m(k),i} - W_{m(k),i}| \ge 3\}$. Fix $\delta > 0$. Invoking the central limit theorem, we may choose $C=C_\delta>0$ sufficiently large such that $\Prob(| N_{m(k)} - m(k)| > C \sqrt{m(k)}) \leq \delta$ for all $n\in\N$. Next, note that, conditional on $N_{m(k)} = M$, the difference $|\tilde W_{m(k), i} - W_{m(k),i}|$ follows a $\mathrm{Bin}(|M-m(k)|, m(k)^{-1})$ distribution. Further, setting $c=\lceil C\sqrt{m(k)} \rceil$ and choosing $M$ such that $| M - m(k)| \leq c$, we have that $\mathrm{Bin}(|M-m(k)|, m(k)^{-1})([3, |M-m(k)|]) \leq \mathrm{Bin}(c, m(k)^{-1})([3,c])$. As a consequence, conditioning on the event $N_{m(k)}=M$, we obtain 
\begin{align*}
        \Prob(A_n) 
        &\leq \Prob(|I_{m(k)}^3| > 0, |N_{m(k)} - m(k)| \leq c) + \delta \\
        & \leq \sum_{i=1}^{m(k)} \Prob(|\tilde W_{m(k),i} - W_{m(k),i} | \geq 3, | N_{m(k)}-m(k)| \leq c) + \delta \\
        & \le m(k) \mathrm{Bin}(c, m(k)^{-1})([3,c]) + \delta \\
        & \leq m(k)\{c m(k)^{-2} + \mathrm{Poisson}(c m(k)^{-1})([3,c])\} + \delta \\
        & = \frac{c}{m(k)} + m(k) e^{-c/m(k)} \sum_{j=3}^{c} \frac{(c m(k)^{-1})^j}{j!}  +\delta \\
        & = O((m(k))^{-1/2}) + \delta,
\end{align*}
where we used the approximation error of the Poisson Limit Theorem. Hence, since $\delta>0$ was arbitrary, $\Prob(A_n)=o(1)$.

Next, note that, on $A_n^c$,
\[
 |I_{m(k)}^1| + |I_{m(k)}^2|  = \sum_{j=1}^\infty |I_{m(k)}^j| =  \sum_{i=1}^{m(k)} \sum_{j=1}^\infty \bm 1(i \in I_{m(k)}^j)  =\sum_{i=1}^{m(k)} |\tilde W_{m(k),i} - W_{m(k),i} |= |N_{m(k)} - m(k)|,
\]
using the fact that $\bm U_{j,m(k)}$ is multinomially distributed with 1 trial. As a consequence,  letting $H_{n,i} =  \bm 1( N_{m(k)} \neq m(k)) \{ \bm 1(i \in I_{m(k)}^1) + \bm 1(i \in I_{m(k)}^2) \} / |N_{m(k)}-m(k)|$ and 
$
    \tilde R_{nb}  = \big| \sum_{i=1}^{m(k)} H_{n,i} D_{n,i} \big|,
$
we obtain that
\begin{align*}
        |R_{nb}| & \leq \bm 1_{A_n^c} \Big| \sqrt{\frac{n}r} \frac1{m(k)}  \sum_{i=1}^{m(k)} D_{n,i} \sum_{j=1}^2 \bm 1(i \in I_{m(k)}^j) \Big|  + o_\Prob(1) 
        = \bm 1_{A_n^c} \sqrt k  \times \frac{|N_{m(k)}-m(k)|}{{\sqrt {m(k)}}}  \times \tilde R_{nb} + o_\Prob(1).
\end{align*}
Therefore, since $|N_{m(k)}-m(k)|/\sqrt {m(k)} = O_\Prob(1)$ by the Central Limit Theorem, the proof of \eqref{eq:proof_efron_boot_1} is finished if we show that $\tilde R_{nb}=o_\Prob(1)$. 
For that purpose note that, for any fixed $\varepsilon > 0$,
\begin{align*}
    &\phantom{{}={}}\Prob\Big( \frac{\bm 1(N_{m(k)} \neq m(k))}{|N_{m(k)}-m(k)|}  \geq \varepsilon\Big)
        \leq \Prob\Big( |N_{m(k)}-m(k)| \leq  \frac{1}{\varepsilon}\Big) = o(1)
\end{align*}
by the Central Limit Theorem, which in turn yields $\max_{i=1,\ldots,m(k)} H_{n,i} \leq 2|(N_{m(k)}-m(k))|^{-1} \bm 1\{N_{m(k)} \neq m(k)\} = o_\Prob(1)$.
    
Now let $\sigma_r(i,j) = \Cov(D_{n,i}, D_{n,j})$ and invoke the Conditional Tschebyscheff inequality 
to obtain that, for fixed $\varepsilon > 0$,
\begin{align*} 
        & \phantom{{}={}} \Prob\big(|\tilde R_{nb}| > \varepsilon \mid (H_{n,i})_{i=1, \dots, m(k)}\big) \\
        &\lesssim \sum_{i,j=1}^{m(k)} H_{n,i}H_{n,j} |\sigma_r(i,j)| \\
        & \leq \Big\{ \max_{j=1}^{m(k)} H_{n,j} \Big\} \times  \Big\{ \sum_{i=1}^{m(k)}H_{n,i}\sigma_r(1,1)+ 2 \sum_{i=1}^{m(k)-1}H_{n,i} |\sigma_r(1,2)| +\sum_{|i-j|\geq 2} H_{n,i} |\sigma_r(i,j)| \Big\}\\
        &\le 
        \Big\{ \max_{j=1}^{m(k)} H_{n,j} \Big\} \times  \Big\{ \sigma_r(1,1)+ 2 |\sigma_r(1,2)| + 2  m(k) \max_{d=2}^{m(k)-1} |\sigma_r(1,1+d)|  \Big\},
\end{align*}
where we used $\sum_{i=1}^{m(k)}H_{n,i} \le 1$. Since $\max_{j=1}^{m(k)} H_{n,j} = o_\Prob(1)$ and by the (proof of the) three assertions in \eqref{eq:proof_kloop_sliding_clt_3}, we obtain that the expression on the right-hand side of the previous display is $o_\Prob(1)$.
Hence, writing $\Prob(|\tilde R_{nb}| > \varepsilon) = \Exp\big[\Prob\big(|\tilde R_{nb}| > \varepsilon \mid (H_{n,i})_i\big)\big]$ and invoking the Dominated Convergence Theorem for convergence in probability, we obtain $\tilde R_{nb}=o_\Prob(1)$ and the proof is finished.
\end{proof}

\begin{proof}[Proof of Proposition~\ref{prop:M_Bst_cons}]
By Equations \eqref{eq:phi_decomposition}, \eqref{eq:vartheta}, \eqref{eq:linerror} and \eqref{eq:bootlin1}, \eqref{eq:bootlin2} we have 
\begin{align*}
        \hat {\bm \theta}_n^{[\mbl]} - {\bm \theta}_r
        &= A(\bm a_r, \bm b_r) \sqrt{\frac{r}{n}}\big( \bar{\mb G}_{n,r}^{[\mbl]}\bm h+ o_\Prob(1) \big) =: \sqrt{\frac{r}{n}}A(\bm a_r, \bm b_r) S_n,\\
        \hat {\bm \theta}_n^{[\cblk],*} - \hat {\bm \theta}_n^{[\cblk]}
        &= A(\bm a_r, \bm b_r) \sqrt{\frac{r}{n}}\big(\hat {\mb G}_{n,r}^{(\cblk, \ast)}\bm h+ o_\Prob(1) \big) =: \sqrt{\frac{r}{n}}A(\bm a_r, \bm b_r) S_n^*.
\end{align*}
By Theorems \ref{theo:cltblocks1} and \ref{theo:cltblocks2} we have $S_n \weak \mc N_q(\bm0, \Sigma_{\bm h}^{[\sbl]})$. 
Theorem \ref{theo:mult_boot} implies
$
        d_K ( \mathcal L ( S_n^* \mid \mathcal X_n) , 
\mathcal L( S_n )) = o_\Prob(1).
$
The assertion then follows from Lemma \ref{lem:ZBstToMBst} of the supplement.
\end{proof}

\begin{proof}[Proof of Theorem~\ref{theo:boot-mle-blockmax}]
The results regarding $(\hat \alpha_n^{[\cblk]}, \hat \sigma_n^{[\cblk]})^\top$ and $(\hat \alpha_n^{[\cblk], *}, \hat \sigma_n^{[\cblk], *})^\top$ follow from an application of Theorem~\ref{theo:asy-mle-frechet} and Theorem~\ref{theo:boot-mle-frechet}, respectively.

For the former we need to show that, with $v_n = \sqrt{n/r}, \omega_n = n$ and $X_{n,s}=M_{r,s}^{[\cblk]}\vee c$, the ``no-ties-condition'' in Equation~\eqref{eq:noties} is met, and that the three convergences in Condition~\ref{cond:fre} are satisfied. 
Equation~\eqref{eq:noties} follows from $\Prob(X_{n,1} = \dots = X_{n,n}) \le \Prob(X_{n,1} = X_{n, r+1})=\Prob((M_{r,1} \vee c)/\sigma_r=(M_{r,r+1} \vee c)/\sigma_r)$, 
which converges to zero by the Portmanteau theorem, observing that $(M_{r,1} \vee c)/\sigma_r, (M_{r,r+1} \vee c)/\sigma_r)$ weakly converges to the product of two independent Fr\'echet$(\alpha_0,1)$ random variables by Lemma 5.1 in \cite{BucSeg18-sl}.
Condition~\ref{cond:fre}(iii) is a consequence of Condition~\ref{cond:bias_Fre}. Condition~\ref{cond:fre}(ii) follows from a straightforward modification of Theorem~\refstar{theo:cltblocks2} to the Fr\'echet domain of attraction. Finally, for $\alpha_+$ chosen sufficiently close to $\alpha_0$, Condition~\ref{cond:fre}(i) follows from such a modification as well, following the argumentation on page 1116 in \cite{BucSeg18a}.

Next, consider the assertions regarding $(\hat \alpha_n^{[\cblk], *}, \hat \sigma_n^{[\cblk], *})^\top$, which follow from Theorem~\ref{theo:boot-mle-frechet} if we additionally show that the ``no-ties-condition'' in Equation~\eqref{eq:noties2} is met, that $\hat \Prob_{n,r}^{[\cblk],*} f=P_{\alpha_0,1} f +o_\Prob(1)$ for all $f \in \mathcal F(\alpha_-, \alpha_+)$ and that \eqref{eq:boot-mle-expansion} is met. The latter two assertions follow from similar arguments as given in the proof of Theorem~\refstar{theo:mult_boot} (adaptations to the Fr\'echet domain of attraction are needed), and it remains to show \eqref{eq:noties2}. 
For that purpose, let $I_1, \dots, I_{m(k)}$ be iid uniformly distributed on $\{1, \dots, m(k)\}$ independent of $(X_{n,1}, \ldots, X_{n,n})$, such that 
\[
    (X_{n,1}^*, \dots, X_{n,n}^*)
    \stackrel{d}{=}
    (X_{n,kr(I_1-1)+1}, \dots, X_{n,kr(I_1-1)+kr},
    \dots, X_{n,kr(I_{m(k)}-1)+1}, \dots, X_{n,kr(I_{m(k)}-1)+kr}).
\]
Then
\begin{align*}
\Prob(X_{n,1}^\ast = \ldots = X_{n,n}^\ast) 
&=
\sum_{i=1}^{m(k)} \Prob( I_1 = i, X_{n,1}^\ast = \ldots = X_{n,n}^\ast) 
\\& \le
\sum_{i=1}^{m(k)} \Prob( I_1 = i, X_{n,kr(I_1-1)+1} = \ldots = X_{n,kr(I_1-1)+kr}) 
\\&=
\frac1{m(k)}\sum_{i=1}^{m(k)} \Prob( X_{n,kr(i-1)+1} = \ldots = X_{n,kr(I_1-1)+kr}) 
=
\Prob( X_{n,1} = \ldots = X_{n,kr}). 
\end{align*}
This probability is bounded by $\Prob(X_{n,1} = X_{n, r+1})$, which converges to zero as shown above when proving Equation~\eqref{eq:noties}.
\end{proof}

%%%%%%%%%%%%%%%%%%%%%%%%%%%%%%%%%%%%%%%%%%%%%%%%%%%%%%%%%%%%%%%%%%%%%%%%%%%%%%%%%%

\begin{proof}[Proof of Corollary~\ref{cor:confint-fremle}] 
Define $S_n^{[\sbl]}=\sqrt{\tfrac{n}r} ( \hat \sigma_{n}^{[\sbl]} / \sigma_{r_n} -1)$ and
$S_n^{[\cblk],*}=  \sqrt{\tfrac{n}r} (\hat \sigma_n^{[\cblk], *} - \hat \sigma_n^{[\cblk]}) / \sigma_{r_n}$, and note that, with $M_2(\alpha_0)$ denoting the second row of $M(\alpha_0)$,
\[
S_n^{[\sbl]} = M_2(\alpha_0) \Gb_{n,r}^{[\sbl]}(f_1, f_2, f_3) + o_\Prob(1),
\qquad
S_n^{[\cblk],*} = \frac{\hat \sigma_n^{[\cblk]}}{\sigma_{r_n}} \sqrt{\frac{n}r}  \Big( \frac{\hat \sigma_n^{[\cblk],*}}{ \hat \sigma_n^{[\cblk]}} - 1 \Big)
=
M_2(\alpha_0) \hat \Gb_{n,r}^{[\cblk],*}(f_1, f_2, f_3) + o_\Prob(1)
\]
by Theorem~\ref{theo:boot-mle-blockmax}. Moreover, arguing as in the proof of that theorem, the weak limit of $S_n^{[\sbl]}$ coincides with the conditional weak limit of $\hat S_n^{[\cblk],*}$ given the observations $\Xc_n$. Therefore, observing that $F_{S_n^{[\cblk],*}}(x) = \Prob(S_n^{[\cblk],*} \le x \mid \Xc_n)$ satisfies $F_{S_n^{[\cblk],*}}(x) =F_{\hat \sigma_n^{[\cblk],*}}(\hat \sigma_n^{[\cblk]} + \sqrt{\tfrac{r}n} \cdot \sigma_{r_n} x)$ and hence
\begin{align*}
\Prob(\sigma_{r_n} \in I_{n,\sigma}^{[\sbl, \cblk]}) 
& = 
\Prob \Big[ \sqrt{\frac{n}r} \cdot \sigma_{r_n}^{-1} \{ (F_{\hat \sigma_n^{[\cblk],*}})^{-1}( \tfrac\beta2) - \hat \sigma_n^{[\cblk]}\} \le \sqrt{\frac{n}r} \Big( \frac{\hat \sigma_{n}^{[\sbl]}}{\sigma_{r_n}} -1\Big) \le  \sqrt{\frac{n}r} \cdot \sigma_{r_n}^{-1} \{ (F_{\hat \sigma_n^{[\cblk],*}})^{-1}(1- \tfrac\beta2) - \hat \sigma_n^{[\cblk]}\} 
\Big]\\
& = \Prob\Big[ (F_{S_n^{[\cblk],*}})^{-1}(\tfrac\beta2) \le S_n^{[\sbl]}  \le (F_{S_n^{[\cblk],*}})^{-1}(1-\tfrac\beta2) \Big], 
\end{align*}
the assertion follows from Lemma 4.2 in \cite{BucKoj2019note}.
\end{proof}

%%%%%%%%%%%%%%%%%%%%%%%%%%%%%%%%%%%%%%%%%%%%%%%%%%%%%%%%%%%%%%%%%%%%%%%%%%%%%%%%%%

%%%%%%%%%%%%%%%%%%%%%%%%%%%%%%%%%%%%%%%%%%%%%%%%%%%%%%%%%%%%%%%%%%%%%%%%%%%%%%%%%%
\section{Auxiliary results}
\label{sec:auxres}

\begin{proposition}[Joint weak convergence of circular block maxima]\label{prop:overlap_wconv_kloop}
Fix $k\in \N_{\ge 2}$ and suppose Conditions \refstar{cond:mda} and \refstar{cond:ser_dep}(a), (b) are met. Then, for any $ \xi, \xi^\prime \in [0, k)$ and $ \bm x, \bm y \in \R^d$, we have 
\begin{align}\label{eq:G_xi_kloop}
    &\phantom{{}={}}\lim_{n \to \infty} 
        \Prob \Big(\bm Z_{r,1+\lfloor\xi r \rfloor}^{[\cblk]} \leq \bm x, \bm Z_{r,1+\lfloor \xi^\prime r \rfloor}^{[\cblk]}\leq \bm y \Big) 
        =
       G_{|\xi-\xi'|}^{(k)} (\bm x, \bm y) := \begin{cases}
    G_{k -|\xi-\xi^\prime|}(\bm x, \bm y), &  |\xi - \xi^\prime| > k-1, \\
    G_{|\xi-\xi^\prime|}(\bm x, \bm y), &|\xi - \xi^\prime| \le k-1,
    \end{cases}
\end{align}
with $G_\xi$ from (\refstar{eq:Gxi-prod}).
Moreover, any circular block maxima taken from distinct $kr$-blocks are asymptotically independent, that is, for any $1 \leq i \ne i'\leq m(k),  \xi, \xi^\prime \in [0, k)$ and $ \bm x, \bm y \in \R^d$,
\begin{align}\label{eq:G_xi_kloop-indep}
\lim_{n \to \infty} 
        \Prob \Big(\bm Z_{r,(i-1)kr+1+\lfloor\xi r \rfloor}^{[\cblk]} \leq \bm x, \bm Z_{r,(i^\prime-1) kr+1+\lfloor \xi^\prime r \rfloor}^{[\cblk]}\leq \bm y \Big) 
        =
       G_{1}(\bm x, \bm y) =G(\bm x) G(\bm y).
\end{align}
\end{proposition}

\begin{proof}
We start by showing marginal convergence (i.e., $\bm y= \bm \infty$); the result then corresponds to Proposition~\refstar{prop:Z_kloop_xi_wconv}.
The case $\xi \leq k-1$ is trivial. For $\xi > k-1$  and fixed $\bm x \in \R^d$, we may proceed as in the proof of Lemma 2.4 in \cite{BucZan23} to obtain
\begin{align*} 
\Prob \Big(\bm Z_{r,1+\lfloor\xi r \rfloor}^{[\cblk]} \leq \bm x \Big)
&= 
\Prob\Big( \frac{\max(\bm X_{1+\lfloor\xi r \rfloor},\ldots, \bm X_{kr}) - \bm b_r}{\bm a_r} \le \bm x,  \frac{\max(\bm X_{1},\ldots, \bm X_{r+\lfloor\xi r \rfloor - kr}) - \bm b_r}{\bm a_r} \leq \bm x\Big) \\
&=
\Prob\Big( \frac{\max(\bm X_{1+\lfloor\xi r \rfloor},\ldots, \bm X_{kr}) - \bm b_r}{\bm a_r} \le \bm x \Big)
\Prob\Big(\frac{\max(\bm X_{1},\ldots, \bm X_{r+\lfloor\xi r \rfloor - kr}) - \bm b_r}{\bm a_r} \leq \bm x\Big)  + R_n,
\end{align*}
where $R_n=O(\alpha(r_n))=o(1)$.
Next, by stationarity, the product on the right-hand side of the previous display can be written as
\begin{align*}
\Prob\Big( \bm Z_{kr - \flo{\xi r}} \le \frac{\bm a_r \bm x + \bm b_r - \bm b_{kr - \flo{\xi r}}}{\bm a_{kr - \flo{\xi r}}} \Big)
\Prob\Big(\bm Z_{r + \flo{\xi r}-kr}  \le \frac{\bm a_r \bm x + \bm b_r - \bm b_{r + \flo{\xi k} -kr }}{\bm a_{r + \flo{\xi r} -kr}} \Big) 
\end{align*}
The convergence in Equation~(\refstar{eq:rvscale}) being locally uniform we obtain, for $j\in \{1, \dots, d\}$,
\begin{align*}
&\lim_{n \to \infty} \frac{a_r^{(j)} x^{(j)} + b_r^{(j)} - b^{(j)}_{kr - \flo{\xi r}}}{a^{(j)}_{kr - \flo{\xi r}}} 
= (k -\xi)^{-\gamma^{(j)}}x^{(j)} + \frac{(k -\xi)^{-\gamma^{(j)}}-1}{\gamma^{(j)}}, \\
&\lim_{n \to \infty} \frac{a_r^{(j)} x^{(j)} + b_r^{(j)} - b^{(j)}_{r + \flo{\xi r} -kr}}{a^{(j)}_{r + \flo{\xi r} -kr}} 
= (1+\xi - k)^{-\gamma^{(j)}}x^{(j)} +\frac{(1+\xi-k)^{-\gamma^{(j)}}-1}{\gamma^{(j)}}.
\end{align*}
Arguing as in the proof of Lemma 8.9 from \cite{BucSta24}, Condition \refstar{cond:mda} and the previous three displays yield, for $n \to \infty,$
\begin{align*}
\Prob(\bm Z_{r,1+\lfloor\xi r \rfloor}^{[\cblk]} \leq \bm x )
    & = G^{k-\xi}(\bm x) G^{1+\xi-k}(\bm x) +o(1) = G(\bm x) + o(1).
\end{align*}

We proceed by proving \eqref{eq:G_xi_kloop}, and start by considering the case $|\xi-\xi^\prime| > k-1$. Without loss of generality let $\xi< \xi^\prime$ such that $1+ \flo{\xi r} \le r$ and $1+\flo{\xi^\prime r}>(k-1)r$.  Then, by the same arguments as for the marginal convergence,
\begin{align*}
        &\phantom{{}={}}\lim_{n \to \infty} 
        \Prob \big(\bm Z_{1+\lfloor\xi r \rfloor}^{[\cblk]} \leq \bm x, \bm Z_{1+\lfloor\xi^\prime r \rfloor}^{[\cblk]} \leq \bm y \big) \\
        & =\lim_{n \to \infty} \Prob \Big( \frac{\max(\bm X_{r,1},\ldots, \bm X_{r,\flo{\xi r}}) -\bm b_r}{\bm a_r} \leq \bm y, 
        \frac{\max(\bm X_{r,1+\flo{\xi r}}, \ldots, \bm X_{r, \flo{\xi^\prime r}-(k-1) r})-\bm b_r}{\bm a_r} \leq \bm x \wedge \bm y \\
        &\hspace{3cm} \frac{\max(\bm X_{r,  \flo{\xi^\prime r}-(k-1) r + 1} ,\ldots, \bm X_{r,r+\flo{\xi r}})-\bm b_r}{\bm a_r} \leq \bm x, 
        \frac{\max(\bm X_{r, 1+ \flo{\xi^\prime r}}, \ldots, \bm X_{kr}) -\bm b_r}{\bm a_r} \leq \bm y\Big) \\
        & = G^{\xi}(\bm y) \cdot G^{\xi' - (k-1) - \xi}(\bm x \wedge \bm y) \cdot  G^{1+\xi  - \xi^\prime + (k-1)}(\bm x) \cdot  G^{k - \xi^\prime}(\bm y) \\
        & = G^{k-(\xi'-\xi)}(\bm y) \cdot G^{1-(k-(\xi'-\xi))}(\bm x \wedge \bm y) \cdot  G^{k - (\xi^\prime-\xi)}(\bm x)  \\
        & = G_{ k - (\xi^\prime-\xi)}(\bm x, \bm y),
\end{align*}
where we used (\refstar{eq:Gxi-prod}) at the last equality.

Next, consider the case $|\xi-\xi^\prime| \leq k-1$, and again assume $\xi < \xi^\prime$. There are two subcases to handle; first let $\xi^\prime \in (k-1,k].$ In that case we have $\xi \geq 1,$ hence there is no overlapping from the left (induced by the $\circmax$-operation). We may then proceed in a similar (but simpler) way as for $|\xi-\xi'|>k-1$ to obtain the stated limit. Alternatively, we have $\xi^\prime \leq k-1.$ In that case, we are in the sliding block maxima case, where the result is known, see e.g. Lemma 8.6 in \cite{BucSta24}.

Finally, the assertion in \eqref{eq:G_xi_kloop-indep} follows from similar arguments as in Lemmas A.7, A.8 from \cite{BucSeg18a}, which yield 
\begin{align*}
        \Prob \big(\bm Z_{(i-1)kr + 1+\lfloor\xi r \rfloor}^{[\cblk]} \leq \bm x, \bm Z_{(i'-1)kr +1+\lfloor\xi^\prime r \rfloor}^{[\cblk]} \leq \bm y \big) 
        & =  \Prob\big(\bm Z_{1+\lfloor\xi r \rfloor}^{[\cblk]} \leq \bm x\big) \cdot \Prob \big(\bm Z_{1+\lfloor\xi^\prime r \rfloor}^{[\cblk]} \leq \bm y\big) +o(1)
\end{align*}
by stationarity. 
The product on the right-hand side converges to $G(\bm x) G(\bm y)$ by the marginal convergence.
\end{proof}

\begin{lemma}[Multiplier central limit theorem] 
\label{lem:aux_mult_clt}
Suppose Conditions \refstar{cond:mda} and \refstar{cond:ser_dep} are met. Fix $k\in\N$ and let $\bm h=(h_1, \dots, h_q)^\top$ satisfy Condition \refstar{cond:int_h} with $\nu > 2/\omega$ and $\omega$ from Condition~\refstar{cond:ser_dep}.
Let $Y_{n,1},\ldots,Y_{n,m(k)}$ be a triangular array of rowwise independent and identically distributed random variables that is independent of $(\bm X_1, \bm X_2, \dots)$ such that $\limsup_{n\to\infty} \Exp[| Y_{n,1}|^{2+\nu}]< \infty$ with $\nu$ from Condition~\refstar{cond:int_h} and such that the limit $\mu_2 := \lim_{n\to\infty} \Exp[Y_{n,1}^2]>0$ exists. 
Then, for $n \to \infty$ and with $\Sigma_{\bm h}^{[\mbl]}$ from (\refstar{eq:asy_Cov}), we have
\begin{align*}        
    \bar \Gb_{n,r}^\circ \bm h :=  \sqrt{\frac{n}r} \frac{1}{n} \sum_{i=1}^{m(k)}  Y_{n,i} 
    \sum_{s \in I_{kr,i}} \big\{ \bm h(\bm Z_{r,s}^{[\cblk]}) - \Exp[\bm h(\bm Z_{r,s}^{[\cblk]})] \big\} 
   \wconv 
   \begin{cases}
   \Nc_q(0, \mu_2 \Sigma_{\bm h}^{[\dbl]}) & \text{ if } k=1, \\
   \Nc_q(0, \mu_2 \Sigma_{\bm h}^{[\sbl]}) & \text{ if } k\in\{2,3,\dots\}.
   \end{cases}
\end{align*}
\end{lemma}

\begin{proof} 
We only consider the case $k\ge 2$; the case $k=1$ is a simple modification. By the Cramér-Wold Theorem, it is sufficient to consider the case $q=1$. Since $m(k)=n/(kr)$, we may write $\bar \Gb_{n,r}^\circ h =  \sqrt{\frac{n}r} \frac1{m(k)} \sum_{i=1}^{m(k)} Y_{n,i} D_{n,i}$, with
$D_{n,i} = \frac{1}{kr} \sum_{s \in I_{kr,i}} \{ h(\bm Z_{r,s}^{[\cblk]}) - \Exp[h(\bm Z_{r,s}^{[\cblk]})] \}$ from \eqref{eq:dni}. Then, since $(Y_{n,i} D_{n,i})_i$ is uncorrelated,
\begin{align*}
\Var(\bar \Gb_{n,r}^\circ h)
         &= k \Exp[Y_{n,1}^2]  \Var (D_{n,1}) = \mu_2 \Sigma_h^{[\sbl]} +o(1)
\end{align*}
by \eqref{eq:proof_kloop_sliding_clt_3}. This implies the assertion for $ \Sigma_h^{[\sbl]}=0$.

If $ \Sigma_h^{[\sbl]}>0$, we may follow a similar line of reasoning as in the proof of Theorem~\refstar{theo:cltblocks2}: decompose
\begin{equation*}
        \bar{\mb G}_{n,r}^\circ h
        = \frac{1}{\sqrt{p}} \sum_{j=1}^p ( T_{n,j}^+ + T_{n,j}^- ),
\quad \text{ where } \quad
        T_{n,j}^\pm := \sqrt{\frac{np}{r}}\frac{1}{m(k)} \sum_{i \in J_j^\pm}  Y_{n,i}D_{n,i},
\end{equation*}
where $p=m(k)/m^\ast$ with $m^\ast=o(m^{\nu/(2(1+\nu))})$ and $J_j^+ := \{(j-1)m^\ast + 1, \ldots, jm^\ast-1\}, J_j^- := \{jm^\ast \}$. Following the arguments from the proof of Theorem~\refstar{theo:cltblocks2}, we have $p^{-1/2}\sum_{j=1}^p T_{n,j}^-=o_\Prob(1)$. Moreover, we may assume independence of $(S_{n,j}^+)_j$, which enables an application of Ljapunov's central limit theorem to $p^{-1/2}\sum_{j=1}^p T_{n,j}^+$ to conclude.
\end{proof}

The next result is possibly well-known; we skip its elementary proof (see also Section 2 in \citealp{KlaWel91}).

\begin{lemma}[Poissonization]\label{lem:poiss}
For $q \in\N_{\ge 2}$ and $(p_{1},\dots, p_{q}) \in [0,1]^q$ with $p_1+ \dots + p_q = 1$, let $(\bm U_{j,q})_{j\in\N}$ be i.i.d.\ multinomial vectors with $1$ trial and $q$ classes with class probabilities $p_1, \dots, p_q$. For $m\in \N$ let 
$
\bm W_{m,q} = (W_{m,q,1}, \dots, W_{m,q,q}) = \sum_{j=1}^m \bm U_{j,q}
$
(which is multinomial with $m$ trials and $q$ classes). Further, for $\lambda>0$, let $N$ be a $\mathrm{Poi}(\lambda)$-distributed random variable that is independent of $(\bm U_{j,q})_{j \in \N}$. Then, the random variables $W_{N,q,1}, \dots, W_{N,q,q}$ are independent and $W_{N,q,j}$ follows a $\mathrm{Poi}(\lambda p_{j})$-distribution for $j=1, \dots, q$.
\end{lemma}

\begin{lemma}
\label{lem:sliding-covpos}
Let $d=1$ and suppose that $h$ is a real-valued function such that $\Var(h(\bm Z)) \in (0,\infty)$ exists. Then, for all $\xi \in [0,1]$,
$
\Cov(h(\bm Z_{1,\xi}), h(\bm Z_{2, \xi})) \ge 0
$
with strict inequality in a neighbourhood of $0$.
\end{lemma}

\begin{proof}
By our assumption $d=1$, the CDF $G$ from (\refstar{eq:firstorder}) corresponds to the $\mathrm{GEV}(\gamma)$- distribution, which has a Lebesgue-density that we denote by $g$. For $\xi \in (0,1)$, we may 
employ the following stochastic construction for $(\bm Z_{1,\xi}, \bm Z_{2,\xi})^\top$: let $U, W$ be i.i.d.\ with CDF $G^\xi$ and independent of $V,$ the latter having the CDF $G^{1-\xi}.$ We then have $\mc L(\bm Z_{1,\xi}, \bm Z_{2,\xi}) = \mc L( U \vee V,W \vee V).$ Without loss of generality assume $\Exp[h(\bm Z_{1,\xi})] = 0$ and note that 
\begin{align}\label{eq:proof_sliding-covpos}
    \Cov(h(\bm Z_{1,\xi}), h(\bm Z_{2, \xi}))
    \nonumber&= \Exp[h( U \vee V) h(W \vee V)] \\
    \nonumber& = 2 \Exp[h(U) h(V) G^\xi(V) \bm 1(V < U)] \\ 
    & \hspace{1cm}+ \Exp[h(U) h(W) G^{1-\xi}(U \wedge W)] + \Exp[h^2(V) G^{2\xi}(V)] 
    \equiv 2I_1 + I_2 + I_3.
\end{align}
Let $H\colon \bar{\R} \to \R, H(x) = \int_{-\infty}^x h(z) g(z) \diff z$ and note that $H(\infty) = 0 = H(-\infty)$, where we made use of $\Exp[h(\bm Z_{1,\xi})]=\int_\R h(z)g(z)\diff z$. Using integration by parts we obtain
\[
    I_1 = (1-\xi)\xi \int_{-\infty}^\infty G^{\xi-1}(x) H^\prime(x) H(x) \diff x 
    = \frac{\xi}{2} (1-\xi)^2 \int_{-\infty}^\infty G^{\xi -2}(x) g(x) H^2(x) \diff x \geq 0.
\]
Similarly, we have $I_2 = \xi^2(1-\xi)\int_{-\infty}^\infty G^{\xi -2}(x) g(x)H^2(x) \diff x \geq 0.$ In view of \eqref{eq:proof_sliding-covpos} and $I_3 \geq 0$ this proves the first assertion. 

For the second assertion note that $I_3 = (1-\xi) \Exp[h^2(\bm Z) G^\xi(\bm Z)] \leq \Var(h(\bm Z)) < \infty$ and $I_j \geq 0$ for $j=1,2,3$. By dominated convergence and \eqref{eq:proof_sliding-covpos} we then have
\[
    \liminf_{\xi \downarrow 0} \Cov(h(\bm W_{1,\xi}), h(\bm W_{2, \xi})) 
    \geq \liminf_{\xi \downarrow 0} I_3 = \Var(h(\bm Z)) > 0,
\]
which lets us conclude.
\end{proof}

%%%%%%%%%%%%%%%%%%%%%%%%%%%%%%%%%%%%%%%%%%%%%%%%%%%%%%%%%%%%%%%%%%%%%%%%%%%%%%%%%%
\subsection{A general result on bootstrapping the Fréchet independence MLE}
\label{sec:general-bootmle}

Throughout this section, let $\mathcal X_n = (X_{n,1}, \dots, X_{n,\omega_n})$ denote a sequence of random vectors in $(0, \infty)^{\omega_n}$ with continuous cumulative distribution functions $F_{n,1},  \dots, F_{n,\omega_n}$, where $\omega_n \to \infty$ is a sequence of integers. Throughout, we assume that asymptotically not all observations are tied, that is, 
\begin{align} \label{eq:noties}
\lim_{n\to \infty} \Prob(X_{n,1} = \dots = X_{n,\omega_n}) = 0.
\end{align}
For some $\alpha_0>0$ and $\sigma_n \to \infty$, proximity of the $F_{n,1},  \dots, F_{n,\omega_n}$ to the Fréchet-distribution $P_{\alpha_0, \sigma_n}$ will be controlled by convergence conditions on empirical means of $X_{n,i}/\sigma_n$: 
for $0<\alpha_- < \alpha_0 < \alpha_+<\infty$, let $\mathcal F(\alpha_-, \alpha_+)$ denote the set of functions containing $x\mapsto \log x, x \mapsto x^{-\alpha}, x \mapsto x^{-\alpha} \log x$ and $x \mapsto x^{-\alpha} \log^2x$ for all $\alpha \in (\alpha_-, \alpha_+)$. Moreover, recall 
$f_1(x) = x^{-\alpha_0}$, $f_2(x) = x^{-\alpha_0} \log x$ and $f_3(x) =\log x$ from \eqref{eq:Fref_i},
considered as functions on $(0,\infty)$, and let $\mathcal H=\{f_1, f_2, f_3\}$. For a real-valued function $f$ defined on $(0,\infty)$ such that the following integrals/expectations exist, let
\[
\mathbb P_n f = \frac1{\omega_n} \sum_{i=1}^{\omega_n} f(X_{n,i}/\sigma_n),
\quad 
P_n f  = \frac1{\omega_n} \sum_{i=1}^{\omega_n} \Exp[f(X_{n,i}/\sigma_n)],
\quad P_{\alpha_0,1}f := \int f(x) \diff P_{\alpha_0,1}(x).
\]

\begin{condition}\label{cond:fre}
\begin{compactenum}[(i)]
\item 
There exists  $0<\alpha_- < \alpha_0 < \alpha_+<\infty$ such that $\mathbb P_n f  \rightsquigarrow P_{\alpha_0,1} f$ for all $f \in \mathcal F(\alpha_-, \alpha_+)$.

\item 
There exists $0<v_n \to \infty$ such that $\bar{\mathbb G}_n (f_1, f_2, f_3)^\top   \weak N_3(\bm 0,  \Sigma_G)$, where
$\bar{\mathbb G}_n f = v_n(\mathbb P_n f - P_n f)$ and where $\Sigma_{G} \in \R^{3 \times 3}$ is positive semidefinite.

\item 
With $v_n$ from (ii), the limit  $\bm B_G := \lim_{n \to \infty}
B_n (f_1, f_2, f_3)^\top$ exists, where $B_n f =  v_n(P_n f - P_{\alpha_0,1} f)$.
\end{compactenum}
\end{condition}

Recall the log-likelihood function $\ell_{\bm \theta}$ of the Fréchet distribution from \eqref{eq:frechetloglik}.

\begin{theorem}[\citealp{BucSeg18a}]
\label{theo:asy-mle-frechet}
(i) Suppose that \eqref{eq:noties} holds and that Condition~\ref{cond:fre}(i) is met.  Then, on the complement of the event $\{X_{n,1}=\dots=X_{n,\omega_n}\}$, the independence Fréchet log-likelihood ${\bm \theta} \mapsto 
\sum_{i=1}^{\omega_n}  \ell_{\bm \theta} (X_{n,i})$ has a unit maximizer $\hat {\bm \theta}_n=(\hat \alpha_n, \hat \sigma_n)$, and that maximizer is consistent in the sense that $(\hat \alpha_n, \hat \sigma_n/\sigma_n) = (\alpha_0,1) +o_\Prob(1)$ as $n\to\infty$.

\smallskip
\noindent
(ii) If, additionally, Condition~\ref{cond:fre}(ii)-(iii) is met, then
\[
v_n \begin{pmatrix}
\hat\alpha_n - \alpha_0 \\ 
\hat \sigma_n /\sigma_n - 1
\end{pmatrix}
=
M(\alpha_0)(\bar \Gb_n + B_n)(f_1, f_2, f_3)^\top + o_\Prob(1)
\weak \Nc_2(M(\alpha_0)\bm B_G, M(\alpha_0) \Sigma_G M(\alpha_0)^\top)
\]
as $n\to\infty$, where 
\[
M(\alpha_0) = \frac{6}{\pi^2} \begin{pmatrix}
\alpha_0^2 & \alpha_0(1-\gamma) & - \alpha_0^2 \\
\gamma-1 & -(\Gamma''(2)+1)/\alpha_0 & 1-\gamma \end{pmatrix}
\]
with
$\Gamma(z) = \int_0^\infty t^{z-1} {e}^{-t} \diff t$ the Euler Gamma function and $\gamma = 0.5772\ldots$ the Euler-Mascheroni constant. 
\end{theorem}

Now, conditional on $\mathcal X_n = (X_{n,1}, \dots, X_{n,\omega_n})$, let $\mathcal X_n^*=(X_{n,1}^*, \dots, X_{n,\omega_n}^*)$ denote a bootstrap sample of $\mathcal X_n$; formally, $\mathcal X_n^*$ is assumed to be a measurable function of both $\mathcal X_n$ and of some additional independent random element $\mathcal W_n$ taking values in some measurable space. For the subsequent consistency statements, the bootstrap sample is assumed to satisfy the not-all-tied condition
\begin{align} \label{eq:noties2}
\lim_{n\to \infty} \Prob(X_{n,1}^* = \dots = X_{n,\omega_n}^*) = 0
\end{align}
or, equivalently, $\Prob\big(X_{n,1}^* = \dots = X_{n,\omega_n}^* \mid \Xc_n \big) = o_\Prob(1)$ as $n\to\infty$.
The bootstrap scheme is assumed to be regular in the sense that the conditional distribution of certain rescaled arithmetic means is close to the distribution of respective arithmetic means of the original sample. For a real-valued function $f$ on $(0,\infty)$, define $\hat \Prob_n^* f = \frac1{\omega_n} \sum_{i=1}^{\omega_n} f(X_{n,i}^*/\sigma_n)$ and $\hat \Gb_n^* f = v_n (\hat \Prob_n^* f - \Prob_n f)$.

\begin{theorem}
\label{theo:boot-mle-frechet}
Suppose that \eqref{eq:noties}, \eqref{eq:noties2} and Condition~\ref{cond:fre} is met, that $\hat \Prob_n^* f = P_{\alpha_0, 1} f + o_\Prob(1)$ for all $f \in \mathcal F(\alpha_-, \alpha_+)$ (or, equivalently, $\Prob(|\hat \Prob_n^* f -  P_{\alpha_0, 1}|>\eps \mid \Xc_n) = o_\Prob(1)$ for all $\eps>0$ and all $f \in \mathcal F(\alpha_-, \alpha_+)$) 
and that 
\begin{align}
\label{eq:fre-cons}
d_K\big( \Lc(\hat \Gb_n^\ast (f_1, f_2, f_3)^\top \mid \Xc_n), \Nc_3(\bm 0, \Sigma_G) \big) = o_\Prob(1). 
\end{align}
Then, on the complement of the event $\{X_{n,1}^*= \dots= X_{n,\omega_n}^*\}$, the independence Fréchet-log-likelihood ${\bm \theta} \mapsto \sum_{i=1}^{\omega_n} \ell_{\bm \theta}( X_{n,i}^*)$ has a unit maximizer ${\bm \theta}_n^*=(\hat \alpha_n^*, \hat \sigma_n^*)$, and that maximizer satisfies
\begin{align}
\label{eq:boot-mle-expansion}
v_n 
\begin{pmatrix}
\hat \alpha_n^\ast  - \hat \alpha_n \\
\hat \sigma_n^*/\hat \sigma_n -1
\end{pmatrix}
= M(\alpha_0) \hat \Gb_n^* (f_1,f_2,f_3)^\top + o_\Prob(1).
\end{align}
As a consequence, if additionally $\bm B_G=0$, we have bootstrap consistency in the following sense
\begin{align}
\label{eq:boot-mle-cons}
        d_K\left[  \mc L \left(v_n 
        \begin{pmatrix}
            \hat \alpha_n^\ast  - \hat \alpha_n \\
            \hat \sigma_n^*/\hat \sigma_n -1
        \end{pmatrix}
        \, \Big| \,  \mathcal X_n \right) , 
        \mc L \left(v_n 
        \begin{pmatrix}
            \hat \alpha_n  - \alpha_0 \\
            \hat \sigma_n/\sigma_n -1
        \end{pmatrix}
        \right) \right]
        = o_{\Prob}(1).
\end{align}
\end{theorem}

Since it concerns empirical means only, the condition in \eqref{eq:fre-cons} is typically a standard result for the bootstrap; see for instance Section 10 and 11 in \cite{Kos08}.

\begin{proof}[Proof of Theorem~\ref{theo:boot-mle-frechet}]
First, by Theorem~\ref{theo:asy-mle-frechet} (which is Theorem 2.3, Theorem 2.5 and Addendum 2.6 in \citealp{BucSeg18a} applied to the sample $\Xc_n$), we have 
\begin{equation}\label{eq:proof_genFreBst-Expa1}
        v_n \begin{pmatrix}
        \hat\alpha_n - \alpha_0 \\ 
        \hat \sigma_n /\sigma_n - 1
        \end{pmatrix} 
        = M(\alpha_0) ( \bar \Gb_n + B_n) (f_1,f_2,f_3)^\top + o_\Prob(1) 
        \wconv M(\alpha_0) \mc N_3(\bm {B}_G,\Sigma_G).
\end{equation}
Next, in view of the fact that \eqref{eq:fre-cons} implies
$\hat \Gb_n^*(f_1, f_2, f_3)^\top \wconv \mc N_3(\bm 0, \Sigma_G)$ (unconditionally) by Lemma 2.3 in \cite{BucKoj2019note},  
we have 
\[
    \Gb_n^\circ(f_1,f_2,f_3)^\top := v_n \big( \hat \Prob_n^\ast(f_1,f_2,f_3)^\top - P_{\alpha_0,1}(f_1,f_2,f_3)^\top\big) \wconv \mc N_3(\bm{B}_G, \Sigma_G).
\]
Hence, we may apply Theorem 2.3, Theorem 2.5 and Addendum 2.6 in \cite{BucSeg18a} to the sample $\Xc_n^*$ (unconditionally) to obtain that
\begin{equation}\label{eq:proof_genFreBst-Expa3}
        v_n \begin{pmatrix}
            \hat \alpha_n^\ast  - \alpha_0 \\
            \hat \sigma_n^\ast/\sigma_n - 1
        \end{pmatrix}
        = M(\alpha_0) {\mb G}_n^\circ(f_1,f_2,f_3)^\top + o_\Prob(1) 
\end{equation}
As a consequence, by \eqref{eq:proof_genFreBst-Expa1}, \eqref{eq:proof_genFreBst-Expa3} and $\hat \sigma_n/\sigma_n = 1+ o_\Prob(1)$, we have
\begin{align}
\label{eq:proof_genFreBst-Expa4}
        v_n \begin{pmatrix}
            \hat \alpha_n^\ast  - \hat \alpha_n \\
            \hat \sigma_n^\ast/\hat \sigma_n - 1
        \end{pmatrix}
        & \nonumber = v_n \begin{pmatrix}
            \hat \alpha_n^\ast  - \alpha_0 - (\hat \alpha_n -\alpha_0) \\
            \sigma_n/\hat \sigma_n\big( \sigma_n^\ast/ \sigma_n - 1 - (\hat \sigma_n / \sigma_n - 1) \big)
        \end{pmatrix} \\
        & \nonumber = \{1+o_\Prob(1)\}  \, \big(M(\alpha_0) (\mb G_n^\circ-\Gb_n - B_n)(f_1,f_2,f_3)^\top+o_\Prob(1) \big) \\
        & = M(\alpha_0) \hat\Gb_n^\ast (f_1,f_2,f_3)^\top + o_\Prob(1),
\end{align}
which is \eqref{eq:boot-mle-expansion}.

Finally, for the proof of \eqref{eq:boot-mle-cons}, let $\Xc_n^{\#}$ denote a second bootstrap sample, generated in the same way as $\Xc_n^*$ and independent of $\Xc_n^*$, conditionally on $\Xc_n$. Denote the respective estimators and empirical measures/processes by $\hat \alpha_n^{\#}, \hat \Prob_n^{\#}$ etc. Then, by the expansions in 
\eqref{eq:proof_genFreBst-Expa4},
 \[
    v_n \begin{pmatrix}
            \hat \alpha_n^{\ast} - \hat \alpha_n \\
            \hat \sigma_n^{\ast} / \hat \sigma_n -1 \\
            \hat \alpha_n^{\#} - \hat \alpha_n \\
            \hat \sigma_n^{\#} / \hat \sigma_n -1 \\
    \end{pmatrix}
    =
    \begin{pmatrix}
        M(\alpha_0) (\hat \Gb_n^*)(f_1, f_2, f_3)^\top \\
        M(\alpha_0) (\hat \Gb_n^{\#})(f_1, f_2, f_3)^\top
    \end{pmatrix}
    + o_\Prob(1).
\]
Equation~\eqref{eq:fre-cons} and Lemma 2.3 in \cite{BucKoj2019note} implies that the dominating term on the right-hand side converges weakly (unconditionally) to $\Nc_2(\bm 0, M(\alpha_0)\Sigma_G M(\alpha_0)^\top)^{\otimes 2}$, which by another reverse application of that lemma implies 
\begin{align*}
        d_K\left[  \mc L \left(v_n 
        \begin{pmatrix}
            \hat \alpha_n^\ast  - \hat \alpha_n \\
            \hat \sigma_n^*/\hat \sigma_n -1
        \end{pmatrix}
        \, \Big| \,  \mathcal X_n \right) , 
        \Nc_2(\bm 0, M(\alpha_0)\Sigma_G M(\alpha_0)^\top) \right]
        = o_{\Prob}(1).
\end{align*}
The assertion then follows from the triangular inequality and \eqref{eq:proof_genFreBst-Expa1}, noting that $\bm B_G=\bm 0$ by assumption.
\end{proof}

%%%%%%%%%%%%%%%%%%%%%%%%%%%%%%%%%%%%%%%%%%%%%%%%%%%%%%%%%%%%%%%%%%%%%%%%%%%%%%%%%%
\subsection{Auxiliary results on bootstrap consistency}
\label{sec:aux-boot}

\begin{lemma}\label{lem:ZtoM}
    Let $\bm S_n, \bm T_n \in \R^q$ with $\bm S_n \wconv \bm S, \bm T_n \wconv \bm S \in \R^q.$ Furthermore, let $\Lc(\bm S)$ be absolutely continuous with respect to the Lebesgue measure on $\R^q.$ We then have 
    \[
        \lim_{n \to \infty} d_K \big[ \mc L(\bm A_n \bm S_n ), \mc L (\bm A_n \bm T_n ) \big] = 0,
    \]
    for any sequence $(\bm A_n)_n \subset \R^{p \times q}$ of matrices.
\end{lemma}

\begin{proof}
Write $\bm A_n = (\bm a_{n,j}^\top)_{j=1,\ldots, p}$ with $\bm a_{n,j} \in \R^q.$ 
A straightforward argument shows that we may assume $\bm a_{n,j} \neq \bm 0$ for all $j=1,\ldots p$.
Next, note that, for all diagonal matrices $\bm D = \operatorname{diag}(d_j)_{j=1, \ldots p}$ with $d_j \neq 0$ and all $\R^{p}$-valued random variables $\bm S, \bm T$, we have $d_K[ \mc L(\bm S), \mc L(\bm T)] = d_K[ \mc L(\bm D \bm S), \mc L( \bm D \bm T)]$. Hence,  letting $d_j := \| \bm a_{n,j}\|_2^{-1}$ and $\tilde {\bm A}_n = (\tilde {\bm a}_{n,j}^\top)_{j=1,\ldots, p}$ with normed $\tilde a_{n,j} = \bm a_{n,j} / \| \bm a_{n,j} \|_2 \in \R^p$, we have  $\tilde{\bm A}_n = \bm D \bm A_n$ and therefore
\begin{align}
    \label{eq:kolmo}
        d_K\big[ \mc L(\bm A_n \bm S_n ), 
        \mc L ( \bm A_n \bm T_n ) 
        \big]
        = d_K\big[ \mc L(\tilde{ \bm A}_n \bm S_n ), 
        \mc L ( \tilde{ \bm A}_n \bm T_n ) 
        \big].
\end{align}
Since $\tilde{\bm A}_n \in [-1,1]^{p\times q}$, the Bolzano-Weierstraß Theorem allows to find a subsequence $n^\prime := n^\prime(n)$ such that $\bm E = \lim_{n \to \infty} \tilde{\bm A}_{n^\prime}$ exists. 
Slutsky's lemma then yields the weak convergences $\tilde{\bm A}_{n^\prime} \bm S_{n^\prime} \wconv \bm E \bm S$ and $\tilde{\bm A}_{n^\prime} \bm T_{n^\prime} \wconv \bm E \bm S$.
We will show below that the CDF of $\bm E \bm S$ is continuous. Since the Kolmogorov-metric $d_K$ metrizes weak convergence to limits with continuous CDFs \citep[Lemma 2.11]{Van98}, an application of the triangular inequality implies
 \[
        d_K\big[ \mc L(\tilde{\bm A}_{n^\prime} \bm S_{n^\prime} ), 
        \mc L ( \tilde{\bm A}_{n^\prime} \bm T_{n^\prime} ) 
        \big]
        \leq d_K\big[ \mc L(\tilde{ \bm A}_{n^\prime} \bm S_{n^\prime} ), 
        \mc L (  \bm E \bm S ) 
        \big] + 
         d_K\big[ \mc L (  \bm E \bm S ) , 
         \mc L(\tilde{ \bm A}_{n^\prime} \bm T_{n^\prime} )
        \big] 
        = o(1)
    \]
for $n'\to\infty$. As the previous argumentation can be repeated for subsequences $n^\prime$ of arbitrary subsequences $n^{\prime \prime},$ we then have $d_K\big[ \mc L(\tilde {\bm A}_n \bm S_n ), \mc L ( \tilde {\bm A}_n \bm T_n )\big]= o(1)$, which implies the assertion by \eqref{eq:kolmo}. 

It remains to show that the CDF of $\bm E \bm S$ is continuous. For that purpose, by Sklar's theorem and the fact that copulas are continuous, it is sufficient to show that each marginal CDF of $\bm E \bm S$ is continuous. Hence, fix $j \in \{1, \dots, p\}$, and note that the $j$th row of $\bm E$, say $\bm e_j^\top$, has Euclidean norm 1. 
We may therefore construct an invertible matrix $\tilde{ \bm E} = \tilde{ \bm E}(\bm e_j) \in \R^{q \times q}$ with the first row of $\tilde{\bm E}$ being equal to $\bm e_j^\top.$ Then $\tilde{\bm E}$ being a diffeomorphism and $\bm S$ being absolutely continuous with respect to the Lebesgue-measure on $\R^q$ implies the latter for $\tilde{\bm E} \bm S.$ Hence, for any $y \in \R$,
$
\Prob((\bm E \bm S)_j = y) 
        = \Prob\big( \tilde{\bm E} \bm S \in \{ y\} \times \R^{q-1} \big) = 0.
$
Since $y\in\R$ was arbitrary, this proves the continuity of $y \mapsto \Prob((\bm E \bm S)_j \le  y).$
\end{proof}

\begin{lemma}\label{lem:ZBstToMBst}
    Let $(\Omega, \mc A, \Prob)$ denote a probability space and $p,q\in\N$. For $n \in \N,$ let $\bm X_n \colon \Omega \to \mc X_n, \, \bm W_n \colon \Omega \to \mc W_n$ denote random variables in some measurable space $\mc X_n, \mc W_n,$ respectively. Let $\bm S_n = \bm S_n(\bm X_n)$ and $\bm S_n^{\ast} = \bm S_n(\bm X_n, \bm W_n)$ be $\R^q$-valued statistics. If  
    \begin{alignat*}{3}
        &(a) \hspace{.5cm} &&d_w\big( \Lc(\bm S_n^\ast \mid \bm X_n), \Lc(\bm S_n) \big) \pto 0, \qquad &&\text{ as } n \to \infty; \\
        &(b)  \hspace{.5cm}  &&\bm S_n \wconv Q, \qquad &&\text{ as } n \to \infty;
    \end{alignat*}
    where $Q$ is absolutely continuous with respect to the Lebesgue-measure on $\R^q$ and $d_w$ denotes any metric characterizing weak convergence on $\R^q$;  then 
    \[
        d_K\big( \Lc(\bm A_n \bm S_n^\ast \mid \bm X_n), \Lc(\bm A_n \bm S_n) \big) \pto 0, \qquad \text{ as } n \to \infty,
    \]
    for any sequence $(\bm A_n)_n \subset \R^{p \times q}$ of matrices.
\end{lemma}
\begin{proof}
Let $n^\prime \subset \N$ denote a subsequence of $\N.$ 
By assumption (a), we may choose a further subsequence $n^{\prime \prime}$ of $n^\prime$ and $\Omega_0 \in \Ac$ with $\Prob(\Omega_0) = 1$ such that 
\[
        \lim_{n \to \infty} d_w\big( \Lc(\bm S_{n^{\prime \prime}}^\ast \mid  \bm X_{n^{\prime \prime}}), \Lc(\bm S_{n^{\prime \prime}}) \big) = 0  \quad \text{ on } \Omega_0.
\]
Hence, since $\bm S_n \wconv Q$ by assumption, we obtain $\Lc(\bm S_{n^{\prime \prime}}^\ast \mid \bm X_{n^{\prime \prime}}) \wconv Q$ on $\Omega_0.$ Lemma \ref{lem:ZtoM} then implies 
\[
        \lim_{n \to \infty} d_K\big( \Lc( \bm A_{n^{\prime \prime}} \bm S_{n^{\prime \prime}}^\ast \mid \bm X_{n^{\prime \prime}}), \Lc(\bm A_{n^{\prime \prime}}\bm S_{n^{\prime \prime}}) \big) = 0 \quad \text{ on } \Omega_0,
\]
which lets us conclude.
\end{proof}

\subsection{A result on circmax biases}
\label{sec:circmaxbias}
Recall the definition of the bias components $\bm D_{\bm h,k}, \bm E_{\bm h}$ from \refstar{eq:cond_bias_cbl} which were introduced by the circmax operation. In this section we present general conditions which imply that both components vanish.

\begin{lemma}\label{lem:Eh_zero} 
Fix $k \in \N_{\geq 2}$ and suppose that Condition \refstar{cond:mda} and \refstar{cond:ser_dep} are met.
Then $\bm D_{\bm h,k} = \bm E_{\bm h} = 0$, provided the time series $(\bm X_t)_t$ is exponentially $\beta$-mixing (i.e., there exists $c>0, \lambda \in (0,1)$ such that $\beta(\ell) \le c \lambda^\ell$ for all $\ell \in \N$), that $\log n = o(r^{1/2})$ and that Condition \refstar{cond:int_h}(b) and
    \begin{align} 
    \label{eq:inttt}
    & \limsup_{n \to \infty} \sup_{t=1,\ldots,r}  \Exp\Big[\Big\|h\Big(\frac{\max (\bm X_1, \dots, \bm X_t, \tilde{\bm X}_{1},  \dots, \tilde{\bm X}_{r-t}) - \bm b_r}{ \bm a_r }\Big)\Big\|^{2+\nu}\Big] 
    < \infty \\
    &\label{eq:inttt2} \limsup_{n \to \infty} \sup_{t=\flo{r^{1/2}}, \ldots, r} \Exp\Big[\Big\|h\Big(\frac{\max (\bm X_1, \dots, \bm X_{t-\flo{r^{1/2}}}, \tilde{\bm X}_{1},  \dots, \tilde{\bm X}_{r-t}) - \bm b_r}{ \bm a_r }\Big)\Big\|^{2+\nu}\Big] 
    < \infty \\
    &\label{eq:inttt3} \limsup_{n \to \infty} \sup_{t=\flo{r^{1/2}}, \ldots, r} \Exp\Big[\Big\|h\Big(\frac{\max (\bm X_1, \dots, \bm X_{t-\flo{r^{1/2}}}, \bm X_{t + 1}, \ldots, \bm X_{r}) - \bm b_r}{ \bm a_r }\Big)\Big\|^{2+\nu}\Big] 
    < \infty
    \end{align}
    hold for some $\nu>2/\omega$ with $\omega$ from Condition~\refstar{cond:ser_dep}, where $(\tilde{\bm X}_t)_t$ is an independent copy of $(\bm X_t)_t$. 
\end{lemma}

\begin{proof}
Without loss of generality we may assume that $q = 1.$ First consider $D_{h,k}$ and define, for $t \in \{1,\ldots, r-1\}$,
\[
    D_{n,h,k,t} = 
    h \Big( \frac{\max ( \bm X_1, \dots,  \bm X_t,  \bm X_{{(k-1)r+t+1}},  \dots,  \bm X_{kr}) -   \bm b_r}{  \bm a_r} \Big)
    -  h \Big( \frac{\max ( \bm X_1, \dots,  \bm X_t, \tilde{\bm X}_{1},  \dots, \tilde{\bm X}_{r-t}) -  \bm b_r}{\bm a_r} \Big).
\]
Note that showing $\sup_{t=1,\ldots,r-1} (n/r)^{1/2} |\Exp[D_{n,h,k,t}]| = o(1)$ implies $D_{h,k}= 0.$ 
By applying Berbee's coupling Lemma \citep{Ber79} for each fixed $t \in\{ 1, \ldots, r-1\}$ to the vectors $(\bm X_1,\ldots, \bm X_t)$ and $(\bm X_{(k-1)r+t+1}, \ldots, \bm X_{kr})$ we may assume that the random vector $(\tilde{\bm X}_{1}, \ldots, \tilde{\bm X}_{r-t})$ satisfies 
   $         \Prob\big((\tilde{\bm X}_{1}, \ldots, \tilde{\bm X}_{r-t}) \neq  (\bm X_{(k-1)r+t+1}, \ldots, \bm X_{kr}) \big)\leq \beta(r)
\leq c \lambda^r;$ where the last inequality follows by assumption. Thus, by Hölder's inequality and \eqref{eq:inttt},
\[
    \sup_{t=1,\ldots, r-1} \sqrt{\frac{n}{r}} \Exp[|D_{n,h,k,t}|]
    \lesssim 
    \sqrt{\frac{n}r} \lambda^{r(1+\nu)/(2+\nu)} = \exp \Big[r \Big( \log\lambda \frac{1+\nu}{2+\nu} + \frac{\log(n/r)}{2r}\Big) \Big].
\]
The expression on the right converges to zero since $\log n = o(r)$. As a consequence, $D_{h,k}=0$.

It remains to show that $E_h=0$. Writing 
\[
E_{n,h,t} = 
     h \Big( \frac{\max ( \bm X_1, \dots,  \bm X_t, \tilde{\bm X}_{1},  \dots, \tilde{\bm X}_{r-t}) - \bm b_r}{ \bm a_r} \Big) - h(\bm Z_{r,1}).
\]
it is sufficient to show that $\sup_{t=1,\ldots,r-1} (n/r)^{-1/2} |\Exp[E_{n,h,t}]| = o(1)$. For that purpose, we split the supremum according to $t> r/2$ or $t \le r/2$; both cases can then be treated similarly, and we only provide details on the former one. For simplicity, we assume that $r/2\in\N$.

Write $\ell = \flo{r^{1/2}}$. The proof is finished once we show that
\begin{align}
\nonumber  
    &\sup_{t = r/2, \dots, r} \Big| \sqrt{\frac{n}{r}} \Exp\Big[h \Big( \frac{\max ( \bm X_1, \dots,  \bm X_t, \tilde{\bm X}_{1},  \dots, \tilde{\bm X}_{r-t}) - \bm b_r}{ \bm a_r} \Big) \\ 
\label{proof:Eh_zero_eq1} &\hspace{5cm} - 
    h\Big(\frac{\max(\bm X_1,\ldots,\bm X_{t-\ell}, \tilde{\bm X}_{1},\ldots, \tilde{\bm X}_{r-t}) - \bm b_r}{\bm a_r}\Big)\Big] \Big| = o(1) \\
\nonumber    
    &\sup_{t = r/2, \dots, r} \Big| \sqrt{\frac{n}{r}} \Exp\Big[h \Big( \frac{\max ( \bm X_1, \dots,  \bm X_{t-\ell}, \tilde{\bm X}_{1},  \dots, \tilde{\bm X}_{r-t}) - \bm b_r}{ \bm a_r} \Big) \\
\label{proof:Eh_zero_eq2} &\hspace{5cm} - 
    h\Big(\frac{\max(\bm X_1,\ldots,\bm X_{t-\ell}, {\bm X}_{t+1},\ldots, {\bm X}_r) - \bm b_r}{\bm a_r}\Big)\Big] \Big| = o(1) \\
\label{proof:Eh_zero_eq3}    
&\sup_{t = r/2, \dots, r} \Big| \sqrt{\frac{n}{r}} \Exp\Big[ 
    h\Big(\frac{\max(\bm X_1,\ldots,\bm X_{t-\ell}, \bm X_{t+1},\ldots, \bm X_r) - \bm b_r}{\bm a_r}\Big) - h(\bm Z_{r,1})\Big] \Big| = o(1).
\end{align}
The proof of \eqref{proof:Eh_zero_eq2} is similar to the proof of $D_{h,k}=0$ given above (invoking \eqref{eq:inttt2} instead of \eqref{eq:inttt}), and is therefore omitted (the final argument requires $\log n = o(r^{1/2})$, which is then exactly met by assumption). The proof of \eqref{proof:Eh_zero_eq2} is a simplified version of the proof of \eqref{proof:Eh_zero_eq3}, so we only prove the latter for the sake of brevity. In view of Condition \refstar{cond:int_h}(b) and \eqref{eq:inttt3}, an application of Hölder's inequality implies
\begin{align*}
    & \phantom{{}={}} \sup_{t = r/2, \dots, r} \Big| \sqrt{\frac{n}{r}} \Exp\Big[ 
    h\Big(\frac{\max(\bm X_1,\ldots,\bm X_{t-\ell}, \bm X_{t+1},\ldots, \bm X_r) - \bm b_r}{\bm a_r}\Big) - h(\bm Z_{r,1}) \Big] \Big| \\
    &\lesssim \sup_{t = r/2, \dots, r} \sqrt{\frac{n}{r}} \Big\{ \Prob\big( \exists j \in \{1, \dots, d\}: \max(X_{t-\ell+1,j},\ldots, X_{t,j}) > \max(X_{1,j},\ldots,X_{t-\ell,j}, X_{t+1,j},\ldots, X_{r,j})\big)\Big\}^{\frac{1+\nu}{2+\nu}}.
\end{align*}
By the union-bound, it is sufficient to show that
\begin{align} \label{eq:rn}
R_n \nonumber & \equiv \sup_{t = r/2, \dots, r} \sqrt{\frac{n}r} \max_{j=1}^d \Big\{ \Prob\big(\max(X_{t-\ell+1,j},\ldots, X_{t,j}) > \max(X_{1,j},\ldots,X_{t-\ell,j}, X_{t+1,j},\ldots, X_{r,j})\big) \Big\}^{\frac{1+\nu}{2+\nu}} \\
&= o(1).
\end{align}
Subsequently, we write $X_t$ instead of $X_{t-j}$, as all subsequent bounds are uniform in $j$. We then have
\begin{align*}
    &\phantom{{}={}} \Prob\big( \max(X_{t-\ell + 1},\ldots, X_t) > \max(X_1,\ldots,X_{t-\ell}, X_{t+1},\ldots, X_r)\big) \\
    &\leq \sum_{i=t-\ell+1}^t \Prob \big( X_i > \max(X_1,\ldots,X_{t-\ell}, X_{t+1},\ldots, X_r)\big).
\end{align*}
For fixed $i$ in the previous sum, let $J \subset \{1, \dots, t-\ell\} $ denote the maximal set of indexes such that $|j_1 - i|\ge \ell$ and $|j_1 - j_2|\ge \ell$ for all distinct $j_1, j_2 \in J$;
note that $|J|=O(r/\ell)$ since $t> r/2$. We then have
\[
\Prob \big( X_i > \max(X_1,\ldots,X_{t-\ell}, X_{t+1},\ldots, X_r)\big)
\le
\Prob \big( X_i > \max(X_j: j \in J)\big).
\]
We may now successively apply Berbee's coupling lemma \citep{Ber79} to construct a vector $(\check X_{j})_{j \in J}$ with iid coordinates that is independent of $X_i$ and satisfies $\check X_j=_d X_j$ for all $j\in J$ such that
\[
\Prob \big( X_i > \max(X_s: s \in J)\big) \le \Prob\big(X_i > \max(\check X_s: s \in J)\big) + |J| \beta(\ell).
\]
More precisely, writing $J=\{j_1, \dots, j_{|J|}\}$, the first application is to $X_{j_1}$ and $(X_{j_2}, \dots, X_{j_{|J|}}, X_i)$. The second application is to $(\check X_{j_1}, X_{j_2})$ and $(X_{j_3}, \dots, X_{j_{|J|}}, X_i)$, where $\check X_{j_1}$ is the random variable constructed with the first application, and so on.

Since $\Prob \big( X_i > \max(\check X_s: s \in J)\big)=2^{-|J|}$ by Fubini's theorem, we obtain from the last three displays that
\[
\sup_{t=r/2, \dots, r}\Prob\big( \max(X_{t-\ell + 1},\ldots, X_t) > \max(X_1,\ldots,X_{t-\ell}, X_{t+1},\ldots, X_r)\big)
\le \ell 2^{-|J|} + \ell|J|\beta(r) \lesssim \ell 2^{-r/\ell} + r \lambda^\ell.
\]
In view of our choice of $\ell=\flo{r^{1/2}}$, and letting $\zeta = \min(2,1/\lambda)>1$, the right hand side is bounded by $2 r \zeta^{-\sqrt r}$, whence $R_n$ from \eqref{eq:rn} can be bounded by
\[
R_n \lesssim  \sqrt{\frac{n}r} \big( r \zeta^{-\sqrt r} \big)^{(1+\nu)/(2+\nu)} 
=
\exp\Big\{ - \sqrt r  \Big( \log(\zeta) \frac{1+\nu}{2+\nu} - \frac{\log r}{\sqrt r} \frac{1+\nu}{2+\nu} - \frac{\log(n/r)}{\sqrt r} \Big) \Big\} = o(1)
\]
by assumption on $r$. This proves \eqref{eq:rn}, and the proof is finished.
\end{proof}

\section{Resampling sliding block maxima vs.\ resampling circular block maxima}
\label{sec:inconsistency-sliding}

In this section, we provide an intuitive heuristic for the inconsistency of naive resampling of sliding block maxima and explain why the method of resampling circular block maxima solves the problem.
We thereby extend the brief discussion in the introduction and in Remark~\refstar{rem:slidboot-inconsistent} from the main article. At the end of this section, we formally prove the statements made in Remark~\refstar{rem:slidboot-inconsistent}.

For simplicity, assume $m=m(2) = n/(2r) \in \N$, and write $I_i = I_{2r,i}=\{(i-1)2r+1, \dots, 2ri\}$ for the $i$th disjoint block of size $2r$. We are interested in the affinely standardized empirical mean
\[
\bar{\Gb}_{n,r}^{[\sbl]}(h) = \sqrt{n/r} \Big\{ \frac1n \sum_{s=1}^n h(Z_{r,s}^{[\sbl]}) - \Exp[h(Z_{r,s}^{[\sbl]})] \Big\}
=
\sqrt{n/r} \frac1m \sum_{i=1}^m D_{n,i}^{[\sbl]}
\]
and in its naive sliding and circular blocks bootstrap analogs
\[
\hat{\Gb}^{[\sbl],*}_{n,r}(h) 
= \sqrt{n/r} \frac1m \sum_{i=1}^m Y_{n,i} D_{n,i},
\qquad 
\hat{\Gb}^{(\cbl),*}_{n,r}(h)
=\sqrt{n/r} \frac1m \sum_{i=1}^m Y_{n,i} D_{n,i}^{(\cbl)};
\]
here, $D_{n,i}^{[\sbl]}=(2r)^{-1} \sum_{s \in I_i} h(Z_{r,s}^{[\sbl]}) - \Exp[h(Z_{r,s}^{[\sbl]})]$, $D_{n,i}^{(\cbl)}=(2r)^{-1} \sum_{s \in I_i} h(Z_{r,s}^{(\cbl)}) - \Exp[h(Z_{r,s}^{(\cbl)})]$ and $(Y_{n,i})_i$ are iid multipliers with expectation 0 and variance 1. Writing $\sigma_r(s,t) = \Cov\big(  h(Z_{r,s}^{[\sbl]}) ,  h(Z_{r,t}^{[\sbl]}) \big)$, the variance of $\Gb_{n,r}^{[\sbl]}(h)$ can be written as
\begin{align}
\Var\big( \bar \Gb_{n,r}^{[\sbl]}(h) \big)
=
\frac{1}{nr} \Var \Big( \sum_{i=1}^m \sum_{s \in I_i} h(Z_{r,s}^{[\sbl]}) \Big) 
&=\nonumber
\frac{1}{nr} \sum_{i,j=1}^m \sum_{s \in I_i} \sum_{t \in I_j}\sigma_r(s,t) \\
& \approx \nonumber
\frac{1}{nr} \sum_{i=1}^m  \sum_{s \in I_i} \sum_{t=s-r+1}^{s+r-1}  \sigma_r(s,t) \\
&= \label{eq:var-sb}
\frac{1}{2r^2} \sum_{s=1}^{2r} \sum_{t=s-r+1}^{s+r-1}  \sigma_r(s,t),
\end{align}
where, at the $\approx$-sign, we have used asymptotic independence of two sliding block maxima in case their respective index sets do not overlap and where, for the ease of notation, we have extended the available sample by $r$  observations at the beginning and the end of the sampling period of length $n$ (which does not matter asymptotically). Moreover, we used stationarity at the last equation. The index set of the double sum in the last line, say $B_r$, is illustrated in green in the left panel of Figure~\ref{fig:index-summation}.

Next, since $(Y_{n,i}D_{n,i})_i$ is uncorrelated, we have
\begin{align}
\Var\big( \hat{\Gb}_{n,r}^{[\sbl],*}(h) \big)
=
\frac{n}{mr} \Var \big( D_{n,1} \big) 
&=
\frac1{2r^2} \sum_{s =1}^{2r}\sum_{t =1}^{2r}  \sigma_r(s,t)
\approx \label{eq:var-sb-boot}
\frac1{2r^2} \sum_{(s,t) \in D_r}\sigma_r(s,t) ,
\end{align}
where 
\[
D_r =\{ (s,t) \in \{1, \dots, 2r\}^2: |s-t|\le r\}
\]
and where we have again used asymptotic independence of two sliding block maxima in case their respective index sets do not overlap (at the $\approx$-sign); see the middle panel of Figure \ref{fig:index-summation} for an illustration of the index set $D_r$. Compared to \eqref{eq:var-sb}, we observe that the summands with index $(s,t) \in B_r$ such that either $t>2r$ or $t \le0$ are missing in \eqref{eq:var-sb-boot}.
This makes a significant difference asymptotically, and ultimately yields inconsisteny of the naive sliding blocks approach.

For the circ-max approach,
writing $ \sigma^{(\cbl)}_r(s,t) := \Cov(h(Z_{r,s}^{(\cbl)}), h(Z_{r,t}^{(\cbl)}))$, we have 
\begin{align*}
\Var\big( \hat{\Gb}^{(\cbl),*}_{n,r}(h) \big)
=
\frac{n}{mr} \Var \big( D_{n,1}^{(\cbl)} \big) 
&=
\frac1{2r^2}\sum_{s =1}^{2r}\sum_{t =1}^{2r}\sigma_r^{(\cbl)}(s,t)\\
&=
\frac1{2r^2}\Big( \sum_{(s,t) \in D_r} + \sum_{(s,t) \in E_r^+} + \sum_{(s,t) \in E_r^-} \Big)\sigma^{(\cbl)}_r(s,t), 
\end{align*}
where
\begin{align*}
    E_r^+ &= \{ (s,t) \in \{1, \dots, 2r\}^2: s-t >r\}, \\
    E_r^- &= \{ (s,t) \in \{1, \dots, 2r\}^2: t-s >r\};
\end{align*}
see the blue sets in the right panel in Figure~\ref{fig:index-summation}.
Now, by Proposition~\ref{prop:overlap_wconv_kloop} and stationarity,
\begin{align*}
\begin{array}{lllll}
\sigma^{(\cbl)}_r(s,t) &\approx  \sigma_r(0, |s-t|) &=  \sigma_r(s,t)  &  \qquad & \text{ if } (s,t) \in D_r ,\\
\sigma^{(\cbl)}_r(s,t) &\approx \sigma_r(0, 2r - (s-t)) &= \sigma_r(s, 2r+t) &  \qquad & \text{ if }  (s,t) \in E_r^+ ,\\
\sigma^{(\cbl)}_r(s,t) &\approx \sigma_r(0, 2r - (t-s)) &= \sigma_r(s, t-2r) &  \qquad & \text{ if } (s,t) \in E_r^-.
\end{array}
\end{align*}
As a consequence, we can match the blue triangles  in Figure~\ref{fig:index-summation} with the red ones by an index shift:
\begin{align*}
\Var\big( \tilde{\Gb}^{(\cbl)}_{n,r}(f) \big)
&\approx
\frac1{2r^2} \Big\{ \sum_{(s,t) \in D_r} \sigma_r(s,t) + \sum_{(s,t) \in E_r^+}\sigma_r(s,2r+t) + \sum_{(s,t) \in E_r^-} \sigma_r(s,t-2r) \Big\} \\
&=
\frac1{2r^2} \Big\{ \sum_{(s,t) \in D_r} \sigma_r(s,t) + \sum_{(s,t) \in \tilde E_r^+}\sigma_r(s,t) + \sum_{(s,t) \in \tilde E_r^-} \sigma_r(s,t) \Big\} \\
&=
\frac{1}{2r^2} \sum_{s=1}^{2r} \sum_{t=s-r+1}^{s+r-1}  \sigma_r(s,t).
\end{align*}
This is the same expression as in \eqref{eq:var-sb}, and heuristically explains consistency of the circmax-approach. \qed

\begin{figure}
    \centering
    \includegraphics[width=0.99\textwidth]{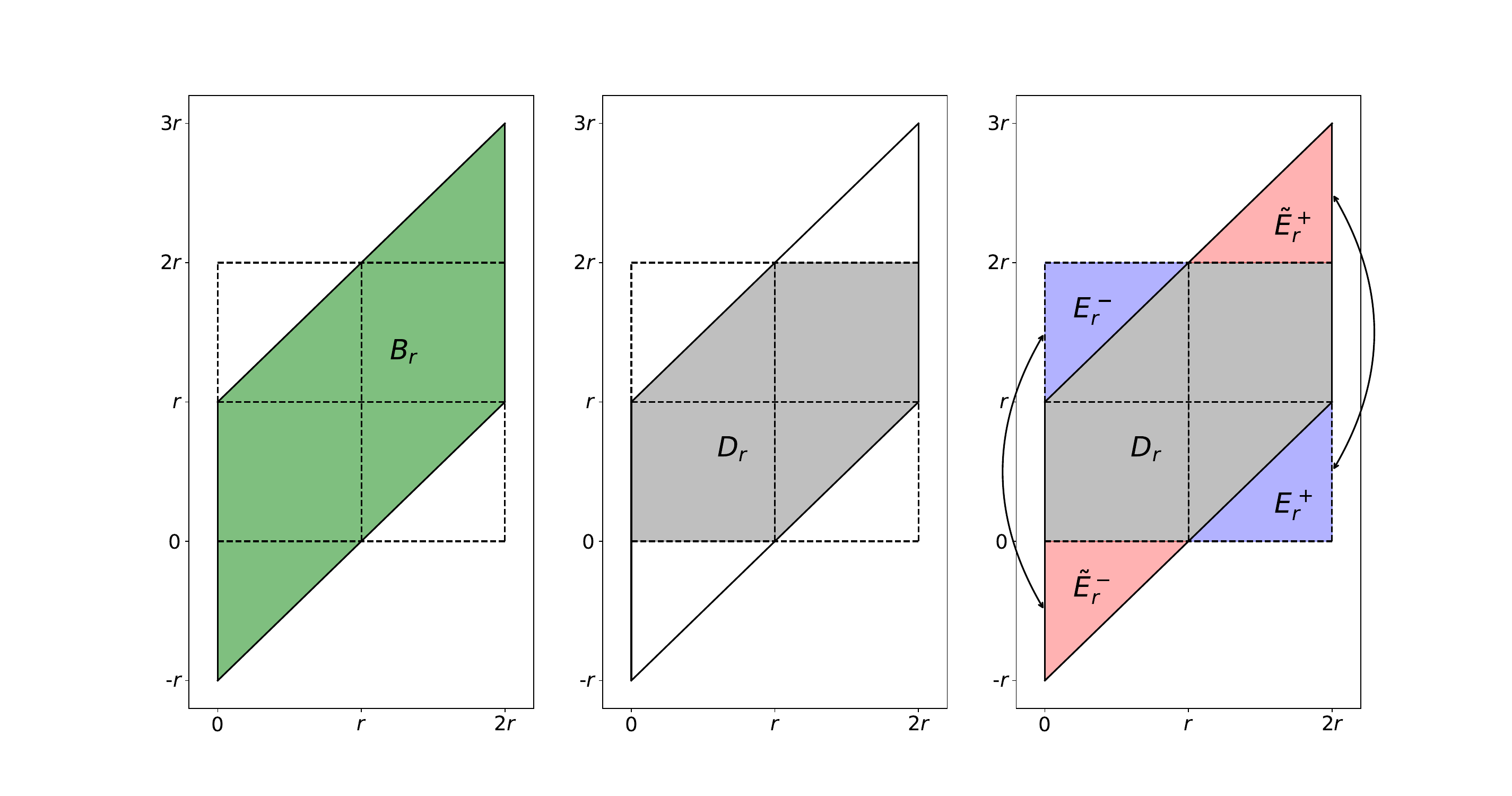}
    \caption{Left: index set for the sliding blocks variance. Middle: index set for the naive sliding blocks bootstrap variance. Right: index sets for the circular blocks bootstrap variance with $k=2$.}
    \label{fig:index-summation}
\end{figure}

\begin{proof}[Proof for Remark~\refstar{rem:slidboot-inconsistent}]
We sketch the proof of the conditional weak convergence in~(\refstar{eq:slidboot-inconsistent}) provided that Conditions~\refstar{cond:mda}, \refstar{cond:ser_dep} and~\refstar{cond:int_h}(a) are met.
Repeating the arguments from the proof of Theorem~\refstar{theo:mult_boot}, the assertion follows from unconditional weak convergence of 
\[
\tilde \Gb_{n,r}^{[\sbl],*} h = \sqrt{\frac{n}r} \frac1{m(k)} \sum_{i=1}^{m(k)} Y_{n,i} D_{n,i},
\]
where $(Y_{n,i})_i$ are iid with expectation 0 and variance 1, and where $D_{n,i}$ are defined as in \eqref{eq:dni}, but with $\sbl$ instead of $\cblk$. We only provide the proof for the convergence of the variance of $\tilde \Gb_{n,r}^{[\sbl],*} h$ to $\Sigma_h^{(k)}$; the remaining arguments are then the same as in the proof of Lemma~\ref{lem:aux_mult_clt}.

Since $(Y_{n,i}D_{n,i})_i$ in uncorrelated, we obtain
\[
\Var(\tilde \Gb_{n,r}^{[\sbl],*)} h) = k \Var(D_{n,1}) = \frac1{kr^2} \sum_{s=1}^{kr} \sum_{t=1}^{kr} \Cov(h(\bm Z_{r,s}^{[\sbl]}), h(\bm Z_{r,t}^{[\sbl]}))
\]
Using the notation after \eqref{eq:proof_kloop_sliding_clt_3}, we obtain 
\[
\lim_{n\to\infty} \Var(\tilde \Gb_{n,r}^{[\sbl],*)} h) 
=
\frac1k \int_0^k \int_0^k g(\xi, \xi') \diff \xi'\diff \xi .
\]
We start by considering the case $k\ge 2$. Then, by the calculation in \eqref{eq:intdec2} but with $g_{(k)}$ replaced by $g$,
\[
\int_0^k \int_\xi^k g(\xi, \xi') \diff \xi'\diff \xi 
=
I_1 + I_2
\]
where
\begin{align*}
I_1 &= (k-1) \Sigma_h^{[\sbl]} / 2 + \int_0^1 \int_0^y u(x) \diff x \diff y,
\qquad
I_2 = \int_0^1 \int_{k-1+\xi}^k u(\xi'-\xi) \diff \xi'\diff \xi = 0.
\end{align*}
The first integral over $u$ can be rewritten as
\begin{align}
\label{eq:intu}
\int_0^1 \int_0^y u(x) \diff x \diff y = \int_0^1 \int_x^1 u(x) \diff y \diff x
=
\int_0^1 (1-x) u(x) \diff x = \Sigma_h^{[\sbl]}/2 - \int_0^1 x u(x) \diff x,
\end{align}
so that
\[
I_ 1= k \Sigma_h^{[\sbl]} / 2 - \int_0^1 x u(x) \diff x.
\]
For symmetry reasons, we also have
$
\int_0^k \int_0^\xi g(\xi, \xi') \diff \xi'\diff \xi = \int_0^k \int_\xi^k g(\xi, \xi') \diff \xi'\diff \xi 
$
so that 
\[
\Var(\tilde \Gb_{n,r}^{[\sbl],*} h) = k \Var(D_{n,1}) =   \Sigma_h^{[\sbl]} - \frac2k \int_0^1 x u(x) \diff x +o(1)
\]
as asserted.

It remains to consider the case $k=1$. A straightforward calculation then shows that
\[
\int_0^1 \int_\xi^1 g(\xi, \xi') \diff \xi'\diff \xi 
=
\int_0^1 \int_{\xi}^1 u(\xi'-\xi) \diff \xi'\diff \xi = \int_0^1 \int_0^{1-\xi} u(y) \diff y \diff \xi =  \int_0^1 \int_0^{x} u(y) \diff y \diff x,
\]
which implies the assertion by \eqref{eq:intu}.
\end{proof}

%%%%%%%%%%%%%%%%%%%%%%%%%%%%%%%%%%%%%%%%%%%%%%%%%%%%%%%%%%%%%%%%%%%%%%%%%%%%%%%%%%
%\appendix
\section{Details on mean estimators in the ARMAX-GPD model}
\label{sec:armax-mu}

\begin{corollary}
\label{cor:mean_est}
Suppose $(X_t)_{t\in\Z}$ is an ARMAX-GPD time series as in Model~\refstar{mod:armax-gpd}, for some $\beta \in [0,1)$ and some $\gamma<1/2$.
Then, if the block size parameter satisfies $r = o(n)$ and $\log n = o(r^{1/2})$,
\begin{equation*}
\frac{\sqrt{n/r}}{(r(1-\beta))^\gamma} \big( \hat \mu_n^{[\mbl]} - \mu_r \big)
\wconv \mc N(0, \sigma^2_{\mbl}), \quad \mbl \in \{\dbl, \sbl, \cblk \},
\end{equation*}
where the asymptotic variance is given by
\begin{align}\label{eq:meanEstAsyVar}
    \sigma^2_{\dbl} = 
    \begin{cases}
        \frac{g_2 - g_1^2}{\gamma^2 } , & \gamma < 1/2, \gamma \neq 0, \\
        \frac{\pi^2}{6}, & \gamma = 0,
    \end{cases}
    \qquad
    \sigma^2_{\cblk(k)} = \sigma^2_{\sbl} = 
    \begin{cases}
        4\Gamma(-2\gamma)I(\gamma), & \gamma < 0,\\
        4(\log 4 -1), &\gamma = 0, \\
        -\frac{g_2}{\gamma} I(\gamma) & 0 < \gamma < 1/2.
    \end{cases}
\end{align}
Here, $\Gamma$ denote the gamma function, $g_j = \Gamma(1-j\gamma)$ and 
$
    I(\gamma) := 2 \int_0^{1/2} \{ \alpha_{2\gamma}(w) -1 \} \{ w^{-\gamma-1}(1-w)^{-\gamma-1} \} \diff w
$
with $\alpha_c(w) = w^{-1} \int_0^w (1-z)^c \diff z$. We have $\sigma^2_{\cblk}=\sigma^2_{\sbl} < \sigma^2_{\dbl}$, and the ratio $\gamma \mapsto \sigma^2_{\dbl} / \sigma^2_{\sbl}(\gamma)$ is presented in Figure \ref{fig:asyVarRatioMean}. 
\end{corollary}

\begin{figure}[h!] 
\centering
\makebox{\includegraphics[width=0.6\textwidth]{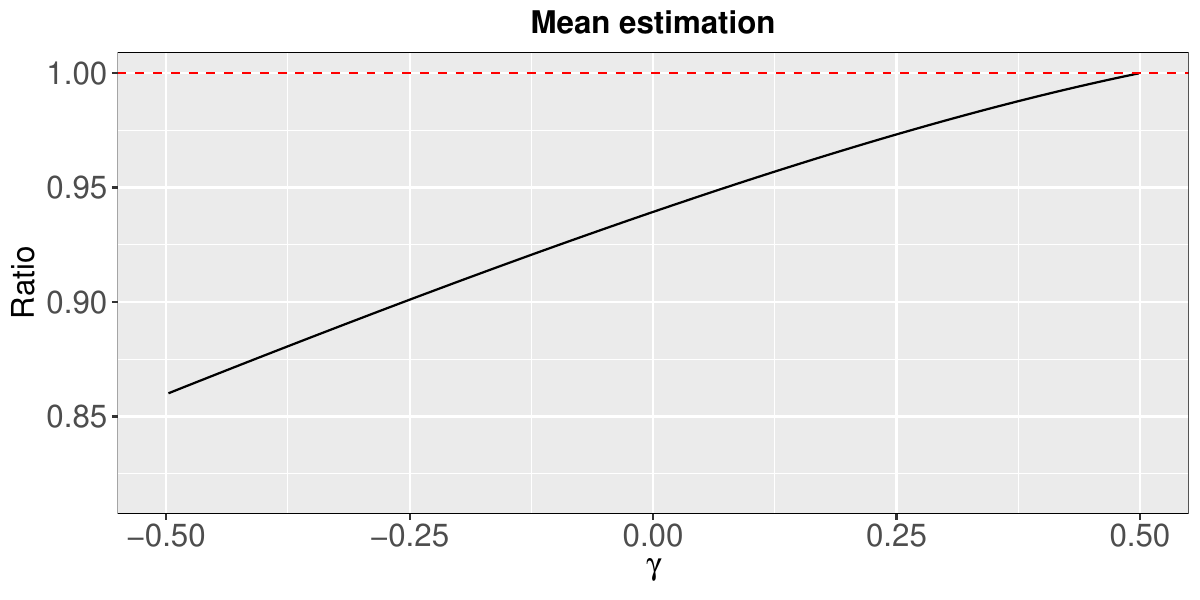}}\vspace{-.1cm}
\caption{Ratio of the asymptotic variances $\sigma^2_{\dbl} / \sigma^2_{\sbl}$ from (\ref{eq:meanEstAsyVar}).
}\label{fig:asyVarRatioMean}	
\vspace{-.3cm}
\end{figure}

\begin{proof}
The proof is akin to the one of Equation (4.3) in \cite{BucSta24}; most details are omitted for the sake of brevity and we only discuss the calculation of the asymptotic variance parameter $\sigma^2_\sbl$ (note that $\sigma^2_\dbl$ simply corresponds to the variance of the $\mathrm{GEV}(\gamma)$-distribution by definition).

We start by considering the case $\gamma \neq 0.$ For $i\in\{1,2\}$, the random variables $S_{i,\xi} := (1+\gamma Z_{i,\xi})^{1/\gamma}$ are standard exponentially distributed, and we have $\Cov(Z_{1,\xi}, Z_{2,\xi}) = \gamma^{-2} \Cov(S_{1,\xi}^{-\gamma}, S_{2,\xi}^{-\gamma}) =: \gamma^{-2} C_\xi.$ For $\gamma < 0,$ Equation (C.8) in the supplement of \cite{BucSta24} implies that
\[
    \int_0^1 C_\xi \diff \xi 
    = 2\gamma^2 \Gamma(-2\gamma) \int_0^{1/2}(\alpha_{2\gamma}(w)-1) \big(w^{-\gamma -1}(1-w)^{-\gamma-1} \diff w) 
    = 2\gamma^2 \Gamma(-2\gamma)I(\gamma),
\]
hence $\sigma^2_{\sbl} = 4 \Gamma(-2\gamma) I(\gamma).$ For $\gamma > 0,$  the display above (C.9) in the same reference implies
\[
    \int_0^1 C_\xi \diff \xi
    = - \gamma \frac{\Gamma(1-2\gamma}{2}I(\gamma).
\]
Hence, $\sigma^2_{\sbl} = -\Gamma(1-2\gamma)/\gamma I(\gamma).$ 

Finally, for $\gamma = 0$ consider the transformed random variables $S_{i,\xi} = \exp(-Z_{i,\xi})$ for $i \in \{1,2\}$, which gives $\Cov(Z_{1,\xi}, Z_{2,\xi}) = \Cov(\log S_{1,\xi}, \log S_{2,\xi}) = C_\xi.$ By formula (C.11) from the same reference we have
\[
    \int_0^1 C_\xi \diff \xi 
    = \int_0^1 \frac{1}{w(1-w)} \int_0^1 -\log(A_\xi(w)) \diff \xi \diff w,
\]
where $A_\xi(w) = \xi + (1-\xi) \big(w \wedge (1-w)\big)$ is the Pickands-dependence function of the associated Copula $C_\xi$ of the Marshall-Olkin distribution with dependence parameter $\xi.$ It follows that 
\begin{equation*}
    \int_0^1 C_\xi \diff \xi 
    = 2 \int_{1/2}^1 \frac{1}{w(1-w)} \int_0^1 -\log(\xi + (1-\xi)w) \diff \xi \diff w
    = \log 4 - 1,
\end{equation*}
which implies the asserted formula.
\end{proof}

\section{Additional simulation results}
\label{sec:sim-add}

\subsection{Fixed block size: mean estimation}
\label{sec:sim-fixr-mean}

Throughout, we consider exactly the same setting as in Section~\refstar{sec:sim} from the main paper, but with the most vanilla target parameter: the expected value $\mu_r = \Exp[M_{r}]$ of the block maximum distribution. 
The respective disjoint, sliding and circular block maxima estimators are $\hat \mu_n^{[\mbl]} := n^{-1} \sum_{i=1}^n M_{r,i}^{[\mbl]}$ with $\mbl \in \{\dbl, \sbl, \cblk\}$, where $\cblk$ denotes the circmax estimator with parameter $k=2$. Recall that the asymptotic behavior of these estimators within the ARMAX-GPD-Model from Model~\refstar{mod:armax-gpd} is explicitly given in Corollary~\ref{cor:mean_est}.

The results are summarized in several figures that directly correspond to figures in the main paper for return level estimation:
\begin{compactitem}
\item The three estimators are compared in terms of their MSE in Figure~\ref{fig:meanRelMseBias}; this is akin to Figure~\refstar{fig:RlRelMseBias} from the main paper.
\item The three bootstrap approaches are compared in terms of histograms and qq-plots of their estimation errors in Figure~\ref{fig:meanHistQQ}; this is akin to Figure~\refstar{fig:RlHists}  from the main paper.
\item The three three bootstrap approaches are compared in terms of their ability to provide reliable estimators for the estimation variance; see Figure~\ref{fig:meanBstVar}, which is akin  to Figure~\refstar{fig:rlBstVar} from the main paper.
\item Results that allow for a comparison of the raw basic bootstrap confidence intervals based on the disjoint and the sliding-circular blocks approaches are presented in Figure~\ref{fig:meanBstCiVar};  this is akin to Figure~\refstar{fig:rlBstCiVar} from the main paper.
\item Finally, results on the data-adaptive size-correction are presented in Figure~\ref{fig:meanBstFacoCi}; this is akin to Figure~\refstar{fig:faCoRlCovWidth} from the main paper.
\end{compactitem}

Overall, the findings are qualitatively comparable to those presented in the Section \refstar{sec:sim} from the main paper. The most notable differences are as follows: the precision of all methods (in particular of the raw basic bootstrap confidence intervals) is much higher, and in particular, the bias is negligible (the squared bias contributes at most 0.2\% of the MSE). This may be explained by the fact that the mean is possibly the most simple target parameter; in particular, no linear expansions are required for studying their asymptotics mathematically. This is also expressed by the fact that much smaller enlargement factors are needed for correcting the raw confidence intervals' coverage: the estimated regression curve is $c(m, \gamma)= 1.222 -0.001 m +0.251 \gamma $. 

\begin{figure}[t!] 
\centering
\makebox{\includegraphics[width=0.95\textwidth]{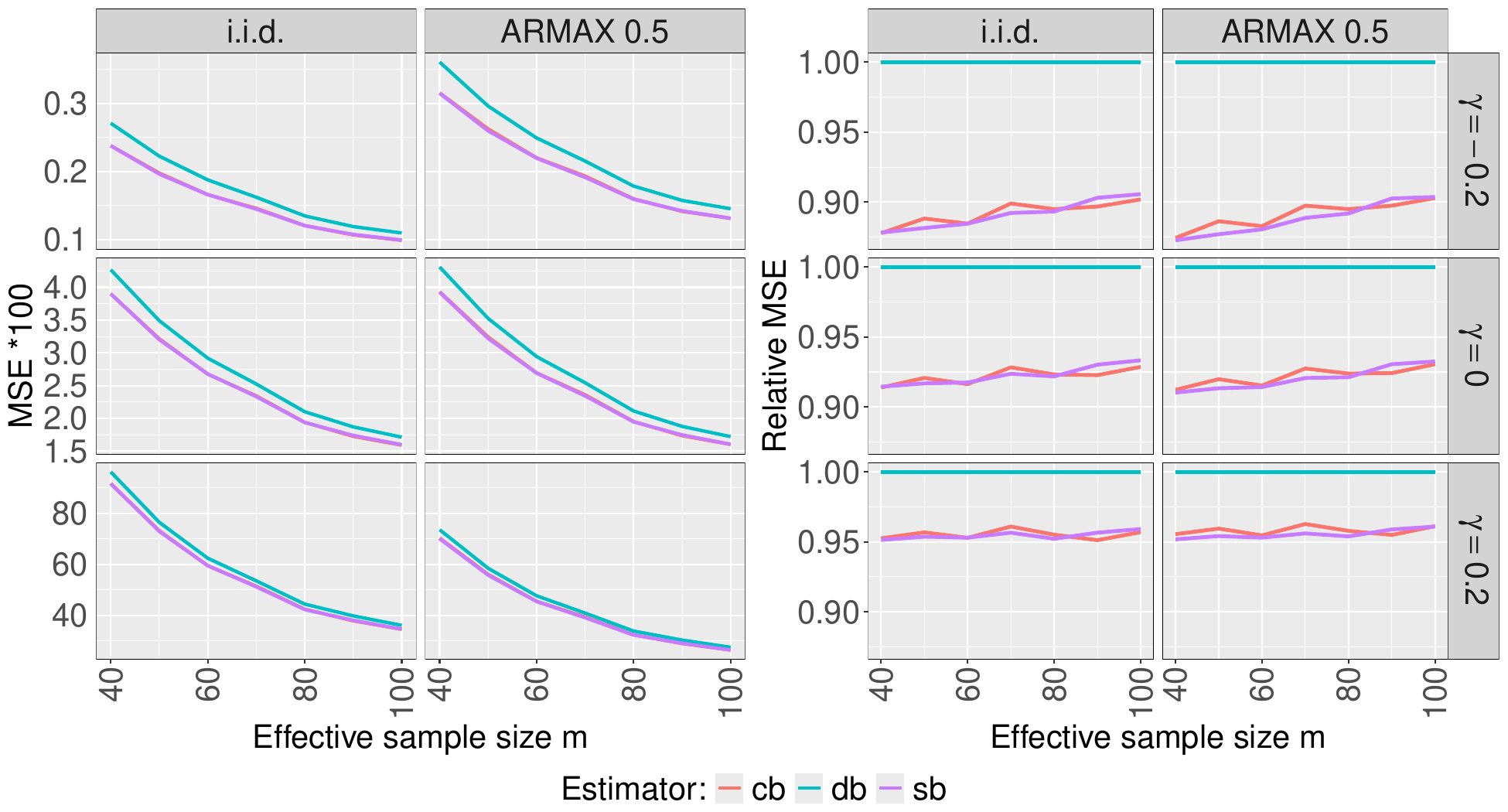}}\vspace{-.1cm}
\caption{Mean estimation with fixed block size $r=365$. Left: Mean squared error $\mathrm{MSE}(\hat \mu_n^{[\mbl]})$. Right: relative MSE with respect to the disjoint blocks method, i.e., $\mathrm{MSE}(\hat \mu_n^{[\mbl]}) / \mathrm{MSE}(\hat\mu_n^{[\dbl]})$.
}\label{fig:meanRelMseBias}	
\vspace{-.1cm}
\end{figure}

\begin{figure}[t!] 
\centering
\makebox{\includegraphics[width=.95\textwidth]{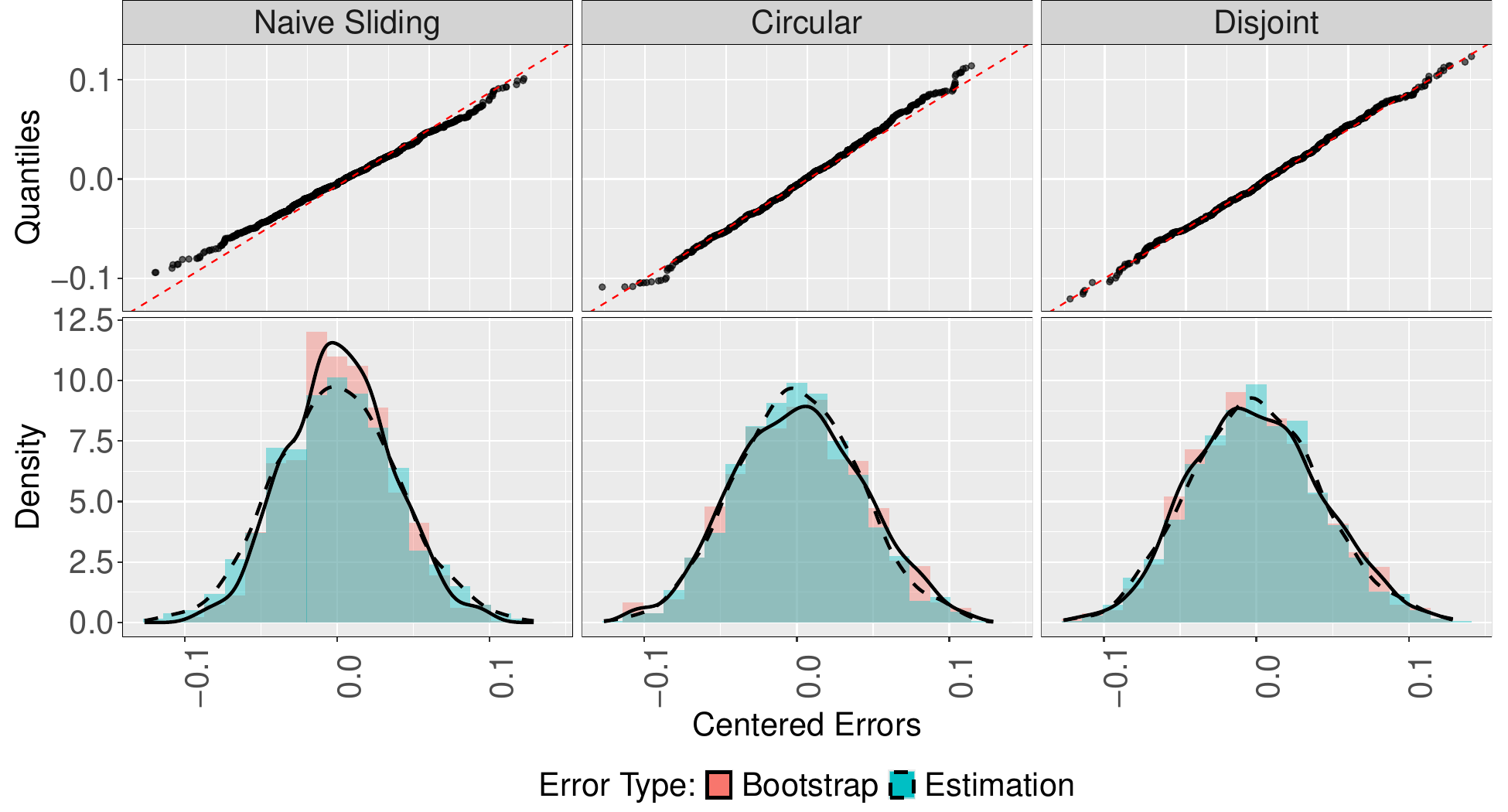}}\vspace{-.1cm}
\caption{Estimation error $\hat \mu_{n}^{[\mbl]}- \mu_r$ vs.\ bootstrap error $\hat \mu_{n}^{(\mbl'),*}- \hat \mu_{n}^{(\mbl')}$ for mean estimators with fixed block size $r=365$. Top row: qq-plots. Bottom row: histograms and kernel density estimates. Underlying data from Model~\ref{mod:armax-gpd} with $\gamma = -0.2, \beta=0.5$ and $m = 80$.
}\label{fig:meanHistQQ}	
\vspace{-.1cm}
\end{figure}

\begin{figure}[t!] 
\centering
\makebox{\includegraphics[width=0.95\textwidth]{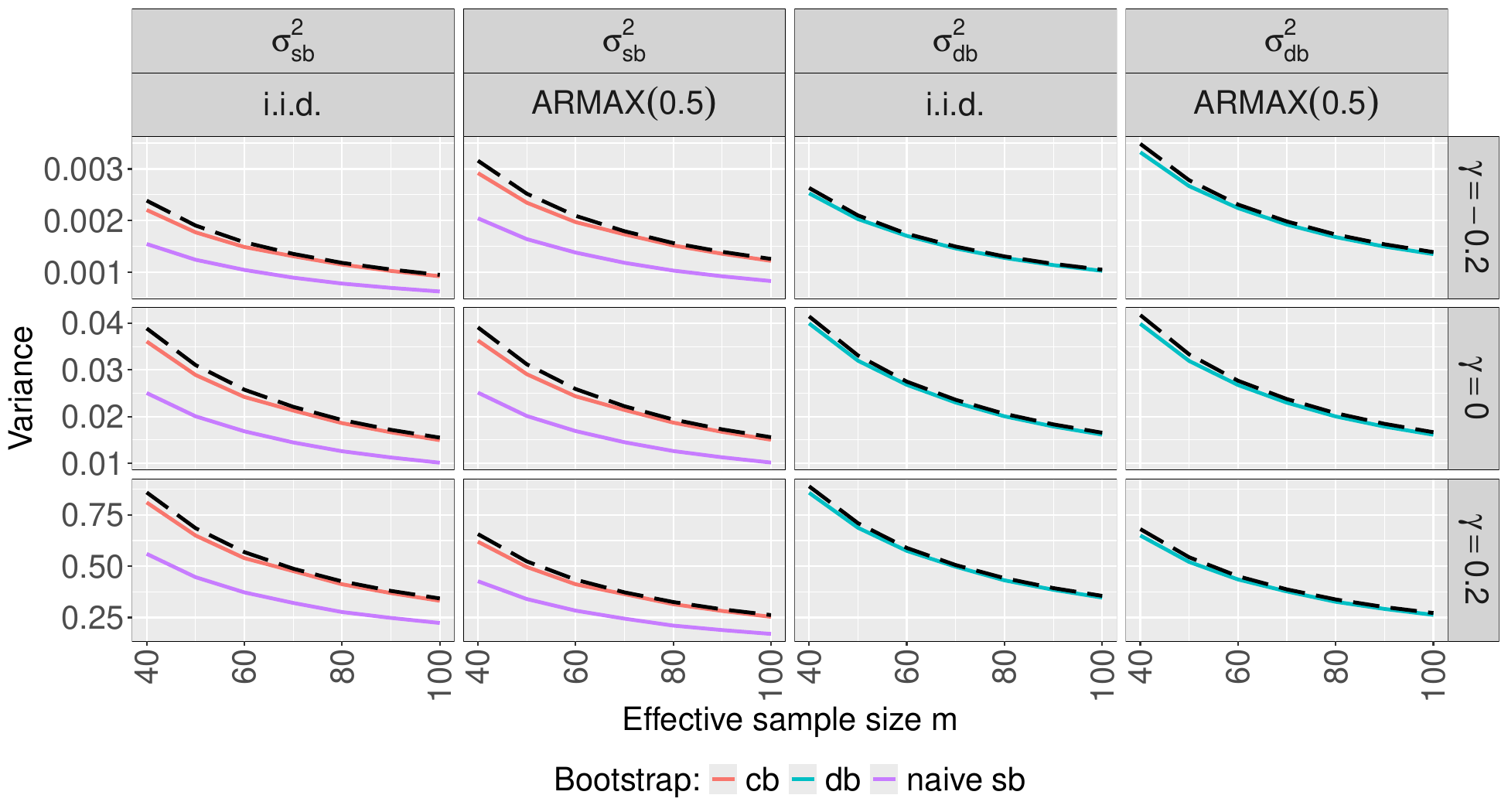}}\vspace{-.1cm}
\caption{Bootstrap-based estimation of the mean estimation variance $\sigma_\mbl^2 = \Var(\hat \mu_n^{[\mbl]})$ with fixed block size $r=365$. Left two columns: target parameter $\sigma_\sbl^2$ (dashed line), with the two colored lines representing the empirical variance of the naive sliding and the circular bootstrap sample, respectively, averaged other $N=1,000$ simulation runs. Right two columns: the same with target parameter $\sigma_\dbl^2$ (dashed line) and the colored line the empirical variance of the disjoint bootstrap sample.
}\label{fig:meanBstVar}	
\vspace{-.1cm}
\end{figure}

\begin{figure}[t!] 
\centering
\makebox{\includegraphics[width=0.95\textwidth]{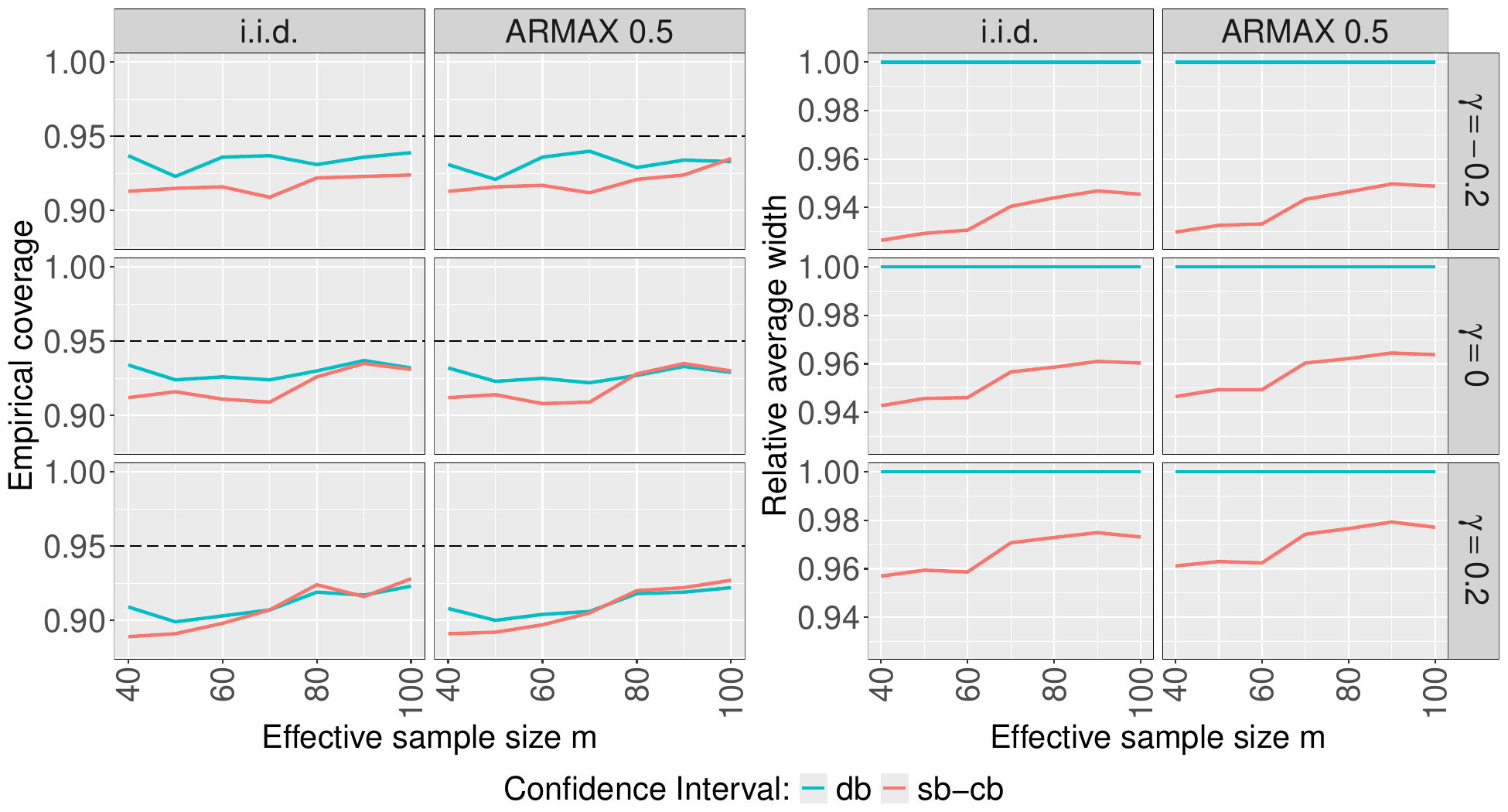}}\vspace{-.1cm}
\caption{Basic bootstrap confidence intervals for $\mu_r$ with fixed block size $r=365$. 
Left: empirical coverage with intended coverage of $95\%$ (dashed line). 
Right: relative average width with respect to the disjoint method, i.e., $\mathrm{width(CI}[\dbl]) / \mathrm{width(CI}[\mbl])$. 
}\label{fig:meanBstCiVar}	
\vspace{-.1cm}
\end{figure}
\begin{figure}[t!] 
\centering
\makebox{\includegraphics[width=0.95\textwidth]{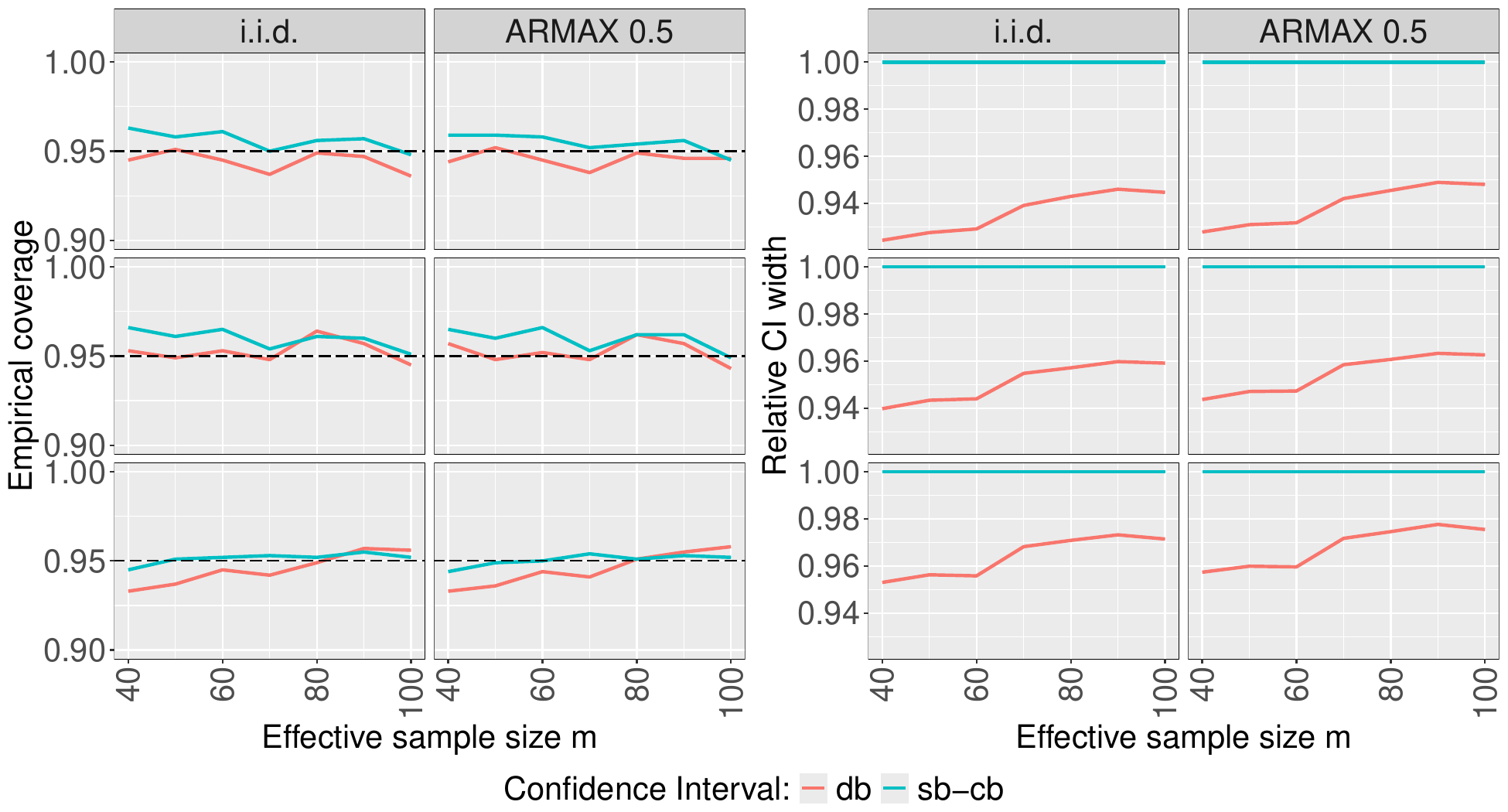}}\vspace{-.1cm}
\caption{Performance of the data-adaptive size-corrected confidence intervals for the mean. 
Left: empirical coverage of the data-adaptive size-corrected confidence intervals $\CI^{[\dbl]}(\hat c)$ and $\CI^{[\sblcbl]}(\hat c)$. Right: Respective relative width $\mathrm{width}(\CI^{[\sblcbl]}(\hat c))  / \mathrm{width}(\CI^{[\dbl]}(\hat c))$
.
}\label{fig:meanBstFacoCi}	
\vspace{-.1cm}
\end{figure}

\subsection{Fixed block size: estimation of Spearman's rho}
\label{sec:sim-fixr-spearman}

In this section, we consider a bivariate target parameter, namely Spearman's $\rho$ of $\bm M_r=(M_{r1}, M_{r2})$, throughout denoted as $\rho_r$. As in Sections~\refstar{sec:sim} and \ref{sec:sim-fixr-mean}, we consider a fixed block size of $r=365$. Seven data generating processes have been considered, each of which is based on an IID sequence of bivariate random vectors with Cauchy(1)-distributed margins, and with copula $C$ being either the independence copula, the  Gaussian copula, the $t_\nu$-copula with $\nu=4$ degrees of freedom, and the Gumbel-Hougaard copula. The parameter of three last-named copulas was chosen in such a way that the associated value of Spearman's rho, denoted $\rho_{\bm X}$,  is in $\{0.2, 0.4\}$. Note that the Gaussian copula is tail independent, while the $t$- and Gumbel copulas exhibit upper tail dependence.

The results are summarized in Figure~\ref{fig:spearman-estimator} (estimator performance) and in Figure~\ref{fig:spearman} (raw basic bootstrap confidence intervals), where we restrict attention to $\rho_{\bm X}=0.4$, as $\rho_{\bm X}=0.2$ provides qualitatively the same results. The findings are similar to the ones for return level and mean estimation: the sliding and circular blocks estimators behave very similar, and they outperform the disjoint blocks estimator. The contribution of the squared bias to the MSE was found to be negligible (less than 1\% for all scenarios). The disjoint and sliding-circular (raw) basic bootstrap confidence intervals exhibit a similar coverage (which is mostly below the intended coverage), but with much smaller intervals for the latter approach. This transfers to the the  size-corrected intervals, which are not displayed for the sake of brevity.

Finally, it should be noted that similar results were obtained when the above-mentioned copula models where used for innovations in a time series model.

\begin{figure}[t!] 
\centering
\makebox{\includegraphics[width=0.95\textwidth]{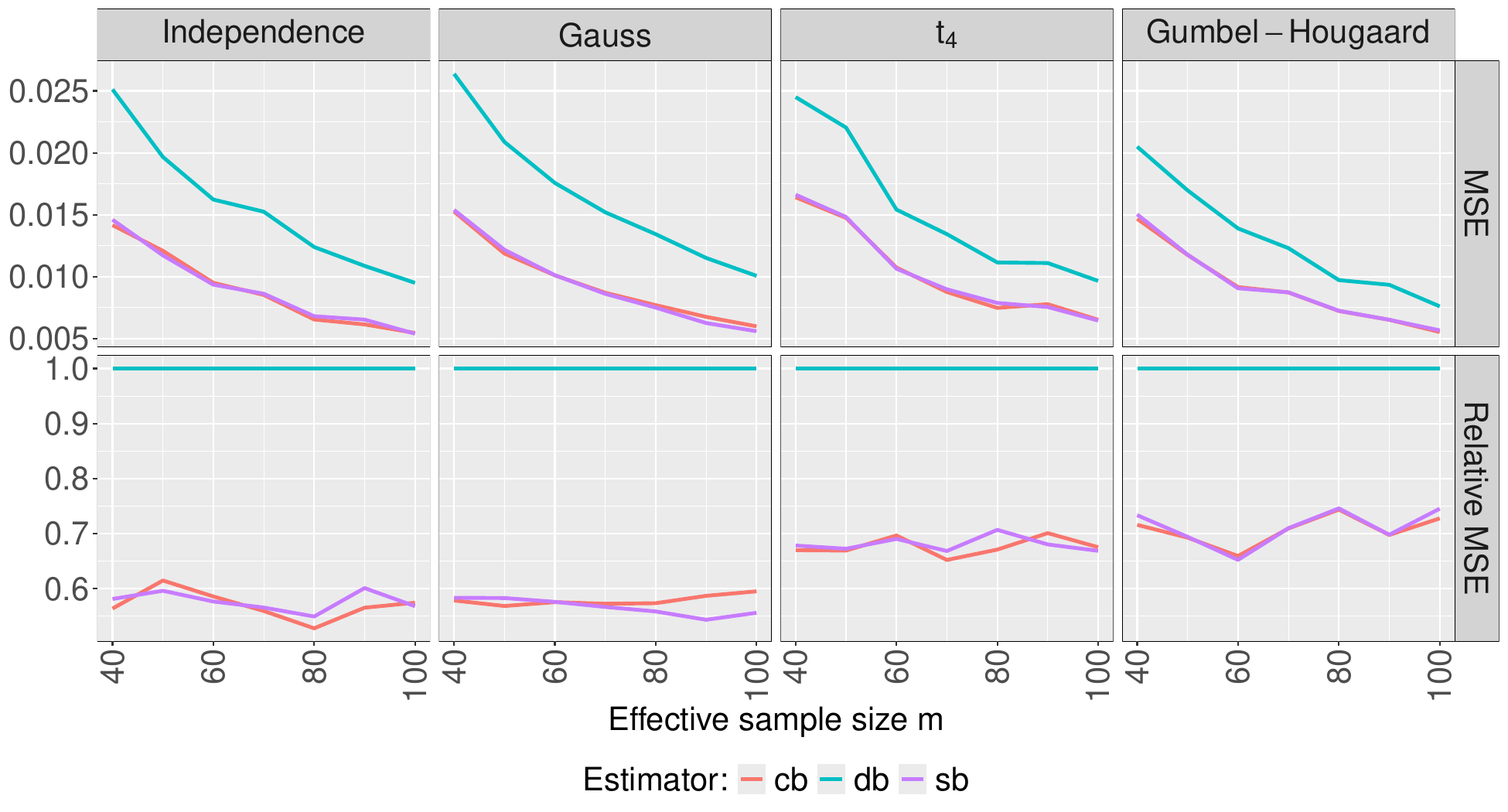}}\vspace{-.1cm}
\caption{Estimation of Spearman's $\rho=\rho_r$ with fixed block size $r = 365$. Top row: mean squared error $\mathrm{MSE}(\hat \rho_r^{[\mbl]})$. Bottom row: relative MSE with respect to the disjoint blocks method, i.e., $\mathrm{MSE}(\hat \rho_r^{[\mbl]}) / \mathrm{MSE}(\hat \rho_r^{[\dbl]})$.
}\label{fig:spearman-estimator}	
\vspace{-.1cm}
\end{figure}

\begin{figure}
\centering
\makebox{\includegraphics[width=0.95\textwidth]{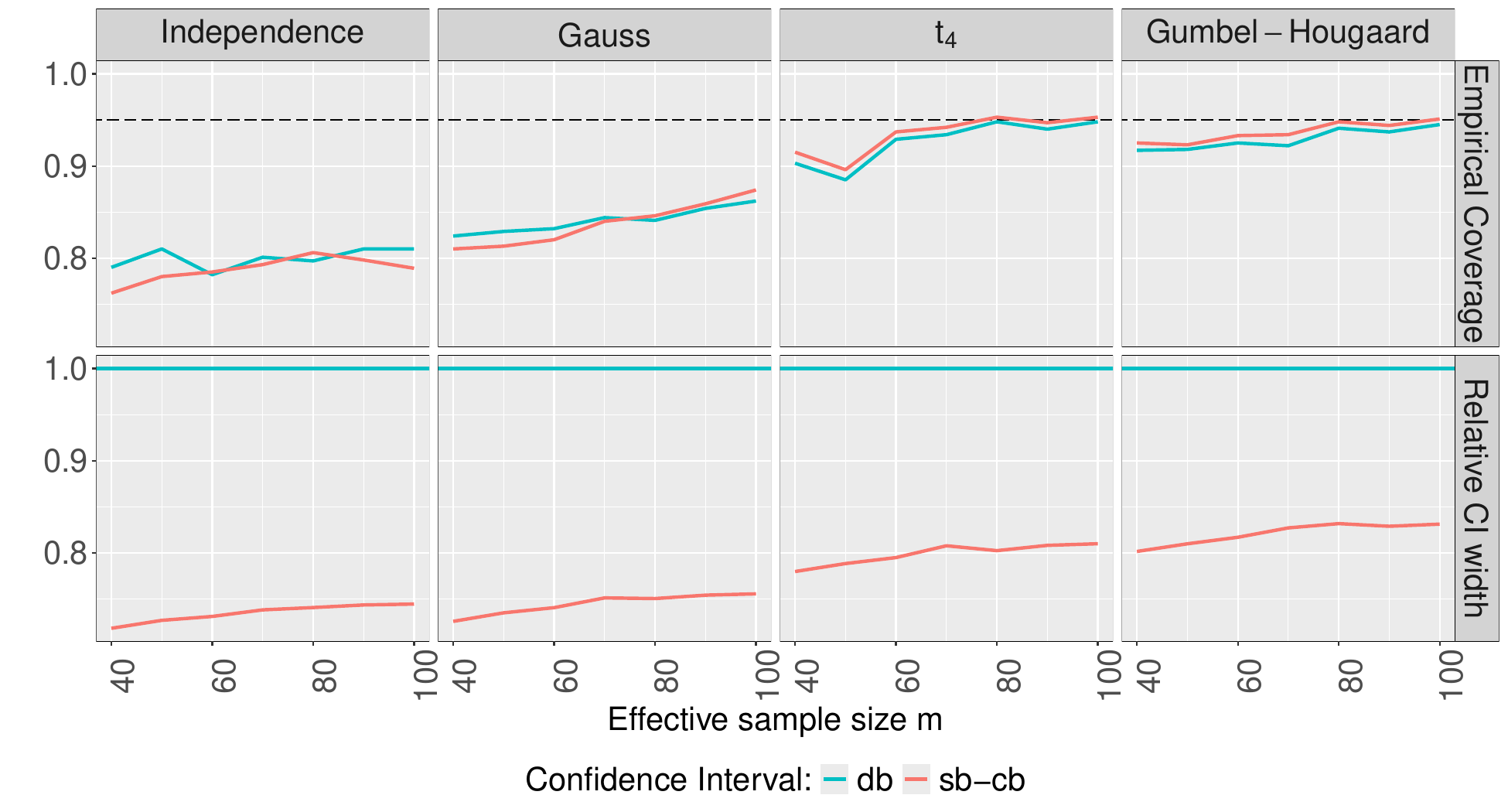}}\vspace{-.1cm}
\caption{Basic bootstrap confidence intervals (no size correction) for Spearman's $\rho=\rho_r$ with fixed block size $r = 365$. Top row: Empirical coverage. Bottom row: relative confidence interval widths, i.e., $\mathrm{width}(\CI^{[\sblcbl]}) / \mathrm{width}(\CI^{[\dbl]})$.
}
\label{fig:spearman}	
\vspace{-.1cm}
\end{figure}

\subsection{Fixed total sample size: shape estimation}
\label{sec:sim-fixn}

In this section, we consider the estimation of target parameters associated with the limiting law of $\bm M_r$ for $r\to\infty$, as well as bootstrap approximations for the respective estimation errors. For illustrative purposes, we restrict attention to the univariate, heavy tailed case and the estimation of the shape parameter $\alpha$ by pseudo-maximum likelihood estimation, as discussed in Section~\refstar{sec:frechet-small} from the main paper. 

\begin{model}[ARMAX-Pareto-Model]
\label{mod:armax-pareto}
$(X_t)_{t\in\Z}$ is a stationary, positive time series whose stationary distribution is the Pareto($\alpha$) distribution for some $\alpha>0$, defined by its CDF $F_\alpha(x) = 1-x^{-\alpha}$ for $x>1$. After transformation to the $\text{Fréchet}(1)$-scale, the temporal dynamics correspond to the ARMAX(1)-model with parameter $\beta\in [0,1)$; see Model~\ref{mod:armax-gpd} for details. Throughout, we consider the six models defined by the combinations of $\alpha\in \{0.5, 1, 1.5\}$ and $\beta\in \{0.0.5\}$.
\end{model}

For the simulation experiment, we fix the total sample size to $n=1,000$, and treat the block size $r$ as a variable tuning parameter akin to the choice of $k$ in the peak-over-threshold method. More precisely, 
we consider choices of $r$ ranging from 8 to 40, resulting in effective sample sizes $m=n/r$ ranging from 125 to 25.

\smallskip
\noindent \textbf{Performance of the estimators.}
In Figure \ref{fig:FreNFixMse} we depict the MSE of $\hat \alpha_n^{(\mbl)}$ from Section~\refstar{sec:frechet-small} with $\mbl \in \{\dbl, \sbl, \cblk\}$ as a function of the effective sample size $m=n/r$ (where we choose $k=2$ for the circular method; other choices lead to similar results). For $\beta=0.5$, we observe the typical bias-variance tradeoff resulting in a u-shaped MSE-curve: the larger the effective sample size, the  smaller the variance and the larger the bias (for $\beta=0$, the MSE is dominated by the variance even for block sizes as small as 8).
Apart from that, the findings are similar to the previous results for the case where $r$ was fixed: the sliding and circular block maxima estimator perform remarkably similar (in particular, the additional bias in the circular maxima is irrelevant), and it uniformly outperforms the  disjoint blocks estimator across all models under consideration.

\begin{figure}[t!] 
\centering
\makebox{\includegraphics[width=0.9\textwidth]{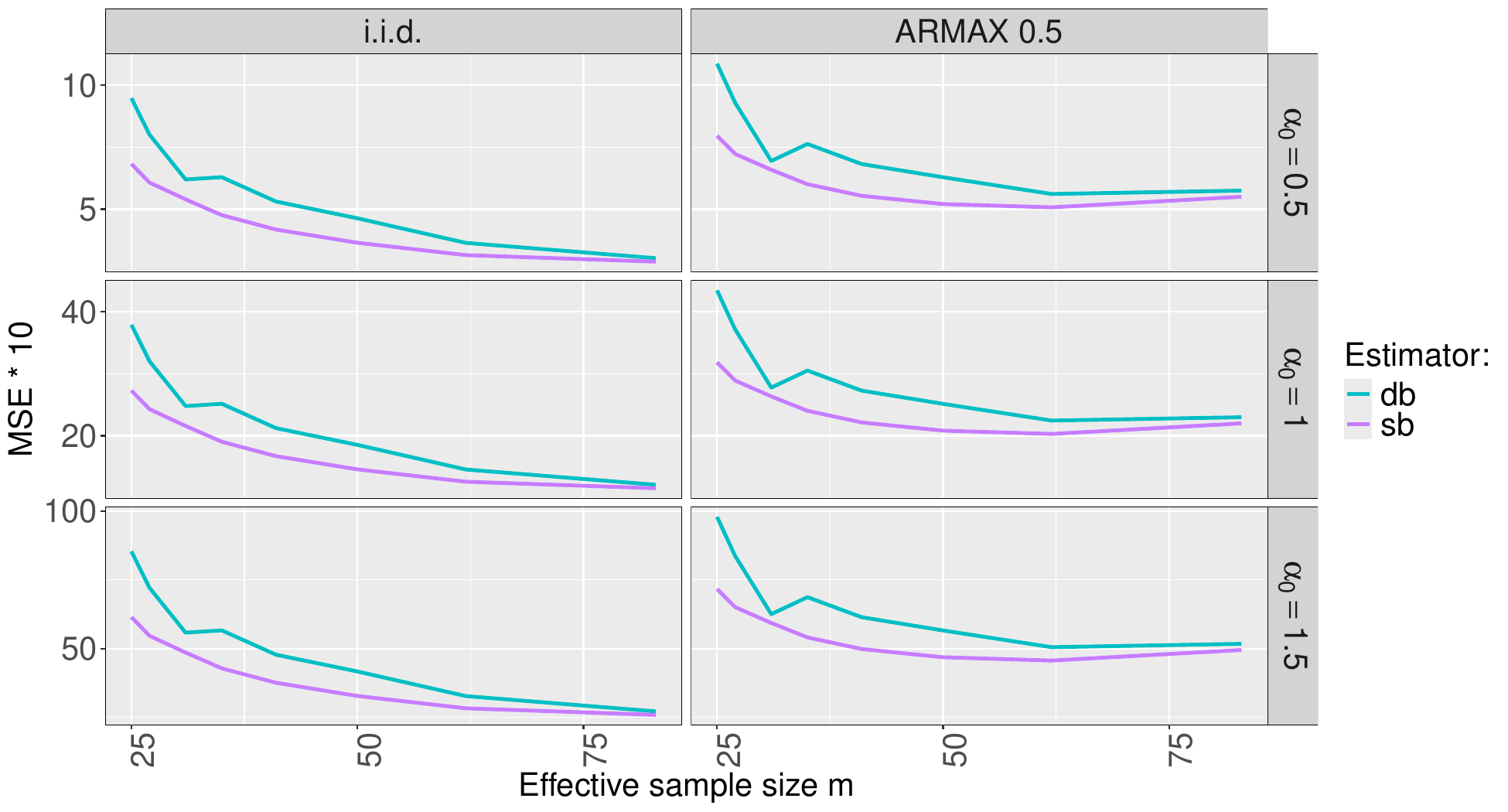}}\vspace{-.1cm}
\caption{
Mean squared error for the estimation of $\alpha$ with fixed sample size $n=1,000$.
}\label{fig:FreNFixMse}	
\vspace{-.1cm}
\end{figure}

\smallskip
\noindent \textbf{Performance of the bootstrap.} 
Next, we consider each of the bootstrap estimators $\hat \alpha_n^{(\mbl), \ast}$ with $\mbl \in \{\dbl, \sbl, \cblk \}$,  with number of bootstrap replications set to $B=1,000$.
Similar as in Section~\refstar{sec:sim}, we start by evaluating the performance of the three bootstrap approaches in terms of their ability to provide accurate estimates of the estimation variance $\sigma_\mbl^2(r) = \Var(\hat \alpha_n^{(\mbl)})$ with $\mbl \in \{\dbl, \sbl\}$.
In Figure~\ref{fig:plotFreFixNVarEst}, we depict the average over the $N=1,000$ bootstrap estimates, along with the true parameters determined in a presimulation based on $10^6$ repetitions. 
The findings are akin to those in Section~\refstar{sec:sim} and \ref{sec:sim-fixr-mean}, the only additional remarkable observation being that the disjoint blocks method tends to overestimate the variance parameter in the IID case, in particular for large block sizes.

\begin{figure}[t!] 
\centering
\makebox{\includegraphics[width=0.95\textwidth]{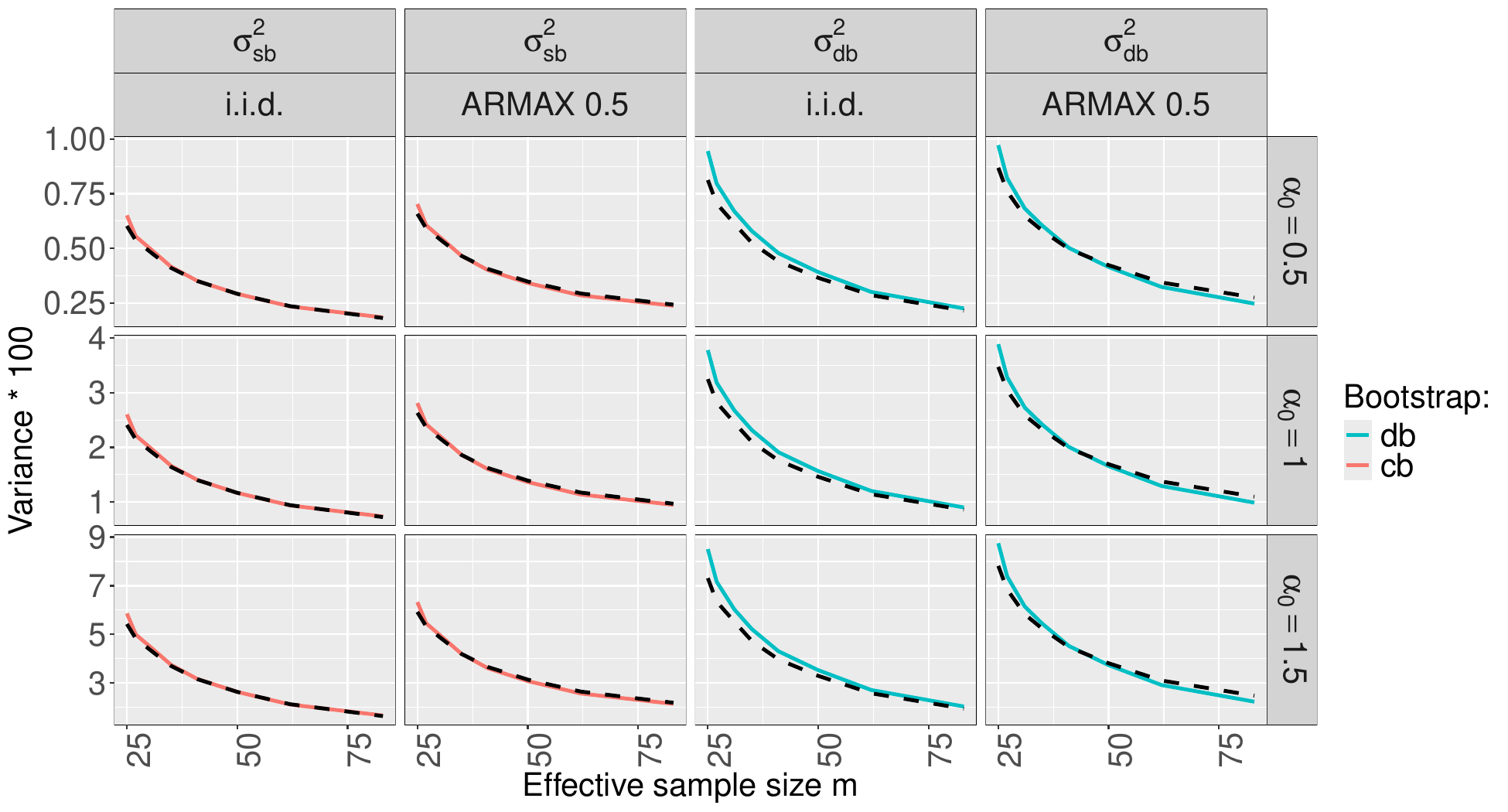}}\vspace{-.1cm}
\caption{Bootstrap-based estimation of the estimation variance $\sigma_\mbl^2(r) = \Var(\hat \alpha_n^{(\mbl)})$ with fixed sample size $n = 1,000$. Left two columns: sliding ($\mbl=\sbl$); the dashed line corresponds to $\sigma_\sbl^2(r)$. Right two columns: disjoint ($\mbl=\dbl$); the dashed line corresponds to $\sigma_\dbl^2(r)$.
}
\label{fig:plotFreFixNVarEst}	
\vspace{-.1cm}
\end{figure}

Finally, we evaluate the performance of the bootstrap approaches in terms of their ability to provide accurate (basic bootstrap) confidence intervals of pre-specified coverage.
The empirical coverage and the average widths of the respective intervals are depicted in Figure~\ref{fig:plotFreFixNCi}, where we omit the naive sliding method because of its inconsistency.  
We find that in most scenarios the desired coverage is almost reached by any method (observed coverage $\ge 92\%)$, as long as the block size is not too small. Overall, the disjoint blocks method has the the best coverage by a small margin, often reaching the intended level exactly (possibly because of the overestimation of the variance parameter observed in Figure~\ref{fig:FreNFixMse}). The observed qualitative behavior may further be explained by the fact that the target parameter is an asymptotic parameter, such that both the sliding and the disjoint blocks method provide bias estimates (with the same asymptotic bias). Smaller confidence intervals for the sliding-circular method hence necessarily imply a smaller coverage because they are concentrated around a biased estimate.

\begin{figure}[t!] 
\centering
\makebox{\includegraphics[width=0.95\textwidth]{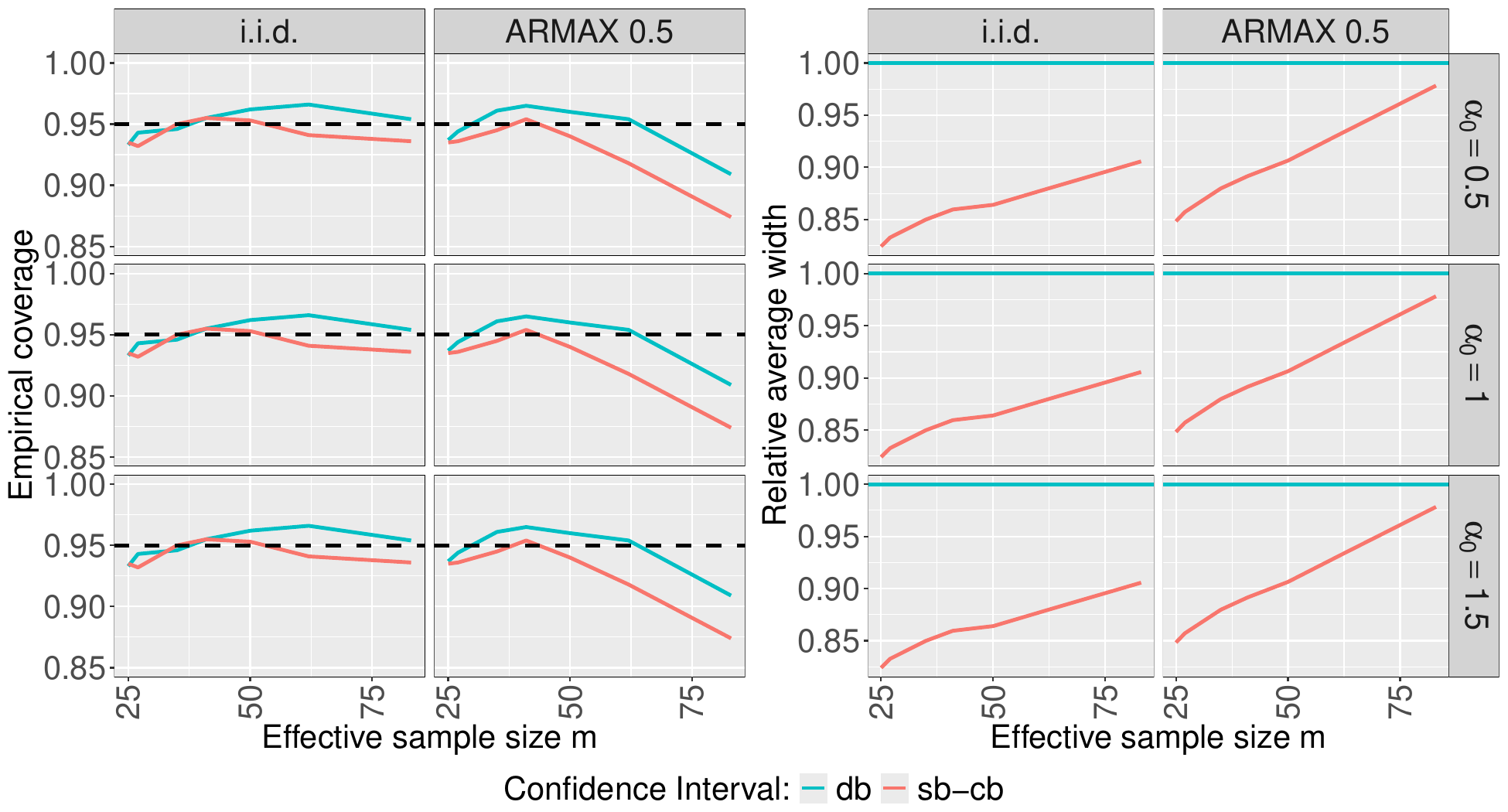}}\vspace{-.1cm}
\caption{Empirical coverage and average width of basic bootstrap CIs for shape estimation ($\alpha$), fixed $n=1{,}000$. 
Left: empirical coverage with intended coverage of $95\%$ (dashed line). 
Right: relative average width with respect to the disjoint method, i.e., $\mathrm{width(CI}(\dbl)) / \mathrm{width(CI}(\mbl))$.  
}\label{fig:plotFreFixNCi}	
\vspace{-.1cm}
\end{figure}

%%%%%%%%%%%%%%%%%%%%%%%%%%%%%%%%%%%%%%%%%%%%%%%%%%%%%%%%%%%%%%%%%%%%%%%%%%%%%%%%%%
\subsection{Runtime comparison}
\label{sec:runtime}

In classical situations, the computational cost of the bootstrap depends linearly on the number of bootstrap replications and is therefore high if a single evaluation of a statistic of interest is computationally intensive. 
Since both the sliding and the circular block maxima samples are much larger than the plain disjoint block maxima sample (sample sizes $n$ vs.\ $n/r$, respectively), one may naively think that the former methods require substantially more computational resources. As we will argue below and prove with simulations, this naive heuristic is not correct.

For any block maxima method, the starting point for a single evaluation of a statistic of interest is the calculation of the respective block maxima samples, which requires evaluating $O(n/r)$ maxima for the disjoint block maxima method, or $O(n)$ maxima for the sliding and circular block maxima method. 
Subsequently, when the statistic of interest is applied to one of the samples, the fact that both the sliding and the circmax samples can be efficiently stored as a weighted sample of size $O(n/r)$ implies that the additional computational cost of the latter two methods is, approximately, only a constant multiple of the additional cost for the plain disjoint block maxima method.
Next, for the bootstrap approaches proposed within this paper, no additional evaluation of maxima is ever required, whence, overall, the only major difference between the three approaches is the initial calculation of the block maxima samples. Therefore, the relative computational effort should only depend on $n$, and only moderately. 

The above heuristic has been confirmed by Monte Carlo simulations. Exemplary results are presented in~Figure \ref{fig:timeComp}, which rely on simulated data from Model~\refstar{mod:armax-pareto} (parameters: $\beta=0.5$ and $\alpha=1.5$) with fixed block size $r=90$ and total sample size ranging form $40 \cdot 90=3,600$ up to $100 \cdot 90=9,000$. The target parameter is the runtime for calculating $B \in \{250, 500, 750, 1,000\}$ bootstrap replicates of $\hat {\bm \theta}_n^{[\mbl]}$ from \eqref{eq:mle-frechet} for $\mbl \in \{ \dbl, \sbl, \cblk\}$, assessed by taking the median over $N=500$ repetitions each. 
The disjoint blocks method has been considered as a benchmark, whence we depict relative runtimes with respect to that method. We find that, as expected, the relative runtime is mostly depending on $n$, with only a moderate loss in performance for the circmax-method.
Similar results were obtained for other estimators and models.

\begin{figure}[ht!] 
\centering
\makebox{\includegraphics[width=0.9\textwidth]{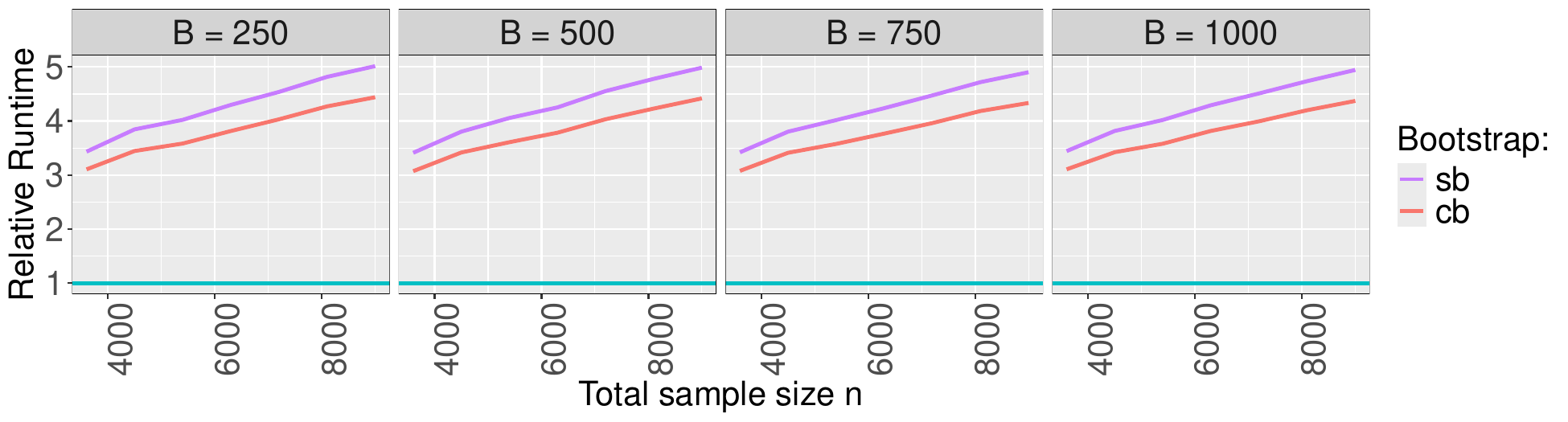}}\vspace{-.1cm}
\caption{Relative median runtimes of different bootstrap algorithms for bootstrapping $\hat {\bm \theta}_{n}^{[\mbl]}$ (relative to the runtime of the disjoint blocks bootstrap) for fixed blocksize $r = 90$ as a function of the effective sample size and for different numbers of bootstrap replicates~$B.$
}\label{fig:timeComp}	
\vspace{-.3cm}
\end{figure}

\putbib
\end{bibunit}

\end{document}